\documentclass[11pt, xcolor = dvipsnames,reqno
]{amsart}
\usepackage[dvipsnames]{xcolor}
\usepackage{tikz} 
\usepackage{pgfplots}
\usepackage{amsmath, amsthm, amssymb,mathrsfs,amsfonts}
\usepackage{mathtools,bbm}
\usepackage{relsize}
\usepackage{enumitem}
\usepackage[utf8]{inputenc}
\usepackage{pgffor}
\usepackage{color}
\usepackage{epstopdf}
\usepackage{graphicx}

\usepackage{varwidth}
\usepackage{verbatim}

\usepackage[foot]{amsaddr}

\usepackage{mathtools}      
\mathtoolsset{showonlyrefs} 

\pgfplotsset{compat=1.17}

\usepackage{fullpage}   

\usepackage[
colorlinks=true, 
linkcolor=blue, 
citecolor=magenta]{hyperref} 

\newtheorem{thm}{Theorem}[section]
\newtheorem{coro}[thm]{Corollary}
\newtheorem{prop}[thm]{Proposition}
\newtheorem{lem}[thm]{Lemma}

\newtheorem*{acknowledgements*}{Acknowledgements}

\theoremstyle{definition}

\theoremstyle{remark}
\newtheorem{rmk}[thm]{Remark}

\definecolor{ao}{rgb}{0, 0.5, 0}

\newcommand{\dist} {\text{dist\! }}

\usepackage[sortcites,
backend=biber,
date=year,
style=alphabetic,
giveninits=true,
doi=false,
isbn=false,
url=true,
eprint=true,
maxnames=50,
maxalphanames=4]{biblatex} 

\DeclareNameAlias{author}{family-given}
\AtEveryBibitem{
  \clearlist{language}
  \clearfield{month}
  \clearfield{url}
  \clearfield{urlyear}
  \clearfield{urlmonth} 
  \ifentrytype{online}{
    \clearfield{year}}
  {
    \clearfield{eprint}}
}
\renewbibmacro*{volume+number+eid}{%
  \printfield{volume}%
  \setunit*{\addnbspace}
  \printfield{number}%
  \setunit{\addcomma\space}%
  \printfield{eid}
  }
\DeclareFieldFormat[article]{volume}{\textbf{#1}}
\DeclareFieldFormat[article]{number}{\mkbibparens{#1}}
\DeclareFieldFormat*{title}{\textit{#1}}
\DeclareFieldFormat{journaltitle}{#1\isdot}
\renewbibmacro{in:}{}
\begin{document}

\title[Almost sure pointwise convergence for quintic NLS]{
Almost sure pointwise convergence for the 2d periodic quintic NLS
}

\author{Daniel Eceizabarrena}
\address{BCAM - Basque Center for Applied Mathematics, Bilbao, Spain.
		  e-mail: {\tt deceizabarrena@bcamath.org}
		}
		
\author{Pablo Merino}
\address{BCAM - Basque Center for Applied Mathematics, Bilbao, Spain.
		  e-mail: {\tt pmerino@bcamath.org}
		}

\thanks{}
\date{\today}

\begin{abstract}
We prove probabilistic pointwise convergence to the initial datum for the 2D periodic quintic NLS for data in $H^s(\mathbb T^2)$ with $s > 0$. 
This is an improvement with respect to the deterministic setting, in which convergence is known to fail if $s < 1/3$. 
The proof is based on a nonlinear maximal characterization for convergence, 
Bourgain's linear-nonlinear decomposition 
and the corresponding nonlinear smoothing for which we employ random tensor estimates. 
\end{abstract}

\maketitle


\section{Introduction}
\noindent
Given the periodic NLS
\begin{equation}\label{eq:cauchy-nls}
		\left\{
		\begin{array}{rcl}
			i\partial_t u + \Delta u & = & |u|^{p-1} u, 
			\\
			u(0,x) & = & f(x),
		\end{array}
        \qquad \qquad 
        f \in H^s(\mathbb T^d),
		\right.
	\end{equation}
    Carleson's convergence problem consists in finding the minimal Sobolev regularity $s$ such that
    \begin{align}\label{eq:desired-limit}
        \lim_{t \to 0} u(t,x) = f(x)
        \quad \text { a.e. }, 
        \qquad 
        \text{ for all } f \in H^s(\mathbb T^d).
    \end{align}
    The question stems from the original problem for the Euclidean linear Schr\"odinger equation,  
    for which we know that the infimum $s$ such that
    \begin{equation}\label{eq:Carleson_Linear}
        \lim_{t \to 0} e^{it\Delta}f(x) = f(x)
        \quad \text { for a.e. } x \in \mathbb R^d, 
        \qquad \text{ holds for all } f \in H^s(\mathbb R^d)
    \end{equation}
    is $\frac{d}{2(d+1)}$, 
    positive results being proved in \cite{Carleson1980,DGL2017,DuZhang2019} and negative ones in \cite{DahlbergKenig1982,Bourgain2016}. 
    The endpoint is open except in $d=1$, in which case convergence holds \cite{Carleson1980}.
    
    In parallel, 
    the periodic problem is open in all dimensions. 
    While convergence like in \eqref{eq:Carleson_Linear} holds if $s > d/(d+2)$  \cite{MoyuaVega2008,WangZhang2019,CLS2020} and counterexamples are known if $s < d/(2(d+1))$ \cite{CLS2020,EceizabarrenaLuca2022}, 
    we do not know what happens if $s \in [\frac{d}{2(d+1)} , \frac{d}{d+2}]$. 
    Partial progress has been recently achieved in \cite{Barron2022,MYZ2023,Demeter2025}. 

    In \cite{CLS2020}, it was shown that the sufficient condition for the linear case is also sufficient for the nonlinear Schr\"odinger equation in \eqref{eq:cauchy-nls} as long as it is locally well-posed. 
    For example, in the periodic case, this means that \eqref{eq:desired-limit} holds if 
    \begin{equation}\label{eq:Carleson_NLS_Sufficient}
        s > \max \Big( \frac{d}{d+2}, \frac{d}{2} - \frac{2}{p-1} \Big),  
    \end{equation}
    and any improvement of the sufficient condition $d/(d+2)$ for the linear problem automatically translates to \eqref{eq:Carleson_NLS_Sufficient}.
    However, no improvement can be expected below the well-posedness threshold given by the scaling critical exponent $d/2 - 2/(p-1)$. 
    
    A probabilistic approach to the problem was also introduced in \cite{CLS2020}, where it was shown that that regularity can be decreased to any $s > 0$. 
    More precisely, for data of the type
    \begin{align}\label{eq:random-data}
        f^{\omega}(x) = \sum_{n \in \mathbb{Z}^d} \frac{g_n^{\omega}}{\langle n \rangle^{\frac{d}{2} + \alpha}} e^{inx}, 
        \qquad \qquad 
        f^\omega \in H^{\alpha -}(\mathbb T^d) := \bigcap_{s < \alpha} H^s \setminus H^\alpha \quad \omega\text{-almost surely},
    \end{align}
    with $(g_n^{\omega})_{n \in \mathbb{Z}^2}$ i.i.d. complex valued standard Gaussian random variables,
    they proved that 
	\begin{align}\label{eq:Linear_Convergence_Probabilistic}
		\text{ for every } \alpha > 0, 
        \qquad \lim_{t \to 0} \lVert  e^{i t \Delta}f^{\omega} - f^{\omega} \rVert_{L^\infty(\mathbb T^d)} 
        =0
        \qquad \omega\text{-almost surely, }
	\end{align}
    in such a way that uniform convergence holds almost surely at regularity $H^{\alpha-}$ for every $\alpha > 0$. 
    Regarding NLS in \eqref{eq:cauchy-nls}, they studied the 1D case and the 2D Wick-ordered cubic NLS, showing that the convergence property \eqref{eq:desired-limit} holds almost surely for data \eqref{eq:random-data} with $\alpha > 0$.  
    Analogous results were obtained for the cubic NLS on $\mathbb S^2$ \cite{MSSZZ2026}
    and for the cubic Klein-Gordon on $\mathbb T^3$ \cite{LucaMerino2026}.

    In this article, following the questions raised in \cite{CLS2020}, we extend these results to the 2D periodic quintic NLS by proving the convergence property \eqref{eq:desired-limit} for data
    \begin{align}\label{eq:random-data_2D}
        f^{\omega}(x) = \sum_{n \in \mathbb{Z}^2} \frac{g_n^{\omega}}{\langle n \rangle^{1 + \alpha}} e^{inx}, 
        \qquad 
        f^\omega 
        \in H^{\alpha -}(\mathbb T^2) 
        \quad \text{ almost surely,}
    \end{align}
    for $\alpha > 0$ fixed but arbitrarily small.
    Hence, choosing small enough $\alpha > 0$, we have data in $H^\varepsilon$ for arbitrality small $\varepsilon > 0$, 
    thus lowering the deterministic threshold $s>1/2$ in \eqref{eq:Carleson_NLS_Sufficient} to $s > 0$. 
    
    Our main result is the following.
    \begin{thm}\label{thm:pconv}
     Let $\alpha > 0$ and $f^{\omega}$ in \eqref{eq:random-data_2D}. 
     Let $u^{\omega}$ be the solution of the periodic 2D quintic NLS 
     \begin{equation}\label{eq:cauchy-nls_2D}
		\left\{
		\begin{array}{rcl}
			i\partial_t u + \Delta u & = & |u|^4 u, 
			\\
			u(0) & = & f^\omega,
		\end{array}
		\right.
        \qquad \text{ on } \mathbb T^2. 
	\end{equation}
    Then, $\omega$-almost surely we have
    \begin{align}\label{eq:cauchy-nls_2D_Convergence}
        \lim_{t \rightarrow 0} u^{\omega}(t,x) = f^{\omega}(x) \quad \text{ for almost every } x \in \mathbb{T}^2.
    \end{align} 
    \end{thm}
    \begin{rmk}
        Theorem~\ref{thm:pconv} holds regardless of the focusing or defocusing nature of the equation. 
    \end{rmk}

    \begin{rmk}
        Let $\alpha > 0$ and $\mu_{\alpha}$ be the centered Gaussian measure with Cameron-Martin space $H^{\alpha}(\mathbb T^2)$ and covariance operator $(\text{Id} - \Delta)^{-\alpha - 1}$, i.e. the one induced by the map $\omega \in \Omega \mapsto f^{\omega}$, with $f^{\omega}$ as in \eqref{eq:random-data_2D}. Then, Theorem \ref{thm:pconv} is equivalent to the following property: given $s \in (0 , \alpha)$, for $\mu_{\alpha}$-almost every $f \in H^s(\mathbb T^2)$, the solution $u$ of \eqref{eq:cauchy-nls_2D} with datum $f$ satisfies
        \begin{align*}
            \lim_{t \rightarrow 0} u(t,x) = f(x) \qquad \text{ for almost every } x \in \mathbb{T}^2.
        \end{align*}
    \end{rmk}

    \subsection*{General strategy}
    The proof of Theorem \ref{thm:pconv} is based on a combination of the maximal estimate characterization of pointwise convergence and a linear-nonlinear decomposition \textit{\`a la Bourgain} \cite{Bourgain1996}.  
    It will be critical to optimize the nonlinear smoothing in this decomposition.

    As usual, the convergence result follows from a maximal estimate characterization. 
    We recall that in the case of the linear equation, convergence to the datum in \eqref{eq:Carleson_Linear} follows from
    \begin{align}\label{eq:max-est}
        \Big\| \sup_{0 < t < \delta} \big| e^{it\Delta} f \big| \Big\|_{L^2(\mathbb{T}^2)} \lesssim \|f\|_{H^s(\mathbb{T}^2)},
        \qquad \forall f \in H^s(\mathbb T^2), 
    \end{align}
    for some $\delta > 0$. 
    In a similar way, as shown in \cite{CLS2020}, 
    the nonlinear counterpart \eqref{eq:desired-limit} follows from the stability of the frequency-truncated flow in the maximal norm, that is, from 
    \begin{align}\label{eq:Stability_Intro}
        \lim_{N \rightarrow \infty} 
			\Big\| \sup_{0 < t < \delta} 
			| \Phi_t f   - \Phi^N_t f  | \Big\|_{L^2(\mathbb{T}^2)} = 0,
    \end{align}
    which we revisit in Lemma \ref{lemma:AdaptMaxEst_NL}.
    To prove this estimate, we will combine probabilistic arguments for the linear flow with the embedding $X^{s,b}_\delta \subset L^2_x(L^\infty_t(0,\delta))$ that follows from the fact that the maximal estimate \eqref{eq:max-est} holds for $s > d/(d+2) = 1/2$.
    The proof then boils down to obtaining enough nonlinear smoothing for the solution. 
    In other words, for $f^\omega \in H^{\alpha -}$ as in \eqref{eq:random-data_2D}, we look for a linear-nonlinear decomposition in such a way that 
    \begin{align}\label{eq:nlsmoothing-gralform1}
        u(t,x) - e^{it\Delta} f^{\omega} \in C([0,\delta_{\omega}] , H^{\alpha + \sigma}(\mathbb{T}^2)),
        \qquad \omega\text{-almost surely,}
    \end{align}
    for some $\delta_\omega > 0$. It is critical that the smoothing $\sigma > 0$ is such that $\alpha + \sigma > 1/2$ for two reasons. 
    First, it is necessary for the embedding above to hold. 
    Second, the solution \eqref{eq:nlsmoothing-gralform1} is found with a fixed-point argument that can only be performed in a deterministically subcritical regularity $\alpha + \sigma > s_c = d/2 - 2/(p-1) =  1/2$. 
    Since $\alpha > 0$ is arbitrary, we will prove a smoothing of $\sigma < 1/2$ in Theorem~\ref{thm:to-pconv} 
    for which, as in \cite{OhWang2025}, we use the random tensor estimates developed in \cite{DNY2022} (see also \cite{Kaneshiro2025}). 
    The fact that this smoothing implies convergence is proved in Section~\ref{sec:NonlinearSmoothing_and_ProofOfTheorem}.

    \begin{rmk}
        It is fundamental in this argument that the scaling critical regularity $s_c = 1/2$ for the 2D quintic NLS matches the best regularity currently available for the maximal estimate $d/(d+2) = 1/2$. 
        In general, in $\mathbb T^d$ and any nonlinearity $|u|^{p-1} u$ with odd $p$, 
        this argument will work anytime we can prove enough smoothing $\sigma > 0$ such that 
        \begin{equation}
            \alpha + \sigma > \max \Big(  \frac{d}{d+2}, \frac{d}{2} - \frac{2}{p-1} \Big). 
        \end{equation}
        Following \cite{DNY2024}, in $\mathbb T^2$ one cannot expect more nonlinear smoothing than  $\sigma < 1/2$, and therefore, the quintic NLS is the only case, apart from the cubic, for which this method allows to reach the whole range $\alpha > 0$. 
        In particular, if $p > 5$ this argument will at most provide a result for 
        \begin{equation}
            \alpha 
            > \Big(1 - \frac{2}{p-1} \Big) - \frac{1}{2}
            > 0, 
        \end{equation}
        and it is thus not strong enough to reach the whole range $\alpha > 0$, which we conjecture to hold for every $p$. 
        The same problem arises in the higher dimensional case on $\mathbb T^d$ with $d \geq 3$.         
    \end{rmk}

    \subsection*{Structure}
    In Section~\ref{sec:Notation} we introduce notation. 
    In Section \ref{sec:setup}, we set up the proof of Theorem~\ref{thm:pconv} for a gauged quintic NLS with a maximal estimate approach which we eventually reduce to a non-linear smoothing estimate in Theorem~\ref{thm:to-pconv}. 
    We also justify the equivalence of the result for the original and gauged equations. 
    In Section~\ref{sec:smoothing}, we use deterministic multilinear estimates and elementary bounds in Fourier space to further reduce the nonlinear smoothing to explicit bounds for the Fourier coefficients of the nonlinearity in Proposition~\ref{prop:Final_Estimate}. 
    We also include some basic notions of random tensor estimates. Finally, in Section~\ref{sec:proof-smoothing} we prove Proposition~\ref{prop:Final_Estimate} by splitting the nonlinearity in terms of pairings.
    We include two appendices with deterministic and probabilistic properties that we use throughout the article.

\section{Notation}
\label{sec:Notation}

\noindent
\textbf{General notation.}
    Given $A, B \in \mathbb{R}$, 
    by $A \lesssim B$ we mean that there is $c > 0$ such that  $A \leq c B$. 
    We write $A \simeq B$ when both $A \lesssim B$ and $A \gtrsim B$. 
    
    In general, we use $\varepsilon$ to denote probabilities close to zero, while arbitrarily small exponents will be denoted by $\varepsilon_1$. 

    For $u \in \mathbb C$, we write $u^+ = u$ and $u^- = \overline u$.
    To account for conjugates in the nonlinearity $|u|^{p-1} u$, we denote $\iota_j \in \{ +, -\}$ and write $u_j^{\iota_j}$ to denote either $u_j$ or $\overline{u_j}$. 
    In general, we will set 
    \begin{align}\label{eq:def-iotaj}
    \iota_j = \left\{ \begin{array}{cc}
        +,  & \text{ if } j \text{ is odd},  \\
        -, & \text{ if } j \text{ is even},   
    \end{array} \right. 
    \qquad \text{ and hence } \qquad 
    u_j^{\iota_j}
    = \left\{
    \begin{array}{ll}
    u_j,  &  \text{ if } j \text{ is odd},  \\
    \overline{u_j}, & \text{ if } j \text{ is even}.
    \end{array}
    \right.
\end{align}

\noindent
\textbf{Indices and dyadic frequencies.}
Given a family of vectors $\{n_1,\dots,n_p\} \subset (\mathbb{Z}^d)^p$ and $C \subset \{1,\dots,p\}$, we denote $n_C = \{n_j\}_{j \in C}$. 
For frequency vectors $n \in \mathbb Z^d$, we denote their dyadic size by $|n| \simeq N \in 2^{\mathbb N}$.

    If $|n_j| \simeq N_j$, we denote the decreasing rearrangement of $\{N_j\}$ by $\{N_{(j)}\}$, in such a way that $N_{(1)} \geq N_{(2)} \geq \ldots \geq N_{(p)}$, and $N_{(j)}$ is the $j$-th largest component. 
    In the same spirit, we denote $\operatorname{max}_C = \operatorname{max} \{ \,  N_j \, : \, j \in C \, \}$ and $\operatorname{min}_C = \operatorname{min} \{ \,  N_j \, : \, j \in C \, \}$. 
    If $|C| = 3$, we denote by $\operatorname{mid}_C$ the middle element.
    If $|C| = 4$, we denote by $\operatorname{mid}_{1,C}$ and $\operatorname{mid}_{2,C}$ the higher and lower middle elements, respectively. For instance, if $p=5$ and $C = \{1,2,4,5\}$, then 
    \begin{align}
        N_1 \geq N_4 \geq N_5 \geq N_2
        \qquad \Longrightarrow \qquad 
        \begin{array}{ll}
            \operatorname{max}_{\{1,2,4,5\}} = N_1,  & \operatorname{min}_{\{1,2,4,5\}} = N_2,  \\
            \operatorname{mid}_{1, \{1,2,4,5\}} = N_4,  & \operatorname{mid}_{2, \{1,2,4,5\}} = N_5. 
        \end{array}
    \end{align}

    Let $A,B$ be two finite sets of indices. Given     $(a_{n_A}) \in \ell_{n_A}^2$, we denote $\|a_{n_A}\|_{n_A} = \|a_{n_A}\|_{\ell^2_{n_A}}$, 
    and for a bounded operator $T : \ell_{n_A}^2 \rightarrow \ell_{n_B}^2$, we denote its operator norm by $\|T\|_{n_A \rightarrow n_B}$. 
        
 We denote the Fourier frequency projections by
\begin{align}
    \widehat{P_{\leq N} f} = \mathbbm 1_{|n| \leq N} \widehat f, 
    \qquad 
    P_N = P_{\leq N} - 	P_{\leq N/2}, 
    \qquad 
    P_{>N} = \operatorname{Id} - P_{\leq N}.
\end{align}

\noindent
\textbf{Restriction spaces.}
For $\delta > 0$ and $b > 1/2$, Bourgain's spaces $X^{s,b}_{\delta}$ are defined by the norm
	\begin{align}\label{eq:RestrSpaceDef}
        \|F\|_{X^{s,b}_{\delta}} = \inf_{G = F \text{ on } t \in [0,\delta]} \|G\|_{X^{s,b}},
    \qquad \qquad 
		\|G\|^2_{X^{s,b}} = 
        \sum_{n \in \mathbb{Z}^2} \int_{\mathbb{R}}   \langle n \rangle^{2s} \, \langle \tau + |n|^2 \rangle^{2b} \,  |\widetilde{G}(\tau,n)|^2 d\tau,
	\end{align}
	where $\widetilde{G}$ is the space-time Fourier transform of $G$. 
    We always work with 
    \begin{align}
        b = \frac12 + \beta, \qquad \text{ with small enough } \beta > 0. 
    \end{align}
    To account for the infimum in \eqref{eq:RestrSpaceDef}, 
    we work with a non-negative $\eta \in C_c^\infty$ such that $\eta = 1$ on $[0,1]$.
    We include a few properties of $X^{s,b}$ in Appendix~\ref{app:Deterministic}.

\section{Setup}\label{sec:setup}

\subsection{Renormalization of the nonlinearity}
\label{sec:Renomalized}
To prove Theorem~\ref{thm:pconv}, 
we work with the renormalized quintic NLS
\begin{equation}\label{eq:cauchy-nls-gauge}
		\left\{
		\begin{array}{rcl}
			i\partial_t u + \Delta u & = & \mathcal{Q}(u), 
			\\
			u(0,x) & = & f(x),
		\end{array}
		\right.
        \qquad \qquad 
        \mathcal{Q}(u) = |u|^4 u - 3 u \int_{\mathbb T^2} |u|^4 \, dx,
	\end{equation}
in such a way that the renormalized nonlinearity $\mathcal{Q}(u)$ does not have linear terms in $u$, which typically prevent nonlinear smoothing.     
We will show in Section~\ref{sec:renorm} that working with \eqref{eq:cauchy-nls-gauge} suffices. 

We denote by $\Phi^N_t$ the frequency-truncated flow, 
in such a way that $\Phi^N_t f$ solves
\begin{equation}\label{eq:nlsTruncated}
	\left\{
	\begin{array}{l}		 
		\partial_t u  + i \Delta u = P_{\leq N} \mathcal{Q} ( u ),	\\
		u(0) = P_{\leq N} f,
		\end{array}
		\right.
	\end{equation}
and	$\Phi_t f  := \Phi^{\infty}_t f$ denotes the non-truncated flow.

\subsection{Reduction to a maximal form}\label{sec:maximalest-approach}
Like in the linear problem, 
pointwise convergence for NLS can be reduced to estimates for maximal norms.
We reproduce the following lemma from \cite{CLS2020} and provide a sketch of the proof for completeness.

\begin{lem}\label{lemma:AdaptMaxEst_NL}
	Let $s \geq 0$ and $f \in  H^{s}(\mathbb{T}^2)$.
    Let $\Phi_t$ and $\Phi_t^N$ be the flows defined in \eqref{eq:nlsTruncated}.
    If for some $\delta  > 0$ we have
	\begin{align}\label{eq:equivconvergence}
		\lim_{N \rightarrow \infty} 
		\Big\| \sup_{0 < t < \delta} 
		| \Phi_t f   - \Phi^N_t f  | \Big\|_{L^2(\mathbb{T}^2)} = 0,
	\end{align}
    then $\lim_{t \to 0} \Phi_t f(x) = f(x)$ for almost every $x \in \mathbb{T}^2$.
\end{lem}
	\begin{proof}
		By the triangle inequality, decompose 
		\begin{align*}
			|\Phi_t f - f | 
            \leq 
            |\Phi_t f - \Phi^{N}_t f| 
            + |\Phi_t^{N} f - P_{\leq N} f| 
            + |P_{>N} f|.
		\end{align*}
        Regarding the second term, given that $P_{\leq N} f \in H^s$ for every $s$ and it is hence smooth, 
        then $\Phi_t^{N} f \rightarrow P_{\leq N} f$ as $t \rightarrow 0$ uniformly in $x \in \mathbb{T}^2$.

        On the other hand, by Chebyshev's inequality, for any $\lambda > 0$ we have
		\begin{align*}
			\big| \big\{ x : \limsup_{t \rightarrow 0} |\Phi_t f - f| > 2\lambda \big\} \big| 
			& \leq \big| \big\{ \, x  \,  : \,  \sup_{0 < t < \delta } 
			|\Phi_t f - \Phi^{N}_t f| > \lambda \big\}\big|  
			+ \big|\big\{ \,  x \,  : \,  |P_{>N} f| > \lambda \big\}\big|  \\ 
			& \lesssim  \frac{1}{\lambda^2} \left( \Big\| \sup_{0 < t < \delta } 
			|\Phi_t f - \Phi^{N}_t f| \Big\|^2_{L^2(\mathbb{T}^2)} 
			+ \| P_{>N}f \|^2_{L^2(\mathbb{T}^2)} \right).
		\end{align*}
		Taking the limit $N \to \infty$ we get 
        \begin{align}
            |\{ \, x  \in \mathbb{T}^2 \, : \, \limsup_{t \rightarrow 0} |\Phi_t f - f| > \lambda  \}| = 0, 
            \qquad \qquad \forall \lambda > 0. 
        \end{align}
        The result follows by considering the union of such sets over $n \in \mathbb N$ with $\lambda_n = 1/n$. 
	\end{proof}

\subsection{Linear-nonlinear decomposition}

To prove the sufficient condition \eqref{eq:equivconvergence} in Lemma~\ref{lemma:AdaptMaxEst_NL}, 
we find a linear-nonlinear decomposition of the solution
\begin{align}\label{eq:bourgain-trick}
    \Phi^N_t f^\omega = P_{\leq N} e^{it\Delta} f^{\omega} + D_N, 
    \qquad \qquad N \in 2^{\mathbb N} \cup \{ \infty \}, 
\end{align}
which is inspired by the Duhamel integral representation of the equation \eqref{eq:nlsTruncated},
\begin{align}\label{eq:Duhamel_Q}
    u(t,x) = P_{\leq N} e^{it\Delta} f + i \int_0^t e^{i(t - t')\Delta}  P_{\leq N} \mathcal Q (u(t',x)) \, dt'.
\end{align}
Since $f^\omega \in H^{\alpha - }$ a.s.  with $0 < \alpha < 1/2$, 
and it is therefore deterministically supercritical, 
we look for some nonlinear smoothing $\sigma > 0$ so that $D_N$ can be found in a deterministically subcritical Sobolev space $H^{\alpha + \sigma}$ with $\alpha + \sigma > 1/2$. 
Plugging \eqref{eq:bourgain-trick} in \eqref{eq:Duhamel_Q}, $D_N$ satisfies the equation 
\begin{equation}\label{eq:cauchy-nls-v}
		\left\{
		\begin{array}{rcl}
			i \, \partial_t D_N + \Delta D_N 
            & = &  P_{\leq N} \mathcal Q ( P_{\leq N} e^{it\Delta} f^{\omega} + D_N ),  
			\\
			D_N(0,x) & = & 0;
		\end{array}
		\right.
\end{equation}
and therefore, if $\alpha + \sigma > 1/2$, we will find it as a fixed point of the map 
\begin{align}\label{eq:map-fp-prob}
    \Gamma_N^\omega : X^{ \alpha + \sigma , b}_{\delta} \rightarrow X^{ \alpha + \sigma , b}_{\delta}, 
    \qquad \qquad 
    \Gamma_N^{\omega}(D) = i \int_0^t e^{i(t - t')\Delta} P_{\leq N} \mathcal Q ( P_{\leq N} e^{i t' \Delta} f^{\omega} + D) \, dt'.
\end{align}
Since $\alpha > 0$ is arbitrary, 
we need a smoothing $\sigma = 1/2 - \varepsilon$ with arbitrarily small $\varepsilon > 0$.

\begin{rmk}\label{rmk:D_to_0}
    Given that the probabilistic linear convergence in \eqref{eq:Linear_Convergence_Probabilistic} holds for $\alpha > 0$,       
    the linear-nonlinear decomposition in \eqref{eq:bourgain-trick} implies that 
    \begin{align}
        \lim_{t\to 0} e^{it\Delta}f^\omega = f^\omega 
        \qquad \Longrightarrow \qquad 
        \lim_{t\to 0} \Phi_t^N f^\omega = 
         P_{\leq N} f^{\omega}  + \lim_{t\to 0} D_N.
    \end{align}
    Thus, roughly speaking, the nonlinear convergence property amounts to proving $D_N \to 0$. 
\end{rmk}

A way to materialize the idea in Remark~\ref{rmk:D_to_0} is the following. 
By Lemma~\ref{lemma:AdaptMaxEst_NL} and \eqref{eq:bourgain-trick}, we write 
\begin{align}
    \Phi_t f - \Phi^N_t f
    & = \big(\Phi_t f - e^{it\Delta} f^\omega \big) 
    - \big( \Phi^N_t f - P_{\leq N} e^{it\Delta} f^\omega \big) 
    + P_{> N} e^{it\Delta} f^\omega \\
    & = D - D_N + P_{> N} e^{it\Delta} f^\omega, 
\end{align}
and therefore, 
\begin{equation}
    \Big\| \sup_{0 < t < \delta} 
		| \Phi_t f   - \Phi^N_t f  | \Big\|_{L^2_x(\mathbb{T}^2)}
        \leq 
        \Big\| \sup_{0 < t < \delta} 
		| D - D_N  | \Big\|_{L^2_x(\mathbb{T}^2)}
        + \Big\| \sup_{0 < t < \delta} 
		| P_{> N} e^{it\Delta} f^\omega | \Big\|_{L^2_x(\mathbb{T}^2)}.
\end{equation}
We first prove that
\begin{equation}\label{eq:Limit_Linear_Term}
    \lim_{N \to \infty} \Big\| \sup_{0 < t < \delta} 
		| P_{> N} e^{it\Delta} f^\omega | \Big\|_{L^2_x(\mathbb{T}^2)} = 0,
        \qquad \qquad \omega\text{-almost-surely, }
\end{equation}
for which we follow the strategy used in \cite{CLS2020} to prove the linear probabilistic convergence in \eqref{eq:Linear_Convergence_Probabilistic}.
Let $M \geq 1$.
By Bernstein's inequality in space-time and Lemma \ref{lemma:Large_Deviation_Schrodinger_Linear}, for every $\varepsilon > 0$ we get
\begin{align}
    \|P_M e^{i t \Delta} f^{\omega}\|_{L^{\infty}_{t,x}(\mathbb{T}^3)} \lesssim M^{6/q} \,  \|P_M e^{i t \Delta} f^{\omega}\|_{L^q_{t,x}(\mathbb{T}^3)}
    \lesssim 
    \Big(  \log \frac{1}{\varepsilon} \Big)^{1/2} \,  \frac{M^{6/q}}{M^{\alpha}},
\end{align}
with probability $\geq 1-\varepsilon$. As shown in Appendix~\ref{app:probprel}, this estimate can be made uniform in $M \in 2^{\mathbb N}$, in the sense that for every $\varepsilon > 0$, we have  
\begin{align}
    \|P_M e^{i t \Delta} f^{\omega}\|_{L^{\infty}_{t,x}(\mathbb{T}^3)} 
    \lesssim 
    \Big(  \log \frac{1}{\varepsilon} \Big)^{1/2} \,  \frac{(\log M)^{1/2}}{M^{\alpha - 6/q}}, \qquad \forall M \in 2^{\mathbb N}, 
\end{align}
with probability $\geq 1 - \varepsilon$, and hence, choosing $q > 6/\alpha$, for every $\varepsilon_1 > 0$ we get
\begin{align}
    \|P_{> N} e^{i t \Delta} f^{\omega}\|_{L^{\infty}_{t,x}(\mathbb{T}^3)} 
    \lesssim 
    \Big(  \log \frac{1}{\varepsilon} \Big)^{1/2} \,  \sum_{M > N}\frac{(\log M)^{1/2}}{M^{\alpha - 6/q}}
    \lesssim 
    \Big(  \log \frac{1}{\varepsilon} \Big)^{1/2} \,  \frac{1}{N^{\alpha - 6/q - \varepsilon_1}}, 
    \qquad \forall N \in 2^{\mathbb N}, 
\end{align}
with probability $\geq 1 - \varepsilon$. 
This corresponds to the tail bound
\begin{align*}
    \mathbb{P}(\|P_{>N} e^{i t \Delta} f^{\omega}\|_{L^{\infty}_{t,x}(\mathbb{T}^3)}  > \lambda) \leq e^{-c \lambda^2 N^{2\alpha - \frac{12}{q} - 2\varepsilon_1}}, 
    \qquad \forall \lambda > 0, \qquad \forall N \in 2^{\mathbb N}, 
\end{align*}
which by Borel-Cantelli implies
\begin{align*}
        \mathbb{P} \Big(\liminf_{N \rightarrow \infty} \big\{ \omega \in \Omega : \|P_{> N} e^{it\Delta}f^{\omega}\|_{L^{\infty}_{t,x}(\mathbb{T}^3)} \leq \lambda \big\} \Big) = 1.
    \end{align*}
    Thus, for almost every $\omega \in \Omega$ there exists $N^{\omega}_{\lambda}$ such that
    \begin{align*}
        \|P_{> N} e^{it\Delta}f^{\omega}\|_{L^{\infty}_{t,x}(\mathbb{T}^3)} \leq \lambda, 
        \qquad \forall N \geq N^{\omega}_{\lambda}.
    \end{align*}
    Since $\lambda > 0$ is arbitrary, we conclude that  
    \begin{align}\label{eq:prob-tail}
        \lim_{N \to \infty} 
        \Big\| \sup_{0 < t < \delta} 
		| P_{> N} e^{it\Delta} f^\omega | \Big\|_{L^2_x(\mathbb{T}^2)} \lesssim 
        \lim_{N \to \infty} \|P_{> N} e^{it\Delta}f^{\omega}\|_{L^{\infty}_{t,x}(\mathbb{T}^3)} = 0. 
    \end{align}

    On the other hand, given that the linear maximal estimate in \eqref{eq:max-est} is valid for $s > d/(d+2) = 1/2$, using the standard transference principle in Lemma \ref{lemma:TransferLemma} in the space $Y = L^{2}(\mathbb{T}^2; L^{\infty}(0, \delta))$ we get
	\begin{align}\label{eq:maxest-rest}
		\Big\| \sup_{0 \leq t \leq \delta} | D - D_N | \Big\|_{L^2(\mathbb{T}^2)}  
        \lesssim
		\|  D  - D_N \|_{X^{s,b}_{\delta }},
        \qquad \text{ for } b > 1/2 \text{ and } s > 1/2.
	\end{align}
    Thus, to prove Theorem~\ref{thm:pconv} it suffices to find $D \in X^{s, b}_\delta$ with $s = \alpha + \sigma > 1/2$ as a solution from \eqref{eq:map-fp-prob}, 
    and in addition, by \eqref{eq:Limit_Linear_Term} and \eqref{eq:maxest-rest}, to prove 
    \begin{align}\label{eq:red-pconv-rest}
        \lim_{N \rightarrow \infty} 
        \|  D  - D_N \|_{X^{\alpha + \sigma,b}_{\delta }} 
        = 0.
    \end{align}

    \subsection{Nonlinear smoothing and proof of Theorem~\ref{thm:pconv}}
    \label{sec:NonlinearSmoothing_and_ProofOfTheorem}

    We now state the almost sure nonlinear smoothing of $ \sigma < 1/2$
    and the corresponding local well-posedness result. 
    In what comes, $X^{s,b}$ spaces are taken such that 
    \begin{align}
        b = \frac12 + \beta, 
        \qquad \text{ and } \qquad 
        b' = \frac12 + \beta', 
        \qquad \qquad 
    \end{align}
    with $\beta, \beta' > 0$ small enough.
    The datum $f^{\omega} \in H^{\alpha-}$ is always as in \eqref{eq:random-data_2D}.

\begin{thm}\label{thm:to-pconv}
    Let $\alpha > 0$. 
    Let $0 \leq \sigma < 1/2$ 
    such that $\alpha + \sigma > 1/2$. 
    Let $\delta > 0$. 
    Then, for every $\varepsilon > 0$, there exists $A_\varepsilon \subset \Omega$ with $\mathbb{P}(A_{\varepsilon}) \geq 1- \varepsilon$ such that for every $D,D' \in X^{\alpha + \sigma,b}_\delta$ with 
    $
        \|D\|_{X^{\alpha + \sigma,b}_\delta}, \|D'\|_{X^{\alpha + \sigma,b}_\delta} \leq 1,
    $
    we have
\begin{align}
    \big\| \mathcal{Q}(P_{\leq M} e^{it\Delta}f^{\omega} + D) \big\|_{X^{\alpha + \sigma, b'-1}_\delta}
    & \lesssim (\log (1/\varepsilon))^{5/2}, \label{eq:QsmoothingD} \\
    \big\| \mathcal{Q}(P_{\leq M}e^{it\Delta} f^{\omega} + D) - \mathcal{Q}(P_{\leq M}e^{it\Delta} f^{\omega} + D') \big\|_{X^{\alpha + \sigma, b'-1}_{\delta}}  
    & \lesssim (\log (1/\varepsilon))^2 \,  \|D - D'\|_{X^{\alpha + \sigma,b}_{\delta}}, \label{eq:QsmoothingDD'} \\
    \big\| \mathcal{Q}(e^{it\Delta}f^{\omega} + D) - \mathcal{Q}(e^{it\Delta} P_{\leq M} f^\omega + D) \big\|_{X^{\alpha + \sigma, b'-1}_\delta} 
    & \lesssim \frac{(  \log (1/\varepsilon) )^{5/2}}{M^{\varepsilon_1}}, \quad 0 < \varepsilon_1 \ll 1, \label{eq:Qsmoothing_LinearDifference}
\end{align} 
for every $\omega \in A_{\varepsilon}$ and for every $M \in 2^{\mathbb N} \cup \{\infty\}$.
\end{thm}
We postpone the proof of Theorem~\ref{thm:to-pconv} to Sections \ref{sec:smoothing} and \ref{sec:proof-smoothing}.
As usual, as a consequence of \eqref{eq:QsmoothingD} and \eqref{eq:QsmoothingDD'}, we get the following local well-posedness result. 

\begin{coro}\label{coro:LWP}
    Let $\alpha > 0$ and $f^{\omega}$ as in \eqref{eq:random-data_2D}. 
    Let $0 \leq \sigma < 1/2$ such that $\alpha + \sigma > 1/2$. Then, for every $\varepsilon > 0$ small enough, the following is true with probability $\geq 1 - \varepsilon$. There exists $\delta > 0$ such that the Cauchy problem
    \begin{equation}\label{eq:cauchy-nls-gauge-pre}
    \left\{
    \begin{array}{rcl}
         i\, \partial_t D_N +  \Delta D_N & = & P_{\leq N}\mathcal{Q}( P_{\leq N} e^{it\Delta} f^{\omega} + D_N), \\
         D_N(0,x) & = & 0,
    \end{array}
    \right.
        \end{equation}
    admits a unique solution $D_N \in X^{\alpha + \sigma, b}_\delta$ for any $N \in 2^{\mathbb N} \cup \{\infty\}$.

    Consequently, for almost every $\omega$ there exists $\delta_\omega > 0$ such that for every $N \in 2^{\mathbb N} \cup \{\infty \}$, \eqref{eq:cauchy-nls-gauge-pre} admits a unique solution $D_N \in X^{\alpha + \sigma,  b}_{\delta_\omega}$. 
        In particular, $u = P_{\leq N} e^{it\Delta}f^\omega + D_N$ solves \eqref{eq:nlsTruncated} and 
        \begin{align}\label{eq:Hs-nsmoothing}
            u - P_{\leq N}e^{it\Delta} f^{\omega} = D_N \in C([-\delta_\omega,\delta_\omega],H^{\alpha + \sigma}(\mathbb T^2)).
        \end{align}
\end{coro}

\begin{proof}
It suffices to prove that the map $\Gamma_N^\omega$ in \eqref{eq:map-fp-prob} defined on the ball
\begin{align}
    B_\delta 
    = \{ D \in X^{\alpha + \sigma,b}_{\delta} \, : \, \|D\|_{X^{\alpha + \sigma, b}_\delta} \leq 1 \}
\end{align}
is a contraction. 
Let $\varepsilon > 0$ and $\gamma > 0$ small enough. By Lemma \ref{lemma:pre-rest-spaces} and \eqref{eq:QsmoothingD} in Theorem~\ref{thm:to-pconv},
\begin{align}\label{eq:Before_Contraction_Pre}
    \|\Gamma_N^\omega(D)\|_{X^{\alpha + \sigma,b}_{\delta}} \lesssim \delta^{\gamma} \| \mathcal{Q} (P_{\leq N}e^{it\Delta} f^{\omega} + D)\|_{X^{\alpha + \sigma, b - 1 + \gamma}_{\delta}} \lesssim \delta^{\gamma}\Big(  \log \frac{1}{\varepsilon} \Big)^{5/2}, 
    \qquad \forall D \in B_\delta, 
\end{align}
with probability $\geq 1 - \varepsilon$, uniformly in $N$. Similarly, by \eqref{eq:QsmoothingDD'},
for $D_1, D_2 \in B_{\delta}$ we have 
\begin{align}\label{eq:Before_Contraction}
    \|\Gamma_N^\omega(D_1) - \Gamma_N^\omega(D_2)\|_{X^{\alpha + \sigma,b}_{\delta}} &\lesssim \delta^{\gamma}  \|\mathcal{Q}(P_{\leq N}e^{it\Delta}f^{\omega} + D_1) - \mathcal{Q}(P_{\leq N}e^{it\Delta}f^{\omega} + D_2)\|_{X^{\alpha + \sigma,b - 1 + \gamma}_{\delta}} \\
    &\lesssim \delta^{\gamma} \Big(  \log \frac{1}{\varepsilon} \Big)^2\|D_1 - D_2\|_{X^{\alpha + \sigma,b}_\delta},
\end{align}
with probability $\geq 1 - \varepsilon$, uniformly in $N$. 
Therefore, 
choosing $\delta = \delta_{\varepsilon, \gamma} > 0$ small enough such that
\begin{align}\label{eq:Condition_on_delta_epsilon_gamma}
    \delta^{\gamma}\Big(  \log \frac{1}{\varepsilon} \Big)^{5/2} < 1  \qquad \text{ and } \qquad  \delta^{\gamma}\Big(  \log \frac{1}{\varepsilon} \Big)^2 < 1,
\end{align}
there is a set $S_\delta$ with 
$\mathbb P(S_\delta) \geq 1 - \varepsilon$ such that $\Gamma_N^\omega$ is a contraction in $B_\delta$ for every $N$ and for every $\omega \in S_\delta$, 
and hence it has a unique fixed point $D_N$ which solves the Cauchy problem \eqref{eq:cauchy-nls-gauge-pre}.

For the second part of the statement, let
\begin{align}
S = \bigcup_{n \in \mathbb{N}} S_{1/n},
\qquad \text{ such that } \qquad 
\mathbb{P}(S^c) 
= \mathbb{P}\Big(\bigcap_{n \in \mathbb{N}} S_{1/n}^c \Big) 
\leq \mathbb{P}(S_{1/k}^c) 
\lesssim e^{-k^{\frac{2\gamma}{5}}},
\quad \forall k \in \mathbb N, 
\end{align}
where the last inequality comes from \eqref{eq:Condition_on_delta_epsilon_gamma}.
Letting $k \to \infty$ implies $\mathbb P(S) = 1$. 
Therefore, for almost every $\omega$, there exists $\delta_\omega > 0$ such that $\omega \in S_{\delta_\omega}$, which proves the statement. 
\end{proof}

Once we have the well-posedness,
Theorem~\ref{thm:to-pconv} and in particular \eqref{eq:Qsmoothing_LinearDifference} imply Theorem~\ref{thm:pconv} for the renormalized equation \eqref{eq:cauchy-nls-gauge}. 
We will extend the result to the original equation \eqref{eq:cauchy-nls_2D} in Section~\ref{sec:renorm}. 

\begin{proof}[Proof of Theorem~\ref{thm:pconv} for renormalized equation in \eqref{eq:cauchy-nls-gauge}]
By Corollary~\ref{coro:LWP}, almost every $\omega$ admits a time of existence $\delta_{\omega}>0$.
Hence, from \eqref{eq:red-pconv-rest}, it suffices to prove that 
    \begin{align}\label{eq:Final_Convergence_Property}
        \lim_{N \rightarrow \infty} \| D  - D_N \|_{X^{\alpha + \sigma, b}_{\delta_\omega}}  = 0.
    \end{align}
Denote, for simplicity, $s = \alpha + \sigma > 1/2$. 
Since $D$ and $D_N$ are the fixed points of $\Gamma^\omega$ and $\Gamma^\omega_N$ respectively, from Lemma \ref{lemma:pre-rest-spaces} we have
    \begin{align}\label{eq:Before_Absortion}
        \|  D - D_N \|_{X^{s,b}_{\delta_{\omega}}} 
        & 
        = \|  \Gamma^\omega(D) -  \Gamma^\omega_N(D_N) \|_{X^{s,b}_{\delta_{\omega}}} \\
        & \lesssim 
        \big\|  \mathcal{Q}(e^{it\Delta}  f^{\omega} + D) - P_{\leq N} \mathcal{Q}(e^{it\Delta} P_{\leq N} f^{\omega} + D_N) \big\|_{X^{s,b-1}_{\delta_{\omega}}}.
    \end{align}
    Decompose the difference inside the norm as
    \begin{align}
        \mathcal{Q}(e^{it\Delta}  f^{\omega} + D) - P_{\leq N} \mathcal{Q}(e^{it\Delta} P_{\leq N} f^{\omega} + D_N) 
        = I + II + III,  
    \end{align}
    where 
    \begin{align}
        I & = P_{\leq N} \left[ \mathcal{Q}(e^{it\Delta}  f^{\omega} + D) - \mathcal{Q}(e^{it\Delta} P_{\leq N} f^{\omega} + D)  \right], \\ 
        II & = P_{\leq N} \left[ \mathcal{Q}(e^{it\Delta} P_{\leq N} f^{\omega} + D) - \mathcal{Q}(e^{it\Delta} P_{\leq N} f^{\omega} + D_N) \right],  \\
        III & = P_{> N}\mathcal{Q}(e^{it\Delta}  f^{\omega} + D). 
    \end{align}
    By the bound in \eqref{eq:Before_Contraction}
    and by the choice of $\gamma$ and $\delta_\omega$ in \eqref{eq:Condition_on_delta_epsilon_gamma},  we have 
    \begin{align*}
        \| II \|_{X^{s,b-1}_{\delta_{\omega}}}
        \lesssim 
        \delta_{\omega}^{\gamma} \, 
        \| II \|_{X^{s,b-1+\gamma}_{\delta_{\omega}}} 
        <  \|D - D_N\|_{X^{s,b}_{\delta_{\omega}}}.
    \end{align*}
    It is thus absorbed in the left-hand side of \eqref{eq:Before_Absortion}, and hence 
    \begin{align}\label{eq:After_Absortion}
        \|  D - D_N \|_{X^{s,b}_{\delta_{\omega}}} 
        \lesssim  \big\|  I \big\|_{X^{s,b-1}_{\delta_{\omega}}} + \big\|  III \big\|_{X^{s,b-1}_{\delta_{\omega}}}.
    \end{align}
    Similarly, again from the choice in \eqref{eq:Condition_on_delta_epsilon_gamma}, from inequality \eqref{eq:Qsmoothing_LinearDifference} we get
    \begin{align}
        \| I \|_{X^{s,b-1}_{\delta_{\omega}}}
        \leq \delta_\omega^\gamma \,  \big\|   \mathcal{Q}(e^{it\Delta}  f^{\omega} + D) - \mathcal{Q}(e^{it\Delta} P_{\leq N} f^{\omega} + D) \big\|_{X^{s,b-1+\gamma}_{\delta_{\omega}}}
        & \lesssim  \frac{1}{N^{\varepsilon_1}},
    \end{align}
    for $\varepsilon_1 > 0$ small enough, 
    and therefore it converges to $0$ as $N \rightarrow \infty$. 
    On the other hand, 
    $III$ being the tail of 
    $\mathcal{Q}(e^{it\Delta}  f^{\omega} + D)$,  and $\| \mathcal{Q}(e^{it\Delta}  f^{\omega} + D) \|_{X^{s, b-1}_{\delta_{\omega}}}$ being finite by \eqref{eq:Before_Contraction_Pre}, 
    the norm of $III$ also converges to $0$ as $N \rightarrow \infty$. Thus, we get \eqref{eq:Final_Convergence_Property} as we wanted. 
    \end{proof}

\subsection{Justification of renormalization with gauge}\label{sec:renorm}
Once we proved Theorem~\ref{thm:pconv} for the renormalized equation \eqref{eq:cauchy-nls-gauge}, let us now prove it for the original equation in \eqref{eq:cauchy-nls_2D}. 
The argument below is based on \cite{Staffilani1997}, which we include for completeness. 

For almost every $\omega$, by Corollary~\ref{coro:LWP} we have the time of existence  $\delta_{\omega}$ of $D$. Denote 
\begin{align*}
    Y^{s, b}_{\delta_{\omega}}
    = \big\{ e^{it\Delta} f^{\omega} + D \, : \,  \|D\|_{X^{s, b}_{\delta_{\omega}}} \leq 1 \big\}, 
    \quad \quad s = \alpha + \sigma > \frac12, \quad b = \frac12 + \beta, 
\end{align*}
which we endow with the norm
\begin{align*}
    \|e^{it\Delta}f^\omega + D\|_{Y^{s, b}_{\delta_{\omega}}} = \| D\|_{X^{s, b}_{\delta_{\omega}}}.
\end{align*}
Observe that for any $\varepsilon_1 > 0$ we have
\begin{align}\label{eq:well-definition-Q}
    \left\| u \right\|_{L^4_{t,x}([0,\delta_\omega] \times \mathbb{T}^2)} \lesssim \|e^{i t \Delta} f^{\omega} \|_{L^4_{t,x}(\mathbb T \times \mathbb{T}^2)} + \| D \|_{X^{\varepsilon_1, b}_{\delta_{\omega}}} < \infty
    \qquad \quad \forall u \in Y^{s,b}_{\delta_{\omega}},  
\end{align}
since the $L^4$ norm of the random linear evolution is almost surely finite by its sub-gaussianity in Lemma \ref{lemma:Large_Deviation_Schrodinger_Linear},  and the bound for $D$ comes from the Strichartz estimate in \eqref{eq:strichartz} for $p=4$ and the transference principle in Lemma \ref{lemma:TransferLemma}. 
Hence, $\left\| u(t) \right\|_{L^4_x(\mathbb{T}^2)} < \infty$ for almost every $t \in [0,\delta_\omega]$. 

To transfer local well-posedness, we define the gauge operator $\mathcal{G} : Y^{s,b}_{\delta_{\omega}} \rightarrow \mathcal{G}(Y^{s,b}_{\delta_{\omega}})$ by
\begin{align*}
    \mathcal{G}u(t,x) = e^{-i\int_0^t \beta_u(t')dt'} \, u(t,x) , \qquad \qquad \beta_u(t) = 3 \int_{\mathbb{T}^2} |u(t,x)|^4 dx,
\end{align*}
which by \eqref{eq:well-definition-Q} is well-defined for almost every $\omega$. 
It is also injective since, from the definition, 
\begin{align}
    \mathcal Gu_1 = \mathcal G u_2
    \quad \Longrightarrow \quad 
    |u_1| = |u_2|
    \quad \Longrightarrow \quad 
    \beta_{u_1} = \beta_{u_2}
    \quad \Longrightarrow \quad 
    u_1 = u_2.
\end{align}
If we endow the image space with the metric
\begin{align*}
    d(v_1,v_2) = \big\| \mathcal{G}^{-1} (v_1) - \mathcal{G}^{-1}(v_2) \big\|_{Y^{s,b}_{\delta_{\omega}}}, 
\end{align*}
as shown in \cite{Staffilani1997} the space $(\mathcal{G}(Y^{s,b}_{\delta_{\omega}}),d)$ is a complete metric space.

For existence, if $u = e^{it\Delta}f^\omega + D$ is the unique solution to the renormalized Cauchy problem \eqref{eq:cauchy-nls-gauge} in $Y^{s,b}_{\delta_{\omega}}$ with datum $f^{\omega}$, a direct computation shows that $v = \mathcal G u \in \mathcal{G}(Y^{s,b}_{\delta_{\omega}})$ is a solution to \eqref{eq:cauchy-nls_2D} with the same datum $f^{\omega}$.

Regarding uniqueness, let $v_1,v_2$ be solutions to \eqref{eq:cauchy-nls_2D} in 
$ (\mathcal{G}(Y^{s,b}_{\delta_{\omega}}),d)$ with datum $f^{\omega}$.
Again, direct computation shows that $u_1 = \mathcal G^{-1}v_1$ and $u_2 = \mathcal G^{-1}v_2$ are solutions to the renormalized equation \eqref{eq:cauchy-nls-gauge} in $Y^{s,b}_{\delta_\omega}$ with datum $f^\omega$, and hence $u_1 = u_2$ by uniqueness.
Then, 
\begin{align}
    d(v_1, v_2) = \lVert u_1 - u_2 \rVert_{Y^{s,b}_{\delta_{\omega}}} = 0, 
\end{align}
and therefore $v_1 = v_2$. 
Moreover, this unique solution has the form 
\begin{align*}
    v = 
    e^{-3 i \int_0^t  \int_{\mathbb T^2} |e^{i t' \Delta} f^\omega(x) + D(t',x)|^4 \, dx dt'} \, 
    (e^{it\Delta} f^{\omega} + D), \qquad \qquad 
    \|D\|_{C([0,\delta_{\omega}] , H^s(\mathbb{T}^2))} \leq 1.
\end{align*}

Finally, we transfer the property of convergence to initial data.
Let $v$ be a solution to \eqref{eq:cauchy-nls_2D} with datum $f^{\omega}$, 
and $u = \mathcal G^{-1} v$ the corresponding solution to \eqref{eq:cauchy-nls-gauge}.
Then, since $v = \mathcal G u$, we have
\begin{align}
    \lim_{t \to 0} v(t,x) = \lim_{t \to 0} e^{-i \int_0^t \beta_u(t')\, dt'} \, u(t,x).
\end{align}
By \eqref{eq:well-definition-Q} and the dominated convergence theorem,  we have
\begin{align}\label{eq:After_DCT}
    \lim_{t \rightarrow 0} e^{i\int_0^t \beta_u(t') dt'} = 1, \qquad \qquad \omega\text{-almost surely, }
\end{align}
so if $u$ satisfies the pointwise convergence property in \eqref{eq:cauchy-nls_2D_Convergence}, we conclude that
\begin{align}\label{eq:Before_DCT}
    \lim_{t \to 0} v(t,x)
    = f^{\omega}(x) \qquad \text{ a.e., }  \qquad \omega\text{-almost surely}, 
\end{align}
and hence Theorem~\ref{thm:pconv} is proved.

\section{Proof of the nonlinear smoothing in Theorem~\ref{thm:to-pconv}}
\label{sec:smoothing}

To simplify notation, let us denote the random linear evolution by $R = e^{it\Delta}f^\omega$, in such a way that the operator $\mathcal Q$ is evaluated in linear-nonlinear decompositions of the form
\begin{align}
    u = P_{\leq M} R + D, \qquad \qquad \text{ for } \quad M \in 2^\mathbb N \cup \{\infty\}. 
\end{align}
In the proofs of \eqref{eq:QsmoothingD} and \eqref{eq:QsmoothingDD'}, the arguments are independent of $M$, and hence we will only write explicitly the case $M = \infty$. 
For \eqref{eq:Qsmoothing_LinearDifference}, the role of $M$ is important and we will keep it. 

Slightly abusing notation, we treat the nonlinear operator $\mathcal Q$ in \eqref{eq:cauchy-nls-gauge} as a multilinear operator\footnote{Strictly speaking, the multilinear operator also has the terms $u_3 \int_{\mathbb T^2} u_1 \overline{u_2}  \overline{u_4} u_5 \, dx$ and $u_5 \int_{\mathbb T^2}  u_1 \overline{u_2} u_3 \overline{u_4} \, dx$. However, $u_1$, $u_3$ and $u_5$ (as well as $u_2$ and $u_4$) are symmetric and hence it suffices to work with \eqref{eq:Q_Multilinear}.}
\begin{align}\label{eq:Q_Multilinear}
    \mathcal{Q}(u_1, u_2, u_3, u_4, u_5)
    = u_1 \overline{u_2} u_3 \overline{u_4} u_5 - u_1 \int_{\mathbb T^2} \overline{u_2} u_3 \overline{u_4} u_5 \, dx, 
\end{align}
in such a way that to prove \eqref{eq:QsmoothingD}, given $s = \alpha + \sigma$ and $b' = 1/2 + \beta'$, we need to bound 
\begin{align}\label{eq:decomp-symm-ui}
    \|\mathcal{Q}(R + D)\|_{X^{s,b'-1}_{\delta}} 
    & \leq  \sum_{u_i \in \{R, \,  D\}} \|\mathcal{Q}(u_1, u_2, u_3, u_4, u_5)\|_{X^{s,b'-1}_{\delta}} 
    \lesssim \sup_{u_i \in \{R, \,  D\} } \|\mathcal{Q}(u_1, u_2, u_3, u_4, u_5)\|_{X^{s,b'-1}_{\delta}}. 
\end{align} 
On the other hand, the left-hand side of \eqref{eq:QsmoothingDD'} can be bounded as
\begin{align}\label{eq:diff-ineqsDD'}
    \begin{array}{rrrrrrrr}
        \| \mathcal{Q}(R + D) - \mathcal{Q}(R + D') \|_{X^{s,b'-1}_{\delta}}
        \leq \big\| \mathcal{Q}( D - D',  & R + D, &  R + D, & R + D, & R + D) \big\|_{X^{s,b'-1}_{\delta}}  \\
        + \big\| \mathcal{Q}( R + D', & D - D', & R + D, & R + D, & R + D) \big\|_{X^{s,b'-1}_{\delta}}  \\ 
        + \big\| \mathcal{Q}( R + D', & R + D', & D - D', & R + D, & R + D) \big\|_{X^{s,b'-1}_{\delta}}  \\ 
        + \big\| \mathcal{Q}( R + D', & R + D', & R + D', & D - D', & R + D) \big\|_{X^{s,b'-1}_{\delta}}  \\
        + \big\| \mathcal{Q}( R + D', & R + D', & R + D', & R + D', & D - D') \big\|_{X^{s,b'-1}_{\delta}}.
    \end{array}
\end{align}
The same procedure shows that \eqref{eq:Qsmoothing_LinearDifference} can be bounded similarly, since 
\begin{align}\label{eq:diff-ineqs}
    & \| \mathcal{Q}(R + D) - \mathcal{Q}(P_{\leq M} R + D) \|_{X^{s,b'-1}_{\delta}}  \\
    & \qquad \qquad \qquad \begin{array}{rrrrrrrr}
        \leq \big\| \mathcal{Q}( P_{> M} R,  & P_{\leq M} R + D, & P_{\leq M} R + D, & P_{\leq M} R + D, & P_{\leq M} R + D) \big\|_{X^{s,b'-1}_{\delta}}  \\
        + \big\| \mathcal{Q}( R + D, & P_{> M} R , & P_{\leq M} R + D, & P_{\leq M} R + D, & P_{\leq M} R + D) \big\|_{X^{s,b'-1}_{\delta}}  \\ 
        + \big\| \mathcal{Q}( R + D, & R + D, & P_{> M} R, & P_{\leq M} R + D, & P_{\leq M} R + D) \big\|_{X^{s,b'-1}_{\delta}}  \\ 
        + \big\| \mathcal{Q}( R + D, & R + D, & R + D, & P_{> M} R, & P_{\leq M} R + D) \big\|_{X^{s,b'-1}_{\delta}}  \\
        + \big\| \mathcal{Q}( R + D, & R + D, & R + D, & R + D, & P_{> M} R) \big\|_{X^{s,b'-1}_{\delta}},
    \end{array}
\end{align}
and again, the proof will not see the difference between $P_{\leq M} R$ and $R$, so we treat all of them as $R$. However, it will see the effect of the tails $P_{> M} R$.
Hence, to prove Theorem \ref{thm:to-pconv} the following proposition suffices. 

\begin{prop}\label{prop:Qsmoothing-simp}
    Let $\delta>0$, $\alpha > 0$ and  $0 \leq \sigma < 1/2$ such that $\alpha + \sigma > 1/2$.  Let $M_i \geq 0$ and $u_i \in \{P_{\geq M_i} R,D\}$ for $i = 1, \ldots, 5$. 
    Define the partition $\mathcal R, \mathcal{D} \subset \{1, \ldots, 5\}$ by
    \begin{align}\label{eq:R-D-sets}
        \mathcal R = \{\, i \, : \, u_i = P_{\geq M_i} R \, \}, 
        \qquad \qquad 
        \mathcal D = \{\, i \, : \, u_i = D \, \}.
    \end{align}
    Then, for every $\varepsilon > 0$, there exists $A_{\varepsilon} \subset \Omega$ with $\mathbb{P}(A_{\varepsilon}) \geq 1 - \varepsilon$ such that
    \begin{align}\label{eq:Qsmoothing-simp-1}
        \|\mathcal{Q}(u_1,u_2,u_3,u_4,u_5) \|_{X^{\alpha + \sigma, b'-1}_{\delta}} 
        \lesssim \Big(  \log \frac{1}{\varepsilon} \Big)^{|\mathcal{R}|/2} \|D\|_{X^{\alpha + \sigma,b}}^{|\mathcal{D}|} \prod_{i \in \mathcal{R}} \frac{1}{\langle M_i \rangle^{\varepsilon_1} }, 
        \qquad \forall \omega \in A_\varepsilon, \quad 0< \varepsilon_1 \ll 1.
    \end{align}
\end{prop}
Indeed, the choice $M_i = 0$ for all $i$ directly implies  \eqref{eq:QsmoothingD}, and it also implies \eqref{eq:QsmoothingDD'} if we treat the cases $u_i \in \{D', \, D - D' \}$ identically to $u_i = D$. 
Finally, \eqref{eq:Qsmoothing_LinearDifference} follows by choosing $M_i = M$ for the relevant index $i$ and $M_j = 0$ for the rest $j \neq i$.

From now on, we fix a choice of 
$u_i \in \{ P_{\geq M_i} R, D\}$ for $i = 1, \ldots, 5$. 
To simplify notation, denote again $s = \alpha + \sigma$, and set $M_i = 0$ for $i \in \mathcal D$ in such a way that if we decompose each $u_i$ into Littlewood-Paley projections, we can write
\begin{align}\label{eq:decomp-lp}
    \lVert& \mathcal{Q}(u_1, u_2, u_3, u_4, u_5) \rVert_{X^{s,b'-1}_\delta}^2 \\
    & \leq  
    \sum_N
    N^{2s} 
    \bigg( 
    \sum_{\substack{N_1, \ldots, N_5 \\ N_i \geq M_i }} \lVert P_N \mathcal{Q}(P_{N_1}u_{1}, P_{N_2}u_{2}, P_{N_3} u_{3}, P_{N_4}u_{4}, P_{N_5}u_{5}) \rVert_{X^{0,b'-1}_\delta} \bigg)^2  \\
    & \lesssim 
    \sum_N
    N^{2s} \bigg( 
    \sum_{\substack{N_{(1)}\geq \ldots\geq N_{(5)} \\ N_i \geq M_i}} \sum_{\gamma \in S_5} \lVert P_N \mathcal{Q} (P_{N_{(\gamma(1))}}u_{1}, P_{N_{(\gamma(2))}}u_{2}, P_{N_{(\gamma(3))}} u_{3}, P_{N_{(\gamma(4))}}u_{4}, P_{N_{(\gamma(5))}}u_{5}) \rVert_{X^{0,b'-1}_\delta} \bigg)^2  \\
    & \lesssim 
    \sup_{\gamma \in S_5}
    \sum_N
    N^{2s} 
    \bigg( 
    \sum_{\substack{N_{(1)}\geq \ldots\geq N_{(5)} \\ N_i \geq M_i}}  \lVert P_N \mathcal{Q} (P_{N_{(\gamma(1))}}u_{1}, P_{N_{(\gamma(2))}}u_{2}, P_{N_{(\gamma(3))}} u_{3}, P_{N_{(\gamma(4))}}u_{4}, P_{N_{(\gamma(5))}}u_{5}) \rVert_{X^{0,b'-1}_\delta} \bigg)^2, 
\end{align}
where $S_5$ is the symmetric group of permutations of $5$ elements. 
Here, each $\gamma \in S_5$ determines how the frequencies $N_1, \ldots, N_5$ are ordered, in such a way that $N_i = N_{(\gamma(i))}$, and $N_{(i)}$ denotes the $i$-th largest frequency.
We denote by $u_{(i)}$ the input element corresponding to the frequency $N_{(i)}$. 

From now on, we fix an order $\gamma \in S_5$. We will prove Proposition~\ref{prop:Qsmoothing-simp} by separating cases in the upcoming subsections. 
Before proceeding, we give a lemma about the integrability of $R$ and $D$.
\begin{lem}\label{lemma:L8bound}
For every $\delta_1 > 0$ and $b > 1/2$,  we have 
\begin{equation}\label{eq:L8bound-1}
    \|
     D 
    \|_{L^{8 + \delta_1}_{t,x}(\mathbb{T} \times \mathbb{T}^2)}
    \lesssim 
    \|
     D 
    \|_{X^{\frac12 + \delta_1, b}}.
\end{equation}
Let $R = e^{it\Delta} f^\omega$ with $f^\omega$ as in \eqref{eq:random-data_2D}, and let $p \geq 1$. For every $\varepsilon > 0$, with probability $\geq 1 - \varepsilon$ we have 
\begin{equation}\label{eq:Lp_bound-Prob}
        \lVert  P_{N}R \rVert_{L_{t,x}^p(\mathbb{T} \times \mathbb{T}^2)}
        \lesssim \Big( \log \frac{1}{\varepsilon} \Big)^{1/2} \frac{(\log N )^{1/2}}{N^{\alpha}}, 
        \qquad \quad \forall N \in 2^{\mathbb N}.
    \end{equation}
\end{lem}
\begin{proof}
The first property follows from the $L^8$ Strichartz estimate in \eqref{eq:strichartz} and the transference principle in Lemma \ref{lemma:TransferLemma}.
The second inequality is proved in Lemma~\ref{lemma:Strichartz_Prob_Uniform}. 
\end{proof}

\subsection{Multilinear estimates} \label{subsec:multi-est} 

Recall the notation $\iota_j$ for conjugates in Section~\ref{sec:Notation}. 
We start with the following deterministic multilinear estimates.

\begin{lem}\label{lemma:Multilinear_Estimates}
Let $s \geq 0$
and $N, N_1, \ldots, N_5 \in 2^{\mathbb{N}}$. For any $\varepsilon_1 > 0$ we have
\begin{align}\label{eq:Multilineal_Main}
    \bigg\lVert P_N \Big( \prod_{i=1}^5 (P_{N_i} u_i)^{\iota_i} \Big) 
    \bigg\rVert_{X^{s, b'-1}} 
    \lesssim 
    N_{(2)}^{\varepsilon_1} \, 
    \lVert P_{N_{(1)}} u_{(1)} \rVert_{X^{s,b}} \, 
    \prod_{i \, : \, N_i \neq N_{(1)}}
    \lVert P_{N_{(i)}} u_{(i)} \rVert_{L^8_{x,t}}, 
    \end{align}
and     \begin{align}\label{eq:Multilineal_Secondary}
    \bigg\lVert P_N \bigg[ 
    P_{N_1} u_1 \, 
    \bigg(\int \prod_{i=2}^5 (P_{N_i} u_i)^{\iota_i}  \, dx 
    \bigg)
    \bigg] 
    \bigg\rVert_{X^{s, b'-1}} 
    \lesssim 
    N_{(2)}^{\varepsilon_1} \, 
    \lVert P_{N_{(1)}} u_{(1)} \rVert_{X^{s,b}} \, 
    \prod_{i \, : \, N_i \neq N_{(1)}} \lVert P_{N_{(i)}} u_{(i)} \rVert_{L^8_{x,t}}.
    \end{align}
\end{lem}

With this, we can prove Proposition \ref{prop:Qsmoothing-simp} when the highest frequency component is a $D$, that is, whenever $u_{(i)} = D$ for some $N_{(i)} \simeq N_{(1)}$. 

\begin{prop}\label{prop:Qsmoothing-mest}
Let $s > 1/2$, and suppose that $u_{(1)} = D$. Let $\varepsilon > 0$. Then, with probability at least $1-\varepsilon$ and for any $\varepsilon_1 >0$ we have
    \begin{align}\label{eq:Qsmoothing-mest}
    &\sum_N N^{2s} \bigg( \sum_{\substack{N_{(1)}\geq \ldots\geq N_{(5)} \\ N_i \geq M_i}}  \|P_N \mathcal{Q}(P_{N_1} u_1 , P_{N_2} u_{2} , P_{N_3} u_{3} , P_{N_4} u_{4} , P_{N_5} u_{5})\|_{X^{0,b'-1}} \bigg)^2 \nonumber \\
    & \hspace{8cm} \lesssim \Big(  \log \frac{1}{\varepsilon} \Big)^{|\mathcal{R}|} \, \|D\|_{X^{s,b}}^{2 |\mathcal{D}|} \,\prod_{i \in \mathcal{R}} \frac{1}{ \langle M_i \rangle^{2\alpha - \varepsilon_1}}.
    \end{align}
\end{prop}
\begin{proof}
Let $\varepsilon_1 > 0$ small enough. From Lemma \ref{lemma:Multilinear_Estimates}, we have
\begin{align}\label{eq:DNlarge-0}
    & \sum_{\substack{N_{(1)}\geq \ldots\geq N_{(5)} \\ N_i \geq M_i}}  \|P_N \mathcal{Q}(P_{N_1} u_1 ,  \ldots, P_{N_5} u_{5})\|_{X^{0,b'-1}}
    \lesssim \sum_{\substack{N_{(1)}\geq \ldots\geq N_{(5)} \\ N_i \geq M_i}} N_{(2)}^{\varepsilon_1} \|P_{N_{(1)}} D \|_{X^{0,b}} \prod_{j \neq 1} \|P_{N_{(j)} } u_{(j)} \|_{L^8_{x,t}}, \qquad 
\end{align}
and we separate the sum into the cases $N_{(1)} \simeq N$ and $N_{(1)} \gg N$. 

\hfill \break
\textbf{Case 1.} 
When $N_{(1)} \simeq N$, by Cauchy-Schwarz in $N_{(1)}$ we have that
\begin{align}\label{eq:DNlarge-1}
    N^{2s} & \bigg( \sum_{\substack{ N \simeq N_{(1)} \geq \ldots \geq N_{(5)} \\ N_i \geq M_i }} N_{(2)}^{\varepsilon_1} \|P_{N_{(1)}} D \|_{X^{0,b}} \prod_{j \neq 1} \|P_{N_{(j)}} u_{(j)} \|_{L^8_{x,t}} \bigg)^2 \nonumber \\
    & \lesssim 
    \bigg(\sum_{N_{(1)} \simeq N} \frac{N^{2s}}{N_{(1)}^{2s}} \|P_{N_{(1)}} D\|_{X^{s,b}}^2\bigg) \, 
    \bigg( \sum_{N_{(1)} \simeq N} \Big| \sum_{\substack{N_{(2)}\geq \ldots\geq N_{(5)} \\ N_i \geq M_i}} N_{(2)}^{\varepsilon_1} \prod_{j \neq 1} \|P_{N_{(j)}} u_{(j)} \|_{L^8_{x,t}} \Big|^2 \bigg). 
\end{align}
Using Cauchy-Schwarz in a similar way in the remaining frequencies, by Lemma \ref{lemma:L8bound} we get
\begin{align}
    &\sum_{N_{(1)} \simeq N} \Big| \sum_{\substack{N_{(2)}\geq \ldots\geq N_{(5)} \\ N_i \geq M_i}} N_{(2)}^{\varepsilon_1} \prod_{j \neq 1} \|P_{N_{(j)}} u_{(j)} \|_{L^8_{x,t}} \Big|^2 \\
    & \qquad \qquad \lesssim 
    \prod_{j \neq 1} \Big( \sum_{\substack{N_{(j)} \\ N_i \geq M_i }} N_{(j)}^{4\varepsilon_1} \|P_{N_{(j)}}u_{(j)}\|_{L^8_{x,t}}^2 \Big) \,  \Big( \sum_{N_{(1)} \simeq N}\sum_{N_{(2)} \geq \ldots \geq N_{(5)}} \frac{1}{N_{(2)}^{2\varepsilon_1} N_{(3)}^{2\varepsilon_1} N_{(4)}^{2\varepsilon_1} N_{(5)}^{2\varepsilon_1}} \Big) \\
    & \qquad \qquad \lesssim \|D\|_{X^{s,b}}^{2(|\mathcal{D}| - 1)} 
    \prod_{i \in \mathcal R} \Big(\sum_{N_i \geq M_i } \frac{\log (1/\varepsilon)}{N_i^{2\alpha - 5\varepsilon_1}} \Big),
\end{align}
with probability $\geq 1 - \varepsilon$, since for each $N$, the number of $N_{(1)} \simeq N$ is finite and independent of $N$. 
Therefore, from \eqref{eq:DNlarge-1} we get
\begin{align}
    & \sum_N N^{2s} \Big( \sum_{\substack{N \simeq N_{(1)} \geq \ldots \geq N_{(5)} \\ N_i \geq M_i }} N_{(2)}^{\varepsilon_1} \|P_{N_{(1)}} D \|_{X^{0,b}} \prod_{j \neq 1} \|P_{N_{(j)}} u_{(j)}\|_{L^8_{x,t}} \Big)^2 \\
    & \qquad \qquad \lesssim \Big(  \log \frac{1}{\varepsilon} \Big)^{|\mathcal{R}|} \|D\|_{X^{s,b}}^{2(|\mathcal{D}| - 1)} \prod_{i \in \mathcal{R}} \frac{1}{ \langle M_i \rangle^{2\alpha - 5\varepsilon_1}}  \sum_{N_{(1)}} \Big( \|P_{N_{(1)}} D\|_{X^{s,b}}^2  \sum_{N:N \simeq N_{(1)}} 1 \Big) \\
    & \qquad \qquad \lesssim \Big(  \log \frac{1}{\varepsilon} \Big)^{|\mathcal{R}|} \|D\|_{X^{s,b}}^{2|\mathcal{D}|} \prod_{i \in \mathcal{R}} \frac{1}{\langle M_i \rangle^{2\alpha - 5\varepsilon_1}}.
\end{align}

\textbf{Case 2.}
When $N_{(1)} \gg N$, we have $N_{(2)} \simeq N_{(1)}$. By Cauchy-Schwarz in $N_{(1)}$,  we bound \eqref{eq:DNlarge-0} by
\begin{align}\label{eq:DNlarge-2}
    & \bigg( \sum_{\substack{N_{(1)} \geq \ldots \geq N_{(5)} \\ N_{(1)} \gg N, \, \, N_i \geq M_i}} N_{(2)}^{\varepsilon_1} \|P_{N_{(1)}} D \|_{X^{0,b}} \prod_{i \neq 1} \|P_{N_{(i)} u_{(i)}} \|_{L^8_{x,t}} \bigg)^2 \nonumber  \\
    & \hspace{1cm}
    \lesssim \Big( \sum_{N_{(1)}} \|P_{N_{(1)}} D\|_{X^{s,b}}^2 \Big) \,
    \Big( \sum_{N_{(1)} \gg N} \frac{1}{N_{(1)}^{2s}}  \Big| \sum_{\substack{  N_{(2)} \geq \ldots \geq N_{(5)} \\ N_{(2)} \simeq N_{(1)}, \, \,  N_i \geq M_i  }} N_{(2)}^{\varepsilon_1} \prod_{i \neq 1} \|P_{N_{(i)}} u_{(i)} \|_{L^8_{x,t}} \Big|^2 \Big).
\end{align}
Again, applying Cauchy-Schwarz in the rest of $N_{(j)}$, by Lemma \ref{lemma:L8bound} we get
\begin{align}\label{eq:After_Multilineal_Case2}
    & \sum_{N_{(1)} \gg N} 
    \frac{1}{N_{(1)}^{2s}}  \Big| \sum_{\substack{ N_{(2)} \geq \ldots \geq N_{(5)} \\ N_{(2)} \simeq N_{(1)}, \, \, N_i \geq M_i}} N_{(2)}^{\varepsilon_1} \prod_{i \neq 1} \|P_{N_{(i)}} u_{(i)} \|_{L^8_{x,t}} \Big|^2 \nonumber \\
    & \hspace{2cm}
    \lesssim \prod_{j \neq 1} \Big( \sum_{N_{(j)} \, : \, N_i \geq M_i} N_{(j)}^{4\varepsilon_1} \lVert P_{N_{(j)}} u_{(j)} \rVert_{L_{x,t}^8}^2 \Big) \, \cdot \, \sum_{N_{(1)} \gg N}   \sum_{\substack{ N_{(2)} \geq \ldots \geq N_{(5)} \\ N_{(2)} \simeq N_{(1)}}}\frac{1}{N_{(1)}^{2s} (N_{(2)} \ldots N_{(5)} )^{2\varepsilon_1}} \nonumber
    \\
    & \hspace{2cm} \lesssim \|D\|_{X^{s,b}}^{2(|\mathcal{D}| - 1)} \prod_{j\in \mathcal R} \frac{  \log (1/\varepsilon)}{\langle M_j \rangle^{2\alpha - 5\varepsilon_1}}  \sum_{ N_{(1)} \gg N}  \frac{1}{N_{(1)}^{2s + 2\varepsilon_1}}, 
\end{align}
with probability $\geq 1 - \varepsilon$. 
From \eqref{eq:DNlarge-2} and \eqref{eq:After_Multilineal_Case2} we get
\begin{align}
    & \sum_N N^{2s} \Big( 
    \sum_{\substack{N_{(1)} \geq \ldots \geq N_{(5)}  \\  N_{(1)} \gg N, \, \, N_i \geq M_i} } 
    N_{(2)}^{\varepsilon} \|P_{N_{(1)}} D \|_{X^{0,b}} \prod_{i \neq 1} \|P_{N_{(i)} u_{(i)}} \|_{L^8_{x,t}} \Big)^2 \\
    & \hspace{3cm}
    \lesssim \|D\|_{X^{s,b}}^{2|\mathcal{D}|} \prod_{i \in \mathcal{R}} \frac{\log (1/\varepsilon)}{ \langle M_i \rangle^{2 \alpha - 5\varepsilon_1}} \, 
    \sum_N N^{2s} \sum_{ N_{(1)} \gg N}  \frac{1}{N_{(1)}^{2s + 2\varepsilon_1}},
\end{align}
which is what we wanted
since $\sum_{N \ll N_{(1)}} N^{2s} \lesssim N_{(1)}^{2s}$. 
\end{proof}

We now prove the multilinear estimates in Lemma~\ref{lemma:Multilinear_Estimates}.
\begin{proof}[Proof of Lemma~\ref{lemma:Multilinear_Estimates}]
    We first prove \eqref{eq:Multilineal_Main}. By duality, write 
    \begin{equation}\label{eq:dual-norm}
        \bigg\lVert P_N \bigg(
        \prod_{i=1}^5 (P_{N_{i}} u_i)^{\iota_i} \bigg)
        \bigg\rVert_{X^{s, b'-1}} 
        = \sup_{\lVert h \rVert_{X^{-s, -b'+1}} = 1} 
        \int \overline{h} \, P_N \bigg( \prod_{i=1}^5 (P_{N_{i}} u_i)^{\iota_i} \bigg).
    \end{equation}  
    Since the product $\prod_{i=1}^5 (P_{N_{i}} u_i)^{\iota_i}$ is Fourier supported in $|n| \lesssim N_{(1)}$, 
    we have $N \simeq |n| \lesssim N_{(1)}$. 
    On the other hand,
    divide the ball $|n| \leq N_{(1)}$ into finitely overlapping cubes $Q$ of side length $N_{(2)}$,  
    in such a way that using Parseval and H\"older we get
    \begin{align}
        \Big| \int \overline{P_N h} \, \cdot \, \prod_{i=1}^5 (P_{N_i} u_i)^{\iota_i} \Big|
        & = \Big| \sum_{Q, Q_1} 
        \int \overline{P_Q P_N h} \cdot 
        (P_{Q_1}P_{N_{(1)}} u_{(1)})^{\iota_{(1)}}
        \cdot \prod_{i=2}^5 (P_{N_{(i)}} u_{(i)})^{\iota_{(i)}} \Big| \\
        & \leq 
        \sum_{Q, Q_1}
        \lVert P_Q P_N h \rVert_{L^4_{x,t}} \,
        \big\lVert  P_{Q_1} P_{N_{(1)}} u_{(1)} \big\rVert_{L^4_{x,t}} \, 
        \prod_{i=2}^5 \big\lVert  P_{N_{(i)}} u_{(i)} \big\rVert_{L^8_{x,t}},
    \end{align} 
    where the sum in $Q,Q_1$ is restricted by the condition $\dist(Q, Q_1) \lesssim N_{(2)}$ due to $|n-n_1| \lesssim N_{(2)}$. 
    By Strichartz estimates, the transference principle from Lemma \ref{lemma:TransferLemma}
    and Lemma \ref{lemma:st-int}, we have 
    \begin{align}\label{eq:Strichartz_Localized}
        \|P_Q F\|_{L^4_{x,t}}
        \lesssim 
        N_{(2)}^{2\varepsilon_1} \, \| P_Q F \|_{X^{0, \frac{1}{2} - \varepsilon_1}},
        \qquad \forall Q \text{ of side length } N_{(2)}, 
    \end{align}
    which implies that
    \begin{align}\label{eq:Multilinear_Proof_Xsb}
        \lVert  P_Q P_N  h \rVert_{L^4_{x,t}}\, 
        \big\lVert  P_{Q_1} P_{N_{(1)}} u_{(1)} \big\rVert_{L^4_{x,t}}
        & \lesssim 
        N_{(2)}^{4\varepsilon_1} \, 
        \lVert P_Q P_N  h \rVert_{X^{0, \frac{1}{2} - \varepsilon_1}} \, 
        \big\lVert  P_{Q_1} P_{N_{(1)}} u_{(1)} \big\rVert_{X^{0, \frac{1}{2} - \varepsilon_1}}. 
    \end{align}
Hence, by Cauchy-Schwartz and the fact that $N \lesssim N_{(1)}$, we get  
    \begin{align}\label{eq:Multilinear_Proof_Box_Localization}
        & \sum_{Q, Q_1} \lVert P_Q P_N  h \rVert_{X^{0, \frac{1}{2} - \varepsilon_1}} \, 
        \big\lVert  P_{Q_1} P_{N_{(1)}} u_{(1)} \big\rVert_{X^{0, \frac{1}{2} - \varepsilon_1}} \nonumber \\
        & \hspace{3cm}
        \leq 
        \Big( \sum_{Q, Q_1 } \lVert P_Q P_N  h \rVert_{X^{0, \frac{1}{2} - \varepsilon_1}}^2
        \Big)^{1/2}
        \, 
        \Big( \sum_{Q, Q_1} \lVert P_{Q_1} P_{N_{(1)}}  u_{(1)} \rVert_{X^{0, \frac{1}{2} - \varepsilon_1}}^2
        \Big)^{1/2} \nonumber \\
        & \hspace{3cm}
        \lesssim 
        \lVert P_N  h \rVert_{X^{0, \frac{1}{2} - \varepsilon_1}} \, 
        \lVert P_{N_{(1)}}  u_{(1)} \rVert_{X^{0, \frac{1}{2} - \varepsilon_1}} \nonumber \\
        & \hspace{3cm}
        \lesssim 
        \lVert P_N  h \rVert_{X^{-s, \frac{1}{2} - \varepsilon_1}} \, 
        \lVert P_{N_{(1)}}  u_{(1)} \rVert_{X^{s, \frac{1}{2} - \varepsilon_1}}, 
    \end{align}
    since for every fixed $Q$, the number of $Q_1$ in the sum is uniformly bounded. 
    Therefore, given that $-b' + 1 = 1/2 - \beta'$, for $\beta'$ small enough we have $\beta' \leq \varepsilon_1$, so
    \begin{align}
        \Big|\int \overline{P_N h} \, \cdot \, \prod_{i=1}^5 
        (P_{N_i} u_i)^{\iota_i} \Big|
        \lesssim
        N_{(2)}^{4\varepsilon_1} \,  
        \lVert P_{N_{(1)}}  u_{(1)} \rVert_{X^{s, \frac{1}{2} - \varepsilon_1}} \, 
        \prod_{i=2}^5 \big\lVert  P_{N_{(i)}} u_{(i)} \big\rVert_{L^8_{x,t}}.
    \end{align}
To prove \eqref{eq:Multilineal_Secondary}, we adapt this duality argument. Notice first that $N \simeq N_1$. 
We separate cases:
\hfill \break

\textbf{Case 1.} 
If $N_{(1)} = N_1$, we proceed as above. By duality with $\lVert h \rVert_{X^{-s, -b'+1}} \leq 1$, the box localization with cubes $Q$ of size $N_{(2)}$, 
Parseval and Cauchy-Schwarz in $t$, we bound
\begin{align}\label{eq:int2-Multilinear_Proof_Box_Localization}
    \Big| \int_{x,t} \overline{P_N h} \,  P_{N_1} u_1 \, \bigg(\int_x \prod_{i=2}^5 (P_{N_i} u_i)^{\iota_i} \bigg) \, dx\, dt \Big|
    & = \Big| \sum_{Q} \int_t \bigg( \int_x \overline{P_Q P_N h} \cdot  P_Q P_{N_1} u_1 \bigg) \,
    \bigg(\int_x \prod_{i=2}^5 (P_{N_i} u_i)^{\iota_i} \bigg) \, dt \Big| \nonumber
    \\
    & \leq \sum_{Q} 
    \bigg\lVert \int_x \overline{P_Q P_N h} \cdot P_Q P_{N_1} u_1\bigg\rVert_{L^2_t} 
    \, \Big\lVert \int_x \prod_{i=2}^5 (P_{N_i} u_i)^{\iota_i} \Big\rVert_{L^2_t} \nonumber \\
    & \leq 
    \sum_{Q} \big\lVert \overline{P_Q P_N h} \cdot P_Q P_{N_1} u_1 \big\rVert_{L^2_{x,t}} 
    \, \Big\lVert \prod_{i=2}^5 (P_{N_i} u_i)^{\iota_i} \Big\rVert_{L^2_{x,t}}, 
\end{align}
where we trivially bounded the $L^p_x$ norms. 
Now, using H\"older in $(x,t)$, and then \eqref{eq:Multilinear_Proof_Xsb} and arguing like in \eqref{eq:Multilinear_Proof_Box_Localization} because $N \simeq N_1$, we get 
\begin{align}
    \eqref{eq:int2-Multilinear_Proof_Box_Localization} & \leq 
    \sum_{Q} 
    \lVert P_Q P_N h  \rVert_{L^4_{x,t}} \, 
    \lVert  P_Q P_{N_1} u_1 \rVert_{L^4_{x,t}} 
    \, \prod_{i=2}^5 
    \big\lVert  P_{N_i} u_i \big\rVert_{L^8_{x,t}} \\
    & \leq 
    N_{(2)}^{4\varepsilon_1} \, \sum_{Q} 
    \lVert P_Q P_N h  \rVert_{X^{0,\frac{1}{2}-\varepsilon_1}} \, 
    \lVert  P_Q P_{N_1} u_1 \rVert_{X^{0,\frac{1}{2}-\varepsilon_1}} 
    \, \prod_{i=2}^5 
    \big\lVert  P_{N_i} u_i \big\rVert_{L^8_{x,t}} \\
    & \lesssim 
    N_{(2)}^{4\varepsilon_1}  \,
    \lVert  P_N h \rVert_{X^{-s, \frac{1}{2} - \varepsilon_1}} \, 
    \lVert  P_{N_1} u_1 \rVert_{X^{s, \frac{1}{2} - \varepsilon_1}} \, 
    \prod_{i=2}^5 
    \big\lVert  P_{N_i} u_i \big\rVert_{L^8_{x,t}}.
\end{align}

\textbf{Case 2. }
If $N_{(1)} = N_i$ for some $i \neq 1$, 
    assume without loss of generality that $N_{(1)} = N_2 \geq N_3 \geq N_4 \geq N_5$. 
    Observe that since $n_2 - n_3 + n_4 - n_5 = 0$, we have $N_2 \simeq N_3$, 
    and also $|n_2 - n_3| \lesssim N_4$. 
    By duality, and box localizing in cubes $Q_2$ and $Q_3$ of size $N_4$, 
    we bound
    \begin{align}\label{eq:N_2Largest}
        & \Big|\int_{x,t} \overline{P_{N_1}P_N h} \,  P_{N_1} u_1 \, \bigg(\int \prod_{i=2}^5 (P_{N_i} u_i)^{\iota_i} \bigg) \, dx\, dt \Big| \nonumber \\
        & \hspace{1cm} = 
        \Big|\int_t \bigg( \int_x \overline{P_{N_1}P_N h} \cdot  P_{N_1} u_1 \bigg)\, \bigg(\int_x \prod_{i=2}^5 (P_{N_i} u_i)^{\iota_i} \bigg) \, dt \Big|   \\
        & \hspace{1cm} = 
        \Big| \int_t \bigg( \int_x \overline{P_{N_1}P_N h} \cdot  P_{N_1} u_1 \bigg)\, \bigg( \sum_{Q_2, Q_3}\int_x \overline{P_{Q_2}P_{N_2} u_2} \cdot P_{Q_3}P_{N_3} u_3  \cdot \overline{P_{N_4} u_4} \cdot P_{N_5} u_5\bigg) \, dt \Big|  \nonumber \\
        & \hspace{1cm} \leq 
        \sum_{Q_2, Q_3}
        \int_t \lVert P_{N_1}P_N h \rVert_{L_x^2}  \, \lVert P_{N_1} u \rVert_{L_x^2}  
        \lVert P_{Q_2} P_{N_2} u_2 \rVert_{L_x^4} \, 
        \lVert P_{Q_3}P_{N_3} u_3 \rVert_{L_x^2} \, 
        \lVert P_{N_4} u_4 \rVert_{L_x^8} \, 
        \lVert P_{N_5} u_5 \rVert_{L_x^8} \, dt. \nonumber
    \end{align}
    By Cauchy-Schwarz and by the condition $\dist(Q_2, Q_3) \lesssim N_4$ we have
    \begin{align}
        \sum_{Q_2, Q_3}
        \lVert P_{Q_2} P_{N_2} u_2 \rVert_{L_x^4} \, 
        \lVert P_{Q_3}P_{N_3} u_3 \rVert_{L_x^2} 
        & \leq 
        \Big( \sum_{Q_2, Q_3}
        \lVert P_{Q_2} P_{N_2} u_2 \rVert_{L_x^4}^2 \Big)^{1/2} \, 
        \Big( \sum_{Q_2, Q_3}
        \lVert P_{Q_3}P_{N_3} u_3 \rVert_{L_x^2}^2 \Big)^{1/2} \\
        & \lesssim 
        \Big( \sum_{Q_2}
        \lVert P_{Q_2} P_{N_2} u_2 \rVert_{L_x^4}^2 \Big)^{1/2} \, 
        \lVert P_{N_3} u_3 \rVert_{L_x^2},
    \end{align}
    and thus, bounding $L^p$ norms trivially and using H\"older, we get
    \begin{align}\label{eq:Multilinear_Almost_Last_Step}
        \eqref{eq:N_2Largest} 
        & \lesssim 
        \int_t \lVert P_{N_1}P_N h \rVert_{L_x^2} \, 
        \prod_{i \neq 2} \lVert P_{N_i} u_i \rVert_{L_x^8} \, 
        \Big( \sum_{Q_2}
        \lVert P_{Q_2} P_{N_2} u_2 \rVert_{L_x^4}^2 \Big)^{1/2} \, dt \nonumber \\
        & \leq 
        \lVert P_{N_1}P_N h \rVert_{L_{x,t}^4} \, 
        \bigg\lVert \Big( \sum_{Q_2}
        \lVert P_{Q_2} P_{N_2} u_2 \rVert_{L_x^4}^2 \Big)^{1/2} \bigg\rVert_{L_t^4} \, 
        \prod_{i \neq 2} \lVert P_{N_i} u_i \rVert_{L_{x,t}^8}.
    \end{align}
    Now, 
    \begin{align}
        \bigg\lVert \Big( \sum_{Q_2}
        \lVert P_{Q_2} P_{N_2} u_2 \rVert_{L_x^4}^2 \Big)^{1/2} \bigg\rVert_{L_t^4}^4
        & = \sum_{Q_2, Q_2'} 
        \int_t 
        \lVert P_{Q_2} P_{N_2} u_2 \rVert_{L_x^4}^2 \,
        \lVert P_{Q_2'} P_{N_2} u_2 \rVert_{L_x^4}^2 \, dt \\
        & \leq \sum_{Q_2, Q_2'}
        \lVert P_{Q_2} P_{N_2} u_2 \rVert_{L_{x,t}^4}^2 \, 
        \lVert P_{Q_2'} P_{N_2} u_2 \rVert_{L_{x,t}^4}^2 \\
        & \lesssim N_4^{8\varepsilon_1} 
        \Big( \sum_{Q_2}
        \lVert P_{Q_2} P_{N_2} u_2 \rVert_{X^{0,\frac{1}{2}-\varepsilon_1}}^2 
        \Big)^2 \\
        & \simeq 
        N_4^{8\varepsilon_1}  
        \lVert P_{N_2} u_2 \rVert_{X^{0,\frac{1}{2}-\varepsilon_1}}^4, 
    \end{align}
    where we used \eqref{eq:Strichartz_Localized}. Hence, since $N_1 \leq N_2$ and $N_1, N_4 \leq N_{(2)}$, by using \eqref{eq:Strichartz_Localized} again we get
    \begin{align}
        \eqref{eq:Multilinear_Almost_Last_Step}
        & \lesssim 
        N_1^{2\varepsilon_1} \, N_4^{2\varepsilon_1} 
        \lVert P_{N_1}P_N h \rVert_{X^{0,\frac{1}{2}-\varepsilon_1}} \, 
        \lVert P_{N_2} u_2 \rVert_{X^{0,\frac{1}{2}-\varepsilon_1}}\, 
        \prod_{i \neq 2} \lVert P_{N_i} u_i \rVert_{L_{x,t}^8} \\
        & \leq 
        N_{(2)}^{4\varepsilon_1} \, 
        \lVert P_{N_2} u_2 \rVert_{X^{s,\frac{1}{2}-\varepsilon_1}}\, 
        \prod_{i \neq 2} \lVert P_{N_i} u_i \rVert_{L_{x,t}^8}.\qedhere
    \end{align}
\end{proof}

\subsection{Reduction to the paraboloid} \label{subsec:red-par} 
Thanks to Proposition~\ref{prop:Qsmoothing-mest}, we can assume that all highest frequency components are $R$. 
We first deal with the case in which the Fourier support is far from the paraboloid $\tau + |n|^2 = 0$. 
For that, we have the following lemma, adapted from \cite{CLS2020}. 
\begin{lem}\label{lemma:reduction-2}
Suppose that $u_{(1)} = R$.
Let $\beta > 0$ and $c \geq 4\beta/(1 - 4\beta)$. 
Let $\varepsilon > 0$. 
Then, the following is true with probability $\geq 1 - \varepsilon$. 
For every $N \in 2^{\mathbb N}$, and for every $\delta_1 > 0$ and $\varepsilon_1 > 0$,
\begin{align}
\Big\| 
P_{\left\{ \langle \tau + |n|^2 \rangle > N^{1+c} \right\}} 
P_N  
\Big( 
\prod_{j=1}^5
(P_{N_j} u_j)^{\iota_j}
\Big) 
\Big\|_{X^{0,-\frac{1}{2}+\beta}}
\lesssim 
\frac{(\log (1/\varepsilon))^{1/2}}{N^{\frac{1}{2}+\frac{c}{4}} \, N_{(1)}^{\alpha - \varepsilon_1}} 
\prod_{j=2}^5 \lVert P_{N_{(j)}} u_{(j)} \rVert_{L^{8+\delta_1}}
\end{align}
and 
\begin{align}
    \bigg\| 
    P_{\left\{ \langle \tau + |n|^2 \rangle > N^{1+c} \right\}} 
    P_N  
    \Big( 
    P_{N_1} u_1 \, 
    \int 
    \prod_{j=2}^5
    (P_{N_j} u_j)^{\iota_j}
    \Big) 
    \bigg\|_{X^{0,-\frac{1}{2}+\beta}}
    \lesssim 
    \frac{(\log 1/\varepsilon)^{1/2}}{N^{\frac12 + \frac{c}{4}} \, N_{(1)}^{\alpha - \varepsilon_1}} \, 
    \prod_{j=2}^5
    \lVert P_{N_{(j)}} u_{(j)}
    \rVert_{L^{8+\delta_1}_{x,t}}.
\end{align}
\end{lem}

This lemma allows us to bound \eqref{eq:decomp-lp} when frequencies $(n,\tau)$ are far from the paraboloid.
We now prove this and postpone the proof of Lemma~\ref{lemma:reduction-2} to the end of the section.

\begin{prop}\label{prop:Qsmoothing-par}
    Let $s > 1/2$ and $c > 4(s - 1/2)$.
    Suppose $u_{(1)} = R$. 
    Then, for any $\varepsilon > 0$,
    the following is true with probability $\geq 1 - \varepsilon$. 
    For any $\varepsilon_1 > 0$, 
    \begin{align*}
        \sum_N
    N^{2s} 
    &\Big( 
    \sum_{\substack{N_{(1)} \geq \ldots\geq N_{(5)} \\ N_i \geq M_i  }} \lVert P_N P_{\langle \tau + |n|^2 \rangle > N^{1 + c}} \mathcal{Q}(P_{N_{1}}u_{1}, P_{N_{2}}u_{2}, P_{N_{3}}u_{3}, P_{N_{4}}u_{4}, P_{N_{5}}u_{5}) \rVert_{X^{0,b'-1}_\delta} \Big)^2 \\
    & \hspace{7cm} \lesssim \Big(  \log \frac{1}{\varepsilon} \Big)^{|\mathcal{R}|} \|D\|_{X^{s,b}}^{2|\mathcal{D}|} \prod_{i \in \mathcal{R}} \frac{1}{\langle M_i \rangle^{2\alpha - \varepsilon_1}}.
    \end{align*}
\end{prop}

\begin{proof}[Proof of Proposition~\ref{prop:Qsmoothing-par}]
By Lemma \ref{lemma:reduction-2} and \eqref{eq:L8bound-1} in Lemma~\ref{lemma:L8bound}, with probability $\geq 1 - \varepsilon$ we have
\begin{align}\label{eq:red-par-int1}
    \sum_N
    &N^{2s} 
    \Big( 
    \sum_{\substack{N_{(1)} \geq \ldots\geq N_{(5)} \\ N_i \geq M_i }} \lVert P_N P_{\langle \tau + |n|^2 \rangle > N^{1 + c}} \mathcal{Q}(P_{N_{1}}u_{1}, P_{N_{2}}u_{2}, P_{N_{3}}u_{3}, P_{N_{4}}u_{4}, P_{N_{5}}u_{5}) \rVert_{X^{0,b'-1}_\delta} \Big)^2 \nonumber \\
    & \lesssim \sum_N N^{2s} \Big( \sum_{\substack{N_{(1)} \geq \ldots\geq N_{(5)} \\ N_i \geq M_i }} \frac{(  \log ( 1/\varepsilon) )^{1/2}}{N^{\frac{1}{2} + \frac{c}{4}} N_{(1)}^{\alpha - \varepsilon_1}} \prod_{j \neq 1} \|P_{N_{(j)}} u_{(j)} \|_{L^{8 + \delta_1}} \Big)^2 \nonumber \\
    & \lesssim 
    \Big(\log \frac{1}{\varepsilon}\Big)^{|\mathcal{R}|} \sum_N \frac{N^{2s}}{N^{1 + \frac{c}{2}}} \bigg( \sum_{N_{(1)}}  \frac{1}{N_{(1)}^{\alpha - \varepsilon_1}} \sum_{\substack{N_{(2)} \geq \ldots\geq N_{(5)} \\ N_i \geq M_i }}  \prod_{i \in \mathcal{D}} \|P_{N_{i}} D \|_{X^{\frac{1}{2} + \delta_1,b}} \prod_{j \in \mathcal{R} \setminus \{(1)\}} \frac{1}{N_{j}^{\alpha - \varepsilon_1}} \bigg)^2,
\end{align}
for any $\delta_1 > 0$  and $\varepsilon_1 > 0$. 
The inner sum in $N_{(2)}, \ldots, N_{(5)}$ is bounded by
\begin{align}
    \prod_{j \in \mathcal R \setminus \{(1)\}} \bigg( \sum_{ N_j \geq M_j }  \frac{1}{N_{j}^{\alpha - \varepsilon_1}} \bigg)
    \, \cdot \, 
    \prod_{i \in \mathcal D} \bigg( \sum_{ N_i \leq N_{(1)} } \|P_{N_{i}} D \|_{X^{\frac{1}{2} + \delta_1,b}} \bigg)
    \lesssim 
    \prod_{j \in \mathcal R \setminus \{(1)\}} \frac{1}{\langle M_j \rangle^{\alpha - \varepsilon_1}} \, N_{(1)}^{\varepsilon_1} \|D\|_{X^{s,b}}^{|\mathcal D|},
\end{align}
since by Cauchy-Schwarz, if $\delta_1 >0$ is small enough so that $1/2 + \delta_1 \leq s$, we have 
\begin{align*}
    \Big( \sum_{N_i \leq N_{(1)} } \|P_{N_i} D \|_{X^{\frac{1}{2} + \delta_1 , b}} \Big)^2 
    & \leq \sum_{N_{i}} \|P_{N_{i}} u_{i}\|_{X^{\frac{1}{2} + \delta_1 , b}}^2 \, \cdot \,  \sum_{N_i \leq N_{(1)}} 1
    \, \lesssim \,   N_{(1)}^{\varepsilon_1} \,  \|D\|_{X^{s,b}}^2.
\end{align*} 
Therefore, choosing $\varepsilon_1$ such that $ 2\varepsilon_1  < \alpha$ and $c > 4 (s  - 1/2)$, we get the bound
\begin{align}
    \eqref{eq:red-par-int1}
    &\lesssim \frac{ \big(  \log 1/\varepsilon \big)^{|\mathcal{R}|} \, \|D\|_{X^{s,b}}^{2|\mathcal{D}|}}{ \prod_{j \in \mathcal{R} \setminus \{(1)\}}  \langle M_j \rangle^{2\alpha - 2\varepsilon_1}} \, 
    \sum_N \frac{1}{N^{1 - 2s + \frac{c}{2}}} \bigg( \sum_{N_{(1)}} \frac{1}{N_{(1)}^{\alpha  - 2\varepsilon_1}} \bigg)^2
    \lesssim  \frac{\big(  \log 1/\varepsilon \big)^{|\mathcal{R}|} \, \|D\|_{X^{s,b}}^{2|\mathcal{D}|}}{ \prod_{j \in \mathcal{R}} \langle M_j \rangle^{2\alpha  - 4\varepsilon_1}}.
\end{align} 
Since $\varepsilon_1$ is arbitrarily small we get the result. 
\end{proof}

To conclude this section, we give the proof of Lemma~\ref{lemma:reduction-2}. 
\begin{proof}[Proof of Lemma~\ref{lemma:reduction-2}]
For a general function $F$ we have 
\begin{align}
    \lVert 
    P_{\left\{ \langle \tau + |n|^2 \rangle > N^{1+c} \right\}} \, 
    P_N \, 
    F
    \rVert_{X^{0,-\frac12 + \beta}}^2
     = 
    \sum_{|n| \simeq N} \int_{\mathbb{R}} \frac{\mathbbm 1_{\langle \tau + |n|^2 \rangle > N^{1+c}}}{\langle \tau + |n|^2 \rangle^{1 - 2\beta}} |
    \widetilde F(\tau,n)|^2 d\tau 
    \lesssim \frac{1}{N^{1 + \frac{c}{2}}}\, 
    \lVert 
    P_N \, F
    \rVert_{L^2_{x,t}}^2,
\end{align}
where we used that $c \geq 4\beta/(1 - 4\beta)$ implies $(1+c)(1 - 2\beta) \geq 1 + c/2$. Therefore, taking $F = P_{N_1}u_1 \, 
    \overline{P_{N_2} u_2} \, 
    P_{N_3}u_3 \, 
    \overline{P_{N_4}u_4} \,
    P_{N_5} u_5$, 
by H\"older's inequality we get
\begin{align}
    \big\|
    P_N 
    \big(
    P_{N_1}u_1 \, 
    \overline{P_{N_2} u_2} \, 
    P_{N_3}u_3 \, 
    \overline{P_{N_4}u_4} \, 
    P_{N_5}u_5 \, 
    \big) \big\|_{L^2_{x,t}} 
    & \leq 
    \| P_{N_{(1)}} R \|_{L^q_{x,t}} 
    \prod_{j=2}^5 
    \|
    P_{N_{(j)}} u_{(j)} 
    \|_{L^{8 + \delta_1}_{x,t}} \\
    & \lesssim 
    \frac{( \log 1/\varepsilon )^{1/2}}{N_{(1)}^{\alpha - \varepsilon_1}} \, 
    \prod_{j=2}^5 
    \|
    P_{N_{(j)}} u_{(j)} 
    \|_{L^{8 + \delta_1}_{x,t}}
\end{align}
for any $\delta_1 > 0$ with $q = 2(8+\delta_1)/\delta_1$, 
and for any $\varepsilon_1 > 0$, 
where the last inequality follows from Lemma~\ref{lemma:L8bound}
with probability $\geq 1 - \varepsilon$.

To prove the second inequality, we first write
\begin{align}\label{eq:red-par-int2}
    \big\lVert 
    P_N P_{N_1} u_1
    \int \left(
    \overline{P_{N_2} u} \, 
    P_{N_3}u \, 
    \overline{P_{N_4}u} \, 
    P_{N_5}u \right) 
    \big\|_{L^2_{x,t}}^2 \nonumber 
    & =
    \int 
    \lVert P_N P_{N_1} u_1 \rVert_{L^2_x}^2
    \left|
    \int 
    \overline{P_{N_2} u_2} 
    P_{N_3}u_3 
    \overline{P_{N_4}u_4} 
    P_{N_5}u_5 \, dx \right|^2
    dt  \\
    & \lesssim 
    \int 
    \lVert P_N P_{N_1} u_1 \rVert_{L^2_x}^2
    \prod_{j=2}^5
    \lVert P_{N_j} u_j \rVert_{L^4_x}^2 \, 
    dt. 
\end{align}
We consider two different cases.  
If $N_{(1)} = N_1$, 
let $q = (8+\delta_1)/\delta_1$ and by H\"older we bound
\begin{align}
    \int 
    \lVert P_N P_{N_{(1)}} R \rVert_{L^2_x}^2
    \prod_{j=2}^5
    \lVert P_{N_j} u_j \rVert_{L^4_x}^2 \, 
    dt 
    & \leq
    \big\lVert \lVert P_N P_{N_{(1)}} R \rVert_{L^2_x}^2 
    \big\rVert_{L^q_t} \, \cdot \, 
    \prod_{j=2}^5
    \big\lVert \lVert P_{N_j} u_j \rVert_{L^4_x}^2 
    \big\rVert_{L^{4+\frac{\delta_1}{2}}_t} \\
    & \lesssim 
    \lVert P_N P_{N_{(1)}} R 
    \rVert_{L^{2q}_{x,t}}^2 \, 
    \prod_{j=2}^5
    \lVert P_{N_j} u_j
    \rVert_{L^{8+\delta_1}_{x,t}}^2 \\
    & \lesssim 
    \frac{\log(1/\varepsilon)}{N_{(1)}^{2\alpha - \varepsilon_1}} \, 
    \prod_{j=2}^5
    \lVert P_{N_{(j)}} u_{(j)}
    \rVert_{L^{8+\delta_1}_{x,t}}^2,
\end{align}
with probability $\geq 1 - \varepsilon$, where we used Lemma~\ref{lemma:L8bound}. 
On the other hand, if $N_{(1)} \neq N_1$, 
then assume $N_{(1)} = N_2$ without loss of generality. By H\"older with the same $q$ as above and Lemma~\ref{lemma:L8bound}, we get
\begin{align}
    \eqref{eq:red-par-int2} = \int 
    \lVert P_N P_{N_1} u_1 \rVert_{L^2_x}^2
    \lVert P_{N_{(1)}} R \rVert_{L^4_x}^2
    \prod_{j=3}^5
    \lVert P_{N_j} u_j \rVert_{L^4_x}^2 \, 
    dt 
    & \leq 
    \int 
    \lVert P_{N_{(1)}} R \rVert_{L^4_x}^2
    \prod_{j\neq 2}
    \lVert P_{N_j} u_j \rVert_{L^4_x}^2 \, 
    dt \\
    & \leq 
    \lVert P_{N_{(1)}} R \rVert_{L^{2q}_{x,t}}^2
    \prod_{j\neq 2}
    \lVert P_{N_j} u_j \rVert_{L^{8+\delta_1}_{x,t}}^2 \\
    & \lesssim 
    \frac{\log(1/\varepsilon)}{N_{(1)}^{2\alpha - \varepsilon_1}} \, 
    \prod_{j=2}^5
    \lVert P_{N_{(j)}} u_{(j)}
    \rVert_{L^{8+\delta_1}_{x,t}}^2.
    \qedhere
\end{align} 
\end{proof}

Thanks to Proposition~\ref{prop:Qsmoothing-par} and in view of \eqref{eq:decomp-lp}, 
only the contribution close to the paraboloid
\begin{align}\label{eq:decomp-parproj}
    \sum_{\substack{N_{(1)} \geq \ldots\geq N_{(5)} \\ N_i \geq M_i }} \lVert P_N P_{\langle \tau + |n|^2 \rangle \leq N^{1+c}} \mathcal{Q}(P_{N_{1}}u_{1}, P_{N_{2}}u_{2}, P_{N_{3}}u_{3}, P_{N_{4}}u_{4}, P_{N_{5}}u_{5}) \rVert_{X^{0,b'-1}_\delta}
\end{align}
remains to be bounded, and only when $u_i = R$ whenever $N_i \simeq N_{(1)}$.

\subsection{Fourier restriction method} 
\label{subsec:frm}

To bound \eqref{eq:decomp-parproj}, we follow the strategy in \cite{Bourgain1996}. Write the deterministic term $D$ as
\begin{align}
    D(t,x) 
    = \sum_n d_n(t) \, e^{inx}
    & = \sum_n \int_\tau \widetilde d_n(\tau - |n|^2) \, e^{inx - i|n|^2 t }\, e^{i\tau t} \, d\tau \\
    & = \sum_n 
    \Big( \int_\tau  \phi(\tau) \, D_n(\tau)  \, e^{it\tau} \, d\tau \Big)
    e^{inx - i|n|^2 t },
\end{align}
where we denote
\begin{equation}
    \phi(\tau) = 
    \Big( \sum_k \langle k \rangle^{2s} | \widetilde d_k(\tau - |k|^2) |^2 \Big)^{1/2}, 
    \qquad 
    D_n(\tau) 
    = \frac{\widetilde d_n(\tau - |n|^2)}{\big( \sum_k \langle k \rangle^{2s} | \widetilde d_k(\tau - |k|^2) |^2 \big)^{1/2}}.
\end{equation}
The coefficient $D_n(\tau)$ satisfies
\begin{equation}\label{eq:Bound_Dn}
    \lVert \langle n \rangle^s D_n(\tau) \rVert_{\ell^2_n}^2
    = \sum_n \langle n \rangle^{2s}\, |D_n(\tau)|^2 = 1, 
    \qquad \forall \tau \in \mathbb R, 
\end{equation}
while 
\begin{equation}\label{eq:Norm_Of_Phi}
    \int \langle \tau \rangle^{2b} \, |\phi(\tau)|^2 \, d\tau
    = 
    \int \langle \tau \rangle^{2b} 
    \sum_k \langle k \rangle^{2s} | \widetilde d_k(\tau - |k|^2) |^2 \, d\tau
    =
    \lVert D \rVert_{X^{s,b}}^2.
\end{equation}
Hence, for every $j \in \{1, \ldots, 5\}$, we can write
\begin{equation}\label{eq:Cases_RD}
    P_M u_j 
    = 
    \sum_{|n| \simeq M } (C_j)_n(t) \, e^{inx - it|n|^2},
    \quad \text{ with } \quad \,  
    (C_j)_n(t) := 
    \left\{
    \begin{array}{ll}
        \displaystyle 
        R_n := \frac{g_n^\omega}{\langle n \rangle^{1+\alpha}}, 
        & \text{ if } j \in \mathcal R, \\
        \displaystyle
        \int  \phi(\tau) \, D_n(\tau) \,  e^{it\tau} \, d\tau, 
        & \text{ if } j \in \mathcal D.
    \end{array}
    \right.
\end{equation}
Denote $\{\iota_j\}_{j=1}^5$ as in \eqref{eq:def-iotaj}, and let
\begin{equation}
    \Omega
    = |n|^2 - |n_1|^2 + |n_2|^2 - |n_3|^2 + |n_4|^2 - |n_5|^2. 
\end{equation}
Then, using \eqref{eq:Cases_RD}, 
we write the Fourier coefficients for $\mathcal{Q}$ in \eqref{eq:decomp-parproj} as
\begin{align}\label{eq:Fourier_Coeff_Q}
    & \mathcal{Q} (P_{N_1}u_1, P_{N_2}u_2, P_{N_3}u_3, P_{N_4}u_4, P_{N_5}u_5)_n \nonumber \\
    & \hspace{1cm}
    = 
    \sum_{n_1, \ldots, n_5}^*
    \prod_{j=1}^5(C_j)_{n_j}(t)^{\iota_j}
    e^{-it (|n|^2 - \Omega)} \\
    & \hspace{1cm}
    = 
    \int 
    \sum_{n_1, \ldots, n_5}^*
    \bigg[
    \prod_{k \in \mathcal R} R_{n_k}^{\iota_k} \cdot 
    \prod_{j \in \mathcal D} D_{n_j}(\tau_j)^{\iota_j} \cdot 
    e^{-it (|n|^2 - \Omega - \sum_{j \in \mathcal D} \iota_j \tau_j ) } \, 
    \bigg] \, 
    \prod_{j \in \mathcal D} \phi(\tau_j) \, d\tau_j,   \nonumber \\
\end{align}
where for simplicity we denote 
\begin{align}
    \sum_{n_1, \ldots, n_5}^* 
    = \sum_{n_1, n_2, n_3, n_4, n_5} 
    \big( \prod_{j=1}^5 \mathbbm 1_{|n_i| \simeq N_i} \big) \, 
    \mathbbm 1_{n = n_1 - n_2 + n_3 - n_4 + n_5} \, 
    \mathbbm 1_{\mathcal A^c}, 
\end{align}
with the set $\mathcal A$ being defined, according to the renormalized nonlinearity $\mathcal Q$, by
\begin{align}
    \mathcal A = 
    \left\{
    \begin{array}{l}
        n_1 = n  \\
        n_3 \neq n \neq n_5 
    \end{array}
    \right\}
    \cup 
    \left\{
    \begin{array}{l}
        n_3 = n  \\
        n_1 \neq n \neq n_5 
    \end{array}
    \right\}
    \cup 
    \left\{
    \begin{array}{l}
        n_5 = n  \\
        n_1 \neq n \neq n_3 
    \end{array}
    \right\}.
\end{align}
Hence, from \eqref{eq:Fourier_Coeff_Q} we have
\begin{align}
    & \mathcal F_{x,t}
    \big(  \, 
    \eta_\delta  \, \mathcal{Q} (P_{N_1}u_1, P_{N_2}u_2, P_{N_3}u_3, P_{N_4}u_4, P_{N_5}u_5)
    \, \big)
    (\tau - |n|^2,n) \\
    & \hspace{1cm}
    = 
    \int 
    \sum_{n_1, \ldots, n_5}^*
    \bigg[
    \prod_{k\in \mathcal R} R_{n_k}^{\iota_k} \cdot 
    \prod_{j \in \mathcal D} D_{n_j}(\tau_j)^{\iota_j} \cdot
    \widehat{\eta_\delta}
    \big(
    \tau - \sum_{j\in \mathcal D} \iota_j \tau_j  - \Omega
    \big)
    \bigg]
    \prod_{j \in \mathcal D} \phi(\tau_j) \, d\tau_j.
\end{align}
Using Cauchy-Schwarz in the $\tau_j$ integrals, by \eqref{eq:Norm_Of_Phi} we get
\begin{align}
    & 
    \big| \mathcal F_{x,t}
    ( 
    \eta_\delta  \mathcal{Q} (P_{N_1}u_1, P_{N_2}u_2, P_{N_3}u_3, P_{N_4}u_4, P_{N_5}u_5)
    )
    (\tau - |n|^2,n) \big|^2 \\
    & \hspace{1.5cm}
    \leq 
    \lVert D \rVert_{X^{s,b}}^{2|\mathcal D|}
    \int 
    \Bigg|
    \sum_{n_1, \ldots, n_5}^*
    \bigg[
    \prod_{k \in \mathcal R} R_{n_k}^{\iota_k} \cdot 
    \prod_{j \in \mathcal D}  D_{n_j}(\tau_j)^{\iota_j} \cdot
    \widehat{\eta_\delta}
    \big(
    \tau - \sum_{j \in \mathcal D} \iota_j \tau_j  - \Omega
    \big)
    \bigg]
    \Bigg|^2
    \prod_{j \in \mathcal D} \frac{d\tau_j}{\langle \tau_j \rangle^{2b}}.
\end{align}
With this, taking the supremum in $\tau, \tau_j$ of the $\ell^2_n$ norm,  
we bound the norm in \eqref{eq:decomp-parproj} by
\begin{align}\label{eq:Last_To_Estimate}
    & \big\lVert P_N  P_{\langle \tau + |n|^2 \rangle \leq N^{1+c}} 
    \, \eta_\delta \, 
     \mathcal{Q} (P_{N_1}u_1, P_{N_2}u_2, P_{N_3}u_3, P_{N_4}u_4, P_{N_5}u_5) \big\rVert_{X^{s,b'-1}} \nonumber \\
    & \quad \lesssim
    \|D\|_{X^{s,b}}^{|\mathcal{D}|}
    \bigg \lVert \frac{ \mathrm1_{|n| \simeq N } \,  \mathrm1_{\langle \tau \rangle \leq N^{1+c}} }{\langle \tau \rangle^{1-b'} \prod_{j \in \mathcal D} \langle \tau_j \rangle^b} \sum_{n_1, \ldots, n_5}^*
    \bigg[
    \prod_{k \in \mathcal R} R_{n_k}^{\iota_k} \cdot 
    \prod_{j \in \mathcal D}  D_{n_j}(\tau_j)^{\iota_j} \cdot
    \widehat{\eta_\delta}
    \big(
    \tau - \sum_{j \in \mathcal D} \iota_j \tau_j  - \Omega
    \big)
    \bigg] 
    \bigg\rVert_{\ell^2_n L^2_\tau L^2_{\{\tau_j\}}} \nonumber \\
    & \quad \lesssim
    \|D\|_{X^{s,b}}^{|\mathcal{D}|} \, N^{ 2\beta'} \sup_{\tau, \{\tau_j\}_{j \in \mathcal D}} \bigg\lVert \sum_{n_1, \ldots, n_5}^*
    \bigg[
    \prod_{k \in \mathcal R} R_{n_k}^{\iota_k} \cdot 
    \prod_{j \in \mathcal D}  D_{n_j}(\tau_j)^{\iota_j} \cdot
    \widehat{\eta_\delta}
    \big(
    \tau - \sum_{j \in \mathcal D} \iota_j \tau_j  - \Omega
    \big)
    \bigg] \bigg\rVert_{\ell^2_{|n|\simeq N}} ,
\end{align}
where we used that due to $2(1-b') = 1 - 2\beta'$ and $2b = 1 + 2\beta$ we have
\begin{align}
    \bigg \lVert \frac{\mathrm1_{\langle \tau \rangle \leq N^{1+c}} }{\langle \tau \rangle^{1-b'} \prod_{j \in \mathcal D} \langle \tau_j \rangle^b}  \bigg\rVert_{ L^2_\tau L^2_{\{\tau_j\}}}^2 = 
    \int \frac{\mathrm1_{\langle \tau \rangle \leq N^{1+c}}}{\langle \tau \rangle^{2(1-b')}} \, d\tau \cdot \prod_{j\in \mathcal D} \int  \frac{d\tau_j}{\langle \tau_j \rangle^{2b}} 
    \lesssim_b N^{4\beta'}.
\end{align}
For fixed $\tau$ and $\{\tau_j\}_{j \in \mathcal D}$, 
by arranging the sum in $n_1, \ldots, n_5$ according to the value of $\Omega$ and the triangle inequality,  we bound
\begin{align}
    & \bigg\lVert \sum_{n_1, \ldots, n_5}^*
    \Big(
    \prod_{k \in \mathcal R} R_{n_k}^{\iota_k} \cdot 
    \prod_{j \in \mathcal D}  D_{n_j}(\tau_j)^{\iota_j} \cdot
    \widehat{\eta_\delta}
    \big(
    \tau - \sum_{j \in \mathcal D} \iota_j \tau_j  - \Omega
    \big)
    \Big) 
    \bigg\rVert_{\ell^2_{|n|\simeq N}} \\
    & \qquad \qquad \lesssim 
    \sum_{\mu \in \mathbb Z} \Big|\widehat{\eta_\delta}
    \big(
    \tau - \sum_{j \in \mathcal D} \iota_j \tau_j  - \mu \big) \Big| \, 
    \bigg\lVert \sum_{ \substack{n_1, \ldots, n_5 \\ \Omega = \mu}}^*
    \prod_{k \in \mathcal R} R_{n_k}^{\iota_k} \cdot 
    \prod_{j \in \mathcal D}  D_{n_j}(\tau_j)^{\iota_j} 
    \bigg\rVert_{\ell^2_{|n|\simeq N}} \\
    & \qquad \qquad \lesssim 
    \sup_{\mu}
    \bigg\lVert \sum_{ \substack{n_1, \ldots, n_5 \\ \Omega = \mu}}^*
    \prod_{k \in \mathcal R} R_{n_k}^{\iota_k} \cdot 
    \prod_{j \in \mathcal D}  D_{n_j}(\tau_j)^{\iota_j} 
    \bigg\rVert_{\ell^2_{|n|\simeq N}} \, \cdot \, 
    \sum_{\mu \in \mathbb Z} \Big| \widehat{\eta_\delta}
    \big(
    \tau - \sum_{j \in \mathcal D} \iota_j \tau_j  - \mu \big) \Big|, 
\end{align}
and since $\sum_{\mu \in \mathbb Z} |\widehat{\eta_\delta} (A - \mu) |\lesssim_\eta 1$ for every $A \in \mathbb R$, we get
\begin{align}\label{eq:To_Bound_Pre}
    \eqref{eq:Last_To_Estimate}
    & \lesssim_{\eta,b}  \|D\|_{X^{s,b}}^{|\mathcal{D}|} \, 
    N^{2\beta'} \, 
    \sup_{\mu, \,  \{\tau_j \}_{j \in \mathcal D}} 
    \bigg\lVert \sum_{ \substack{n_1, \ldots, n_5 \\ \Omega = \mu}}^*
    \prod_{k \in \mathcal R} R_{n_k}^{\iota_k} \cdot 
    \prod_{j \in \mathcal D}  D_{n_j}(\tau_j)^{\iota_j} 
    \bigg\rVert_{\ell^2_{|n|\simeq N}}.
    \end{align}
Hence, it suffices to bound
\begin{align}\label{eq:To_Bound}
    \sum_{|n| \simeq N}
    \bigg|
    \sum_{ \substack{n_1, \ldots, n_5 \\ \Omega = \mu}}^*
    \prod_{k \in \mathcal R} R_{n_k}^{\iota_k} \cdot 
    \prod_{j \in \mathcal D}  D_{n_j}(\tau_j)^{\iota_j}
    \bigg|^2 
    = 
    \sum_{|n| \simeq N}
    \Big|
    \sum_{\substack{n_1, \ldots, n_5 \\ \Omega= \mu}}^*
    (u_1)_{n_1} \, \overline{(u_2)_{n_2}} \, (u_3)_{n_3} \, \overline{(u_4)_{n_4}} \, (u_5)_{n_5} 
    \bigg|^2,
\end{align}
uniformly in $\mu$ and $\{ \tau_j \}_{j \in \mathcal D}$ and for every combination of $u_i \in \{R, D (\tau_i)\}$ such that $u_{(1)} = R$.  To do this, we split the sum above into parts 
\begin{equation}\label{eq:To_Bound_j}
    \sum_{|n| \simeq N}
    \bigg|
    \sum_{\substack{(n_1, \ldots, n_5) \in \mathcal A_j \\ \Omega= \mu}}^*
    (u_1)_{n_1} \, \overline{(u_2)_{n_2}} \, (u_3)_{n_3} \, \overline{(u_4)_{n_4}} \, (u_5)_{n_5} 
    \bigg|^2, 
\end{equation}
where $\mathcal A_j \subset \mathcal A = \{ \, (n_1,n_2,n_3,n_4,n_5) \in (\mathbb Z^2)^5 \, : \, n_1 - n_2 + n_3 - n_4 + n_5 = n \, \} $ are defined as 
\begin{align}\label{eq:Aj}
    \begin{array}{rl}
    \mathcal{A}_1 & = \mathcal A \cap \{ \, n_1, \, n_3, \, n_5 \neq n, \, n_2, \, n_4 \, \}, \\
    \mathcal{A}_2 & = \mathcal A \cap \{ \, n_4 = n_5 \text{ and }  \, n_1, \, n_3 \neq n_2 \, \}, \\
    \mathcal{A}_3 & = \mathcal A \cap \{ \, n_4 = n_5 = n \text{ and } \, n_1, \, n_3 \neq n_2 \,  \}, \\
    \mathcal{A}_4 & = \mathcal A \cap \{  \,  n_1 = n_4 = n_5  \text{ and } \, n_1, \, n_3 \neq n_2 \, \}, \\
    \mathcal{A}_5 & = \mathcal A \cap \{  \,  n_2 = n_4 = n_5  \text{ and } \, n_1, \, n_3 \neq n_2 \,  \}, \\
    \mathcal{A}_6 & = \mathcal A \cap \{ \,  n_3 = n_5 = n \,  \text{ and } \, n_2, \, n_4 \neq n_1 \, \}, \\
    \mathcal{A}_7 & = \mathcal A \cap \{ \,  n_1 = n_3 = n_5 = n \, \text{ and } \, n_2 + n_4 = 2n \, \text{ and } \, n_2 \neq n \,\}, \\
    \mathcal{A}_8 & = \mathcal A \cap \{ \, n_1 = n_3 = n_4 = n_5 \text{ and } \, 2n_1 - n_2 = n \, \text{ and } \, n_1 \neq n \,  \}, \\
    \mathcal{A}_9 & = \mathcal A \cap \{ \,  n_2 = n_4 = n_5 = n \text{ and } \, n_1 + n_3 = 2n \, \text{ and } \, n_1 \neq n \,  \}, \\
    \mathcal{A}_{10} & = \mathcal A \cap \{ \, n_1 = n_2 \text{ and } \, n_3 = n_4 = n_5 = n \, \}, \\
    \mathcal{A}_{11} & = \mathcal A \cap \{ \,  n_1 = n_2 = n_3 = n_4 = n_5 = n \,  \};
    \end{array}
\end{align}
and correspond to different resonance conditions coming from the decomposition 
\begin{align}
    \mathcal{Q}(u) = \sum_{j=1}^{11} \mathcal{Q}_j(u),
\end{align}
of the renormalized nonlinearity \eqref{eq:cauchy-nls-gauge} as in \cite{NahmodStaffilani2015}, 
where each $\mathcal Q_j(u)$ is defined as
\begin{align}\label{eq:qnl-fcoeff}
\begin{array}{rcl}
    \mathcal{Q}_1(u)_n  & = & \sum_{ \substack{  n_1 - n_2 + n_3 - n_4 + n_5 = n \\ n_1, \, n_3, \, n_5 \neq n,  \, n_2, \, n_4  } } u_{n_1} \, \overline{u_{n_2}} \, u_{n_3} \, \overline{u_{n_4}} u_{n_5}, \\
    \mathcal{Q}_2(u)_n & = & 
     6 \, M \, \mathcal W_3(u)_n, \\
    \mathcal{Q}_3(u)_n & = & 
    - 6 \, |u_n|^2 \mathcal W_3(u)_n, \\
    \mathcal{Q}_4(u)_n & = & - 6 \mathcal W_3(\mathcal P_3(u),u,u)_n, \\
    \mathcal{Q}_5(u)_n  & = & - 3 \mathcal W_3(u,\mathcal P_3(u),u)_n, \\
    \mathcal{Q}_6(u)_n & = & 
    - 3\, u_n^2 \, \overline{\mathcal W_3(u)_n}, \\
    \mathcal{Q}_7(u)_n & = & u_n^3 \sum_{  \substack{n_2 + n_4 = 2n \\ n_2 \neq n } } \overline{u_{n_2}} \, \overline{u_{n_4}}, \\
    \mathcal{Q}_8(u)_n & = & 2 \sum_{ \substack{ 2n_1 - n_2 = n \\ n_1 \neq n } } |u_{n_1}|^2 u_{n_1}^2 \overline{u_{n_2}}, \\
    \mathcal{Q}_9(u)_n & = & 3 |u_n|^2 \overline{u_n} \sum_{ \substack{ n_1 + n_3 = 2n \\ n_1 \neq n } } u_{n_1} u_{n_3}, \\
    \mathcal{Q}_{10}(u)_n & = &
    - 6\, M \, |u_n|^2 \,  u_n, \\
    \mathcal{Q}_{11}(u)_n & = & 4\, |u_n|^4 \, u_n,
\end{array}
\end{align}
and where we are denoting the conserved mass by $M = \int |u|^2 \, dx$ and 
\begin{align}
    \mathcal P_3(u)_n = |u_n|^2 u_n, 
    \qquad \text{ and } \qquad 
    \mathcal W_3(u)_n = \big( |u|^2u  - 2 M u \big)_n. 
\end{align}
In this setting, again slightly abusing notation, denote the corresponding multilinear operators by
\begin{align}\label{eq:Def_Q_ell_Abuse}
    \mathcal{Q}_\ell(u_1, \ldots, u_5)_n = 
    \sum_{\substack{ 
    (n_1, \ldots, n_5) \in \mathcal{A}_\ell  \\
    |n_i| \simeq N_i, \, \, \,  \Omega = \mu } } (u_1)_{n_1} \overline{(u_2)_{n_2}} (u_3)_{n_3} \overline{(u_4)_{n_4}} (u_5)_{n_5}, \qquad \ell \in \{1,\ldots,11\}.
\end{align}
Then, according to \eqref{eq:Last_To_Estimate}, \eqref{eq:To_Bound_Pre} and \eqref{eq:To_Bound}, to prove Proposition~\ref{prop:Qsmoothing-simp} it suffices to show that for any $\varepsilon > 0$ we have
\begin{align}\label{eq:gralbound_0}
    \sum_N N^{2s + 4\beta'} 
    \bigg( \sum_{\substack{N_{(1)} \geq \ldots \geq N_{(5)} \\ N_i \geq M_i }}\Big( \sum_{|n| \simeq N}  \big| \mathcal{Q}_\ell(u_1, \ldots, u_5)_n \big|^2 \Big)^{\frac12}  \bigg)^2 
    \leq  \Big(\log \frac{1}{\varepsilon} \Big)^{|\mathcal{R}|} \prod_{i \in \mathcal{R}}  \frac{1}{\langle M_i \rangle^{2\varepsilon_1}}
\end{align}
with probability $\geq 1-\varepsilon$. 
In turn, for this, it suffices to prove the following result.
\begin{prop}\label{prop:Final_Estimate}
Suppose $u_{(1)} = R$. Then, for any $\varepsilon > 0$, with probability $\geq 1 - \varepsilon$ we get
    \begin{align}\label{eq:gralbound}
    \sum_{|n| \simeq N}  \big| \mathcal{Q}_\ell(u_1, \ldots, u_5)_n \big|^2 
    \lesssim  
     \Big(\log \frac{1}{\varepsilon} \Big)^{|\mathcal{R}|} \frac{1}{N_{(1)}^{1 + 2\alpha - 2\varepsilon_1}},
     \qquad \forall \varepsilon_1 > 0, 
\end{align}
independently of the choice of $u_i \in \{ R, D(\tau_i) \}$, of $\ell \in \{1,\ldots,11\}$,  of $\mu$,  of $\{\tau_i\}_{i \in \mathcal{D}}$ and of $N$. 
\end{prop}

Indeed, if Proposition~\ref{prop:Final_Estimate} holds and recalling $s = \alpha + \sigma$, the left hand-side of \eqref{eq:gralbound_0} is bounded, up to the $\log(1/\varepsilon)^{|\mathcal R|}$ factor, by
\begin{align}
    \sum_N N^{2s + 4\beta'} 
    \bigg( \sum_{\substack{N_{(1)} \geq \ldots \geq N_{(5)} \\ N_i \geq M_i }} \frac{1}{N_{(1)}^{1/2 + \alpha - \varepsilon_1}}  \bigg)^2
    \lesssim 
    \sum_N \frac{N^{2s + 4\beta'}}{N^{1 + 2\alpha - 4\varepsilon_1}}
    \bigg( \sum_{\substack{N_{(1)} \geq \ldots \geq N_{(5)} \\ N_i \geq M_i }} \frac{1}{N_{(1)}^{\varepsilon_1}}  \bigg)^2
    \lesssim \prod_{j \in \mathcal R} \frac{1}{\langle M_j \rangle^{2\varepsilon_1/5}}
\end{align}
as long as
\begin{align}
    1 + 2\alpha - 4\varepsilon_1 - 2s - 4\beta' 
    = 1 - 2\sigma - 4\varepsilon_1 - 4\beta'
    > 0 
    \qquad \Longleftrightarrow \qquad 
    \sigma < \frac12 - 2 \varepsilon_1 - 2\beta'. 
\end{align}
Since $\varepsilon_1$ and $\beta'$ can be made arbitrarily small, the result holds for any $\sigma < 1/2$.

\begin{rmk}
The proof of Proposition~\ref{prop:Final_Estimate} shows that \eqref{eq:gralbound} is bounded by $N_{(1)}^{-1-2\alpha+2\varepsilon_1} N_{(2)}^{-(2\sigma - 1 + 2\alpha)}$. What is more, except in the cases $RDDDD$ and $DRDDD$ without pairings (which are computed in Sections \ref{subsec:rdddd-nopairing} and \ref{subsec:drddd-nopairing}), we obtain a better decay of $N_{(1)}^{-1-2\alpha+2\varepsilon_1} N_{(2)}^{-1} N_{(3)}^{-(2\sigma - 1 + 2\alpha)}$.
\end{rmk}

We prove Proposition~\ref{prop:Final_Estimate} in Section~\ref{sec:proof-smoothing}, for which we will use random tensor estimates.

\subsection{Random tensor estimates}
\label{subsec:rt}

Let $A$ and $B$ be finite index sets, and let $r \in \mathbb{N}$. For a deterministic base tensor 
\begin{align}
    h^b : (\mathbb Z^d)^r \times  (\mathbb Z^d)^{|A|} \times (\mathbb Z^d)^{|B|} \to \mathbb R, 
    \qquad \qquad h^b = h^b_{n_1\ldots n_rn_An_B}, 
\end{align}
and for $(g_n^{\omega})_{n \in \mathbb Z^d}$ i.i.d. standard Gaussian random variables, define the random tensor
    \begin{equation}\label{eq:randomop-gral}
        H_{n_An_B} = \sum_{n_1, \ldots, n_r} h^b_{n_1\ldots n_rn_An_B} \,  g_{n_1}^{\iota_1} \ldots g_{n_r}^{\iota_r}, 
    \end{equation}
where $\iota_1, \ldots, \iota_r \in \{+,-\}$ and $g_i^+ = g_i$ and $g_i^- = \overline{g_i}$. 
This random tensor will act as the kernel of an operator $T: \ell^2_{n_A} \to \ell^2_{n_B}$.
The following estimate from \cite[Proposition 4.14]{DNY2022} gives an estimate for its operator norm. We present here a slightly adapted version, following \cite{Kaneshiro2025}.

\begin{thm}[Random tensors estimate]\label{thm:Random_Tensor}
    Let $N \geq 1$ and $\varepsilon_1 > 0$. Assume that on the support of the base tensor $h^b$ we have 
    \begin{equation}
        \max_{a \in A}|n_a|, \,  \max_{b \in B} |n_b| \lesssim N, 
        \qquad \text{ and } \qquad 
        \max_{i=1, \ldots, r}|n_i| \lesssim N. 
    \end{equation}
    Assume also that $h^b$ does not allow pairings, that is, 
    \begin{equation*}
        n_i = n_j
        \quad \text{ and } \quad 
        \iota_i \neq \iota_j
        \quad \text{ for some } i \neq j
        \qquad \Longrightarrow \qquad h_{n_1\ldots n_rn_An_B} = 0. 
    \end{equation*}
    Then, for the random operator $T: \ell^2_{n_A} \to \ell^2_{n_B}$ defined by 
    \begin{align*}
        T(c_{n_A})_{n_B} 
        = \sum_{n_A} H_{n_An_B} c_{n_A}
        = \sum_{n_A} \Big( \sum_{n_1, \ldots, n_r} h^b_{n_1\ldots n_r n_A n_B} \, g_{n_1}^{\iota_1} \ldots g_{n_r}^{\iota_r} \Big)\, c_{n_A} 
        ,
    \end{align*}
    we have
    \begin{align}
        \mathbb E \left[\lVert T \rVert_{n_A \to n_B}^p
        \right]^{1/p}
        \lesssim
        p^{r/2} \, (\log N)^{r/2} \, 
        \max_{X \dot\cup Y = \{1, \ldots, r\}} \lVert h \rVert_{n_An_X \to n_Bn_Y}
    \end{align}
    for all $p \geq 1$, 
    where $\{X, Y\}$ is a partition of $\{1, \ldots, r\}$
    and where the norm $\lVert h \rVert_{n_An_X \to n_Bn_Y}$ refers to the operator norm of the deterministic operator $T_D:\ell^2_{n_An_X} \to \ell^2_{n_Bn_Y}$ defined by 
    \begin{equation*}
        T_D(c_{n_An_X})_{n_Bn_Y} = \sum_{n_A, \,  n_X} h_{n_Xn_Yn_An_B} \, c_{n_An_X}
        = \sum_{n_A, \,  n_X} h_{n_1 \ldots n_r n_An_B} \, c_{n_An_X}.
    \end{equation*}
\end{thm}
\begin{rmk}
As a consequence of Theorem~\ref{thm:Random_Tensor} and Chebyshev's inequality, as well as the strategy in the proof of Lemma \ref{lemma:Strichartz_Prob_Uniform}, 
for any $\varepsilon > 0$ we have 
    \begin{align}\label{eq:Random_Tensor_pointwise}
        \lVert T \rVert_{n_A \to n_B} 
        \lesssim \Big(\log \frac{1}{\varepsilon}\Big)^{r/2} \, (\log N)^{r/2} \, 
        \max_{X \dot\cup Y = \{1, \ldots, r\}} \lVert h \rVert_{n_An_X \to n_Bn_Y}, 
        \qquad \forall N \in 2^{\mathbb N}, 
    \end{align}
with probability $\geq 1-\varepsilon$. 
\end{rmk}
In our case, to bound \eqref{eq:gralbound}, 
if we define our base tensor by
\begin{align}\label{eq:basetensor-def}
    h^b_\ell 
    = h^b_\ell(n_1, \ldots, n_5,n)
    = \prod_{i=1}^5 \mathbbm 1_{|n_i| \simeq N_i} \, \cdot \, \mathbbm 1_{|n| \simeq N}
    \mathbbm 1_{n_1 - n_2 + n_3 - n_4 + n_5 = n} \,\mathbbm 1_{\Omega = \mu} \, \mathbbm 1_{\mathcal{A}_\ell},
\end{align}
then we can write $\mathcal{Q}_\ell$ defined in \eqref{eq:Def_Q_ell_Abuse} as
\begin{align}\label{eq:Our_Quantity_As_A_Tensor}
    \mathcal{Q}_\ell(u_1, \ldots, u_5)_n
    = \sum_{n_1, \ldots, n_5} h^b_j \, \,  (u_1)_{n_1} \overline{(u_2)_{n_2}} (u_3)_{n_3} \overline{(u_4)_{n_4}} (u_5)_{n_5}  
    & = \sum_{n_\mathcal{D}} \bigg( \sum_{n_\mathcal{R}} h^b_\ell \prod_{k \in \mathcal{R}} R_{n_k}  \bigg) \prod_{j \in \mathcal{D}} D_{n_j} \nonumber \\
    & = T_\ell \Big( \prod_{j \in \mathcal D} D_{n_j} \Big),
\end{align}
where $T_\ell : \ell^2_{n_{\mathcal{D}}} \rightarrow \ell^2_n$ is the random operator defined as
\begin{align}\label{eq:Kernel_RT}
    T_\ell (a_{n_{\mathcal{D}}})_n =
    \sum_{n_\mathcal{D}} K^\ell_{n n_{\mathcal{D}}} \, a_{n_{\mathcal D}}, 
    \qquad \text{ with kernel } \qquad 
    K^\ell_{n n_{\mathcal{D}}} 
    = \sum_{n_\mathcal{R}} h^b_\ell \prod_{k \in \mathcal{R}} \frac{g_{n_k}^{\omega}}{\langle n_k \rangle^{1 + \alpha}}.
\end{align}
Then, since the norms of the $D$ components are controlled in \eqref{eq:Bound_Dn}, we can bound
\begin{align}\label{eq:btensor-intro}
    \sum_{|n| \simeq N}  \big| \mathcal{Q}_\ell(u_1, \ldots, u_5)_n \big|^2 
    = \Big\lVert T_\ell \Big( \prod_{j \in \mathcal D} D_{n_j} \Big) \Big\rVert_{\ell_n^2}^2 
    \leq 
    \left\| 
    T_\ell\right\|_{n_{\mathcal{D}} \rightarrow n}^2  \, 
    \prod_{j \in \mathcal{D}} \|D_{n_j}\|_{n_j}^2
    \lesssim 
    \left\| 
    T_\ell\right\|_{n_{\mathcal{D}} \rightarrow n}^2  \, 
    \prod_{j \in \mathcal{D}} \frac{1}{N_j^{2s}}, 
\end{align}
and if the base tensor $h^b_\ell$ has no pairings (which is determined by $\mathbbm 1_{\mathcal A_\ell}$), 
by Theorem~\ref{thm:Random_Tensor} we have
\begin{align} 
    \left\| 
    T_\ell\right\|_{n_{\mathcal{D}} \rightarrow n}^2 
    : = \big\| 
    K^\ell_{nn_{\mathcal D} } \big\|_{n_{\mathcal{D}} \rightarrow n}^2
    \lesssim 
    \Big( \log \frac{1}{\varepsilon} \Big)^{|\mathcal R|} \, 
    N_{(1)}^{\varepsilon_1} \, 
    \prod_{k \in \mathcal{R}} \frac{1}{ N_{k}^{2 + 2\alpha}}  \, 
    \max_{X \dot\cup Y = \mathcal R} 
    \| h^b_\ell \|_{n_{\mathcal{D}} n_X \rightarrow n n_Y}^2. 
\end{align}
We will then bound the operator norms of the deterministic base tensor $h_\ell^b$ by the Schur test in Lemma~\ref{lemma:schur-test}.  
What is more, for any $\mathcal{D}' \subsetneq \mathcal{D}$, by Cauchy-Schwarz in $n_{\mathcal D'}$ we can bound
\begin{align}
    T_\ell \Big( \prod_{j \in \mathcal D} D_{n_j} \Big) 
    \leq 
    \sum_{n_{\mathcal D'}} 
    \prod_{j \in \mathcal D'} |D_{n_j}|^2 
    \cdot 
    \sum_{n_{\mathcal D'}}
    \Big| 
    \sum_{n_{D \setminus \mathcal D'}} K^\ell_{nn_{\mathcal D}} \prod_{s \in \mathcal D \setminus \mathcal D'} D_{n_s} 
    \Big|^2, 
\end{align}
and therefore, by Theorem~\ref{thm:Random_Tensor} we get
\begin{align}
    \Big\lVert T_\ell \Big( \prod_{j \in \mathcal D} D_{n_j} \Big) \Big\rVert_{\ell_n^2}^2 
    & \leq 
    \prod_{j \in \mathcal D'} \lVert D_{n_j}\rVert_{n_j}^2
    \cdot 
    \sum_{n, n_{\mathcal D'}}
    \Big| 
    \sum_{n_{D \setminus \mathcal D'}} K^\ell_{nn_{\mathcal D}} \prod_{s \in \mathcal D \setminus \mathcal D'} D_{n_s} 
    \Big|^2 \\
    & \leq
    \prod_{j \in \mathcal D} \lVert D_{n_j}\rVert_{n_j}^2
    \cdot 
    \lVert K^\ell_{nn_{\mathcal D}} \rVert_{n_{\mathcal D \setminus \mathcal D'} \to nn_{\mathcal D'}}^2.
\end{align}
This shows that, as long as there are no pairings, 
bounding \eqref{eq:To_Bound} will be reduced 
to bounding 
\begin{align}
    \bigg\| \sum_{n_\mathcal{R}} h^b_\ell \prod_{k \in \mathcal{R}} g_{n_{k}}^{\omega} \bigg\|_{n_{\mathcal{D} \setminus \mathcal{D}'} \rightarrow n n_{\mathcal{D}'}},
\end{align}
for the most convenient choice of $\mathcal D' \subset \mathcal D$.

\section{Proof of Proposition~\ref{prop:Final_Estimate}}
\label{sec:proof-smoothing}

Recall that, thanks to Proposition \ref{prop:Qsmoothing-mest}, we are restricted to $u_i = R$ whenever $N_i \simeq N_{(1)}$.
We bound each $\mathcal{Q}_\ell$ in \eqref{eq:qnl-fcoeff} separately. To lighten notation, we denote $g_n = g_n^{\omega}$.

\subsection{No pairings: $\boldsymbol{\mathcal{Q}_1}$}
We study different choices of $u_i \in \{R,D\}$ separately.
By symmetry, we can assume that $N_{(1)} = N_1$ or $N_{(1)} = N_2$.

\subsubsection{Case $RRRRR$}
\label{sec:RRRRR}
By denoting $h_1^b$ as in \eqref{eq:basetensor-def}, 
we compute the expectation of \eqref{eq:gralbound} as follows:
    \begin{equation}\label{eq:forexample-ghs}
     \begin{split}
         \mathbb E  \sum_{|n| \simeq N}  \big| \mathcal{Q}_1(u_1, \ldots, u_5)_n \big|^2
         & = \mathbb E \sum_{|n| \simeq N} \bigg| \sum_{ n_1, \ldots, n_5 } h^b_1\frac{g_{n_1} \overline{g_{n_2}}  g_{n_3} \overline{g_{n_4}} g_{n_5} } {( \langle n_1 \rangle \langle n_2 \rangle \langle n_3 \rangle \langle n_4 \rangle \langle n_5 \rangle )^{1 + \alpha}} \bigg|^2   \\
         & \hspace{-2cm} = \sum_{|n| \simeq N} \sum_{ \substack{ n_1, \ldots, n_5 \\ n_1', \ldots , n_5'}}  h^b_1(n_1, \ldots, n_5, n) \, h^b_1(n_1', \ldots, n_5',n) \,  \frac{\mathbb E \big[ g_{n_1} \overline{g_{n_2}}  g_{n_3} \overline{g_{n_4}} g_{n_5} \overline{g_{n_1'}} g_{n_2'} \overline{g_{n_3'}} g_{n_4'} \overline{g_{n_5'}} \big] } { \prod_{j=1}^5 ( \langle n_j \rangle \langle n_j' \rangle )^{1 + \alpha}}  \\
         & \hspace{-2cm}  \lesssim  \frac{1}{(N_1N_2N_3N_4N_5)^{2+2\alpha}} \sum_n \sum_{ n_1, \ldots, n_5} h^b_1, 
     \end{split}
\end{equation}
since by the Wick-Isserlis theorem the expectation cancels unless $\{n_1, n_3, n_5 \} = \{n_1', n_3', n_5'\}$ and $\{n_2, n_4\} = \{n_2', n_4'\}$.
Then, the counting in Lemma~\ref{lemma:combest-nopairings-cubic} gives 
\begin{align}
    \sum_n \sum_{ n_1, \ldots, n_5} h^b_1
    \lesssim N_{(1)} N_{(2)}^{1+\varepsilon_1}(N_{(3)}N_{(4)}N_{(5)})^2, 
\end{align}
and therefore 
\begin{align}
    \mathbb E  \sum_{|n| \simeq N}  \big| \mathcal{Q}_1(u_1, \ldots, u_5)_n \big|^2
    \lesssim \frac{N_{(2)}^{\varepsilon_1}}{ N_{(1)}^{1 + 2\alpha} N_{(2)}^{1 + 2\alpha} (N_{(3)}N_{(4)}N_{(5)})^{2\alpha}}. 
\end{align}
Hence, by the Wiener chaos estimate in Proposition \ref{lemma:wiener-chaos}, with probability at least $1-\varepsilon$ we get 
\begin{align}\label{eq:wiener-app1}
    \sum_{|n| \simeq N}  \big| \mathcal{Q}_1(u_1, \ldots, u_5)_n \big|^2
    \lesssim \Big(\log \frac{1}{\varepsilon} \Big)^5  \frac{N_{(1)}^{2\varepsilon_1}}{N_{(1)}^{1 + 2\alpha}  N_{(2)}^{1 + 2\alpha} (N_{(3)}N_{(4)}N_{(5)})^{2\alpha} }.
\end{align}

\subsubsection{Case $RDDDD$}\label{subsec:rdddd-nopairing}
The quantity to bound from \eqref{eq:gralbound} is 
\begin{align}\label{eq:No_Pairings_ToBound_RDDDD}
    \sum_{|n| \simeq N}  \big| \mathcal{Q}_1(u_1, \ldots, u_5)_n \big|^2
    = 
    \sum_{|n|\simeq N} \bigg| \sum_{n_1, \ldots, n_5 } h^b_1 \frac{g_{n_1}}{\langle n_1 \rangle^{1+\alpha}} \overline{D_{n_2}} D_{n_3} \overline{D_{n_4}} D_{n_5} \bigg|^2.
\end{align}
In this case, $N_{(1)} = N_1$, and by symmetry we assume $N_3 \geq N_5$ and $N_2 \geq N_4$ without loss of generality. 
Separate $n_4$ by Cauchy-Schwarz and apply the random tensor estimate in Theorem~\ref{thm:Random_Tensor} to get
\begin{align}\label{eq:No_Pairings_ToBound_RDDDD_2}
    \eqref{eq:No_Pairings_ToBound_RDDDD}
    &\lesssim 
    \sum_{n, n_4} \bigg| \sum_{n_2,n_3,n_5}  \bigg( \sum_{n_1} h^b_1 \frac{g_{n_1}}{\langle n_1 \rangle^{1+\alpha}} \bigg) \overline{D_{n_2}} D_{n_3} D_{n_5}  \bigg|^2
    \, \sum_{n_4} |D_{n_4}|^2 \nonumber \\
    & \lesssim 
    \Big\lVert  \sum_{n_1} h^b_1 \frac{g_{n_1}}{\langle n_1 \rangle^{1+\alpha}}  \Big\rVert_{n_2n_3n_5 \to nn_4}^2 \, 
    \prod_{i \neq 1} \lVert D_{n_i} \rVert_{n_i}^2 \nonumber \\
    & \lesssim \Big( \log \frac{1}{\varepsilon} \Big)
    \frac{N_{(1)}^{\varepsilon_1}}{N_1^{2+2\alpha} \, (N_2N_3N_4N_5)^{2\alpha + 2\sigma}} \, \max\{  \lVert  h^b_1 \rVert_{n_1n_2n_3n_5 \to nn_4}^2, \lVert  h^b_1 \rVert_{n_2n_3n_5 \to nn_1n_4}^2  \}. 
\end{align}
where $\lVert D_{n_i} \rVert_{n_i}^2 \lesssim N_i^{-2\alpha - 2\sigma}$ by \eqref{eq:Bound_Dn}. 
On the other hand, 
by the Schur test \ref{lemma:schur-test} and the counting Lemma \ref{lemma:combest-nopairings-cubic}, the deterministic norms of the base tensor $h_1^b$ are bounded by
\begin{align}\label{eq:Counting_RDDDD_4}
    \lVert  h^b_1 \rVert_{n_1n_2n_3n_5 \to nn_4}^2
    & \lesssim \Big(\sup_{n, n_4} \sum_{n_1, n_2, n_3, n_5} h^b_1  \Big)
    \, \Big( \sup_{n_1, n_2, n_3, n_5} \sum_{n, n_4} h^b _1\Big) \nonumber \\
    &\lesssim \operatorname{max}_{\{2,3,5\}} (\operatorname{mid}_{\{2,3,5\}})^{1+{\varepsilon_1}} (\operatorname{min}_{\{2,3,5\}})^2 \, \cdot \, N_4^{\varepsilon_1}, 
\end{align}
and 
\begin{align}\label{eq:Counting_RDDDD_5}
    \lVert  h^b_1 \rVert_{n_2n_3n_5 \to nn_1n_4}^2 
    &  \lesssim \Big(\sup_{n, n_1, n_4} \sum_{n_2, n_3, n_5} h^b_1  \Big)
    \, \Big( \sup_{n_2, n_3, n_5} \sum_{n, n_1, n_4} h^b _1\Big) \nonumber \\
    &\lesssim (\operatorname{mid}_{\{2,3,5\}})(\operatorname{min}_{\{2,3,5\}})^{1+{\varepsilon_1}}  \, \cdot \, N_1 \, N_4^{1+\varepsilon_1}.
\end{align}
Up to the $\log (1/\varepsilon)$ term, from \eqref{eq:Counting_RDDDD_4} we get the bound
\begin{align}
    \eqref{eq:No_Pairings_ToBound_RDDDD_2}
    \lesssim 
    N_{(1)}^{3\varepsilon_1} \,
    \frac{\operatorname{max}_{\{2,3,5\}} \operatorname{mid}_{\{2,3,5\}} \operatorname{min}_{\{2,3,5\}}^2 }{N_1^{2+2\alpha}(N_2N_3N_4N_5)^{2\alpha + 2\sigma}} 
    \lesssim 
    N_{(1)}^{3\varepsilon_1} \,
    \frac{\operatorname{min}_{\{2,3,5\}} 
    }{N_1^{2+2\alpha}N_4^{2\alpha + 2\sigma}(N_2N_3N_5)^{2 \alpha + 2\sigma - 1}}, 
\end{align}
while from from \eqref{eq:Counting_RDDDD_5} we get
\begin{align}
    \eqref{eq:No_Pairings_ToBound_RDDDD_2}
    \lesssim 
    \frac{N_{(1)}^{3\varepsilon_1}}{N_1^{1+2\alpha}  \, \operatorname{max}_{\{2,3,5\}} \, (N_2 N_3 N_4 N_5)^{2 \alpha + 2\sigma - 1} }
    = 
    \frac{N_{(1)}^{3\varepsilon_1}}{N_1^{1+2\alpha}  \, N_{(2)} \, (N_2 N_3 N_4 N_5)^{2 \alpha + 2\sigma - 1} },
\end{align}
since $N_4 \leq N_2$ implies $\operatorname{max}_{\{2,3,5\}} = N_{(2)}$.
The first bound is the worst if $N_4 \ll N_2, N_3, N_5$, in which case we get 
\begin{align}
    \eqref{eq:No_Pairings_ToBound_RDDDD_2}
    \lesssim 
    \Big( \log \frac{1}{\varepsilon} \Big) 
    \frac{ N_{(1)}^{3\varepsilon_1}}{N_{(1)}^{1+2\alpha}N_4^{2\alpha + 2\sigma}(N_2N_3N_5)^{2 \alpha + 2\sigma - 1}}.
\end{align}

\begin{rmk}\label{eq:rmk-haitian}
Even if this bound provides the required smoothing $\sigma < 1/2$, a better estimate by $(N_{(1)} N_{(2)})^{-1}$ is possible by a more refined analysis based on a dyadic Fourier localization of the ansatz $y_N = \Phi^N_t f^{\omega}- \Phi^{N/2}_t f^{\omega}$, in which case one can exploit the independence of $R_{n_1}$ and $D_{n_j}$ when $N_1 \gg N_j$. The same happens for the case $DRDDD$. We do not pursue this here and we refer to \cite{DNY2024} for details.
\end{rmk}

\subsubsection{Case $DRDDD$}\label{subsec:drddd-nopairing}
We now have
\begin{align}\label{eq:No_Pairings_ToBound_DRDDD}
    \sum_{|n| \simeq N}  \big| \mathcal{Q}_1(u_1, \ldots, u_5)_n \big|^2
    = 
    \sum_{|n|\simeq N} \bigg| \sum_{n_1, \ldots, n_5 } h^b_1 \frac{\overline{g_{n_2}}}{\langle n_2 \rangle^{1+\alpha}} D_{n_1} D_{n_3} \overline{D_{n_4}} D_{n_5} \bigg|^2, 
\end{align}
in which case we have $N_{(1)} = N_2$. By symmetry we assume $N_1 \geq N_3 \geq N_5$ without loss of generality. 
Proceeding analogously to $RDDDD$, with probability $\geq 1 - \varepsilon$ we bound 
\begin{align*}
    \eqref{eq:No_Pairings_ToBound_DRDDD}
    &\lesssim \Big( \log \frac{1}{\varepsilon} \Big)
    \frac{N_{(1)}^{\varepsilon_1}}{N_2^{2+2\alpha} \, (N_1N_3N_4N_5)^{2\alpha + 2\sigma}} \, \max\{  \lVert  h^b_1 \rVert_{n_1n_2n_3n_5 \to nn_4}^2, \lVert  h^b_1 \rVert_{n_1n_3n_5 \to nn_2n_4}^2  \}. 
\end{align*}
Using Lemma \ref{lemma:combest-nopairings-cubic}, we adapt the bounds \eqref{eq:Counting_RDDDD_4} and \eqref{eq:Counting_RDDDD_5} by
\begin{align}
    \lVert  h^b_1 \rVert_{n_1n_2n_3n_5 \to nn_4}^2
    & \lesssim N_1 N_3^{1+{\varepsilon_1}} N_5^2 \, \cdot \, N_4^{\varepsilon_1},  \label{eq:Counting_RDDDD_4_1} \\
    \lVert  h^b_1 \rVert_{n_1n_3n_5 \to nn_2n_4}^2  
    &\lesssim N_3^{\varepsilon_1} N_5^2  \, \cdot \, N_2^{\varepsilon_1} \, N_4^2. \label{eq:Counting_RDDDD_4_2}
\end{align}
From \eqref{eq:Counting_RDDDD_4_1}, we get
\begin{align*}
    \eqref{eq:No_Pairings_ToBound_DRDDD}
    \lesssim 
    \Big( \log \frac{1}{\varepsilon} \Big)
    \frac{N_{(1)}^{3\varepsilon_1} N_5 }{N_2^{2 + 2\alpha} N_4 (N_1 N_3 N_4 N_5)^{2 \alpha + 2\sigma - 1} } \lesssim \Big( \log \frac{1}{\varepsilon} \Big)
    \frac{N_{(1)}^{3\varepsilon_1}  }{N_{(1)}^{1 + 2\alpha} N_4 (N_1 N_3 N_4 N_5)^{2 \alpha + 2\sigma - 1} };
\end{align*}
while from \eqref{eq:Counting_RDDDD_4_2} we get
\begin{align}\label{eq:No_Pairings_ToBound_RDDDD_int1} 
    \eqref{eq:No_Pairings_ToBound_DRDDD}
    \lesssim 
    \Big( \log \frac{1}{\varepsilon} \Big)
    \frac{N_{(1)}^{3\varepsilon_1} N_4 N_5 }{N_2^{2 + 2\alpha} N_1 N_3 (N_1 N_3 N_4 N_5)^{2 \alpha + 2\sigma - 1} } \lesssim \Big( \log \frac{1}{\varepsilon} \Big)
    \frac{N_{(1)}^{3\varepsilon_1} }{N_{(1)}^{1 + 2\alpha} N_1 (N_1 N_3 N_4 N_5)^{2 \alpha + 2\sigma - 1} }.
\end{align}
Thus, with probability $\geq 1 - \varepsilon$ we get
\begin{align}
    \eqref{eq:No_Pairings_ToBound_DRDDD}
    \lesssim 
    \Big( \log \frac{1}{\varepsilon} \Big) 
    \frac{ N_{(1)}^{3\varepsilon_1}}{N_{(1)}^{1+2\alpha} \operatorname{min}_{\{1,4\}}  (N_1N_3N_4N_5)^{2 \alpha + 2\sigma - 1}},
\end{align}
which yields a worse decay than $(N_{(1)} N_{(2)})^{-1}$ whenever $N_1 \gg N_4$ or $N_1 \ll N_4$.

\subsubsection{Case $RRDDD$.}
For the first case with two $R$, we have
\begin{align}\label{eq:No_Pairings_ToBound_RRDDD}
    \sum_{|n| \simeq N}  \big| \mathcal{Q}_1(u_1, \ldots, u_5)_n \big|^2
    = 
    \sum_n \bigg| \sum_{n_1, \ldots, n_5 } h^b_1 \frac{g_{n_1}}{\langle n_1 \rangle^{1+\alpha}} \frac{\overline{g_{n_2}}}{\langle n_2 \rangle^{1+\alpha}} D_{n_3} \overline{D_{n_4}} D_{n_5} \bigg|^2.
\end{align}
In this case we have $N_{(1)} = \operatorname{max}_{\{1,2\}}$, 
and we can assume $N_3 \geq N_5$ by symmetry. By Cauchy-Schwarz we split the smallest frequency between $n_4$ and $n_5$.  

\begin{itemize}
    \item Subcase 1. If $N_4 \leq N_5$,  again by \eqref{eq:Bound_Dn} and Theorem~\ref{thm:Random_Tensor} we have 
\begin{align}
    \eqref{eq:No_Pairings_ToBound_RRDDD}
    & \lesssim 
    \sum_{n,n_4} \bigg| \sum_{n_3,n_5 } \bigg( \sum_{n_1,n_2} h^b_1 \,  \frac{g_{n_1} \overline{g_{n_2}}}{\langle n_1 \rangle^{1+\alpha} \langle n_2 \rangle^{1+\alpha}} \bigg) D_{n_3}  D_{n_5} \bigg|^2
    \sum_{n_4} |D_{n_4}|^2 \\
    & \lesssim  
    \Bigg \lVert \sum_{n_1,n_2 } h^b_1 \frac{g_{n_1}}{\langle n_1 \rangle^{1+\alpha}} \frac{\overline{g_{n_2}}}{\langle n_2 \rangle^{1+\alpha}} \Bigg \rVert_{n_3n_5 \to nn_4}^2 \, \frac{1}{(N_3N_4N_5)^{2\alpha + 2\sigma}} \\
    & 
    \lesssim \Big(\log \frac{1}{\varepsilon}\Big)^2
    \frac{N_{(1)}^{\varepsilon_1}}{(N_1N_2)^{2 + 2\alpha}} \frac{1}{(N_3N_4N_5)^{2\alpha + 2\sigma}}
    \max_{X \dot{\cup} Y = \{1,2\}} \big\{
    \lVert  h^b \rVert_{n_Xn_3n_5 \to n_Ynn_4}^2
    \big\}.
\end{align}
By the Schur test \ref{lemma:schur-test} and the counting Lemma \ref{lemma:combest-nopairings-cubic}, 
we have the following possibilities:
\hfill \break
\begin{enumerate}
    \item Since
    $\lVert  h^b \rVert_{n_1n_2n_3n_5 \to nn_4}^2
    \lesssim \operatorname{mid}_{1,\{1,2,3,5\}} \operatorname{mid}_{2,\{1,2,3,5\}}^{1+\varepsilon_1} \operatorname{min}_{\{1,2,3,5\}}^2 N_4^{\varepsilon_1}$, 
    we get the bound
    \begin{align}
     \eqref{eq:No_Pairings_ToBound_RRDDD} 
     &\lesssim \Big(\log \frac{1}{\varepsilon}\Big)^2 \frac{N_{(1)}^{3\varepsilon_1} \, \min_{\{1,2,3,5\}} }{N_{(1)} (N_1N_2)^{1 + 2\alpha}N_4^{2\alpha + 2\sigma} (N_3N_5)^{2 \alpha + 2\sigma - 1}}
     \\
     &\lesssim 
     \Big(\log \frac{1}{\varepsilon}\Big)^2 \frac{N_{(1)}^{3\varepsilon_1} }{N_{(1)}^{2 + 2\alpha} \min_{\{1,2\}}^{2\alpha} N_4^{2\alpha + 2\sigma} (N_3N_5)^{2 \alpha + 2\sigma - 1}},
    \end{align} 
    since $\max_{\{1,2,3,5\}} = N_{(1)}$, and $\operatorname{min}_{\{1,2,3,5\}}\leq \operatorname{min}_{\{1,2\}}$. 
    \item Given that $\lVert  h^b_1 \rVert_{n_1n_3n_5 \to nn_2n_4}^2
    \lesssim 
    \operatorname{mid}_{\{1,3,5\}}^{\varepsilon_1} \operatorname{min}_{\{1,3,5\}}^2 \cdot \operatorname{max}_{\{2,4\}}^{\varepsilon_1} \operatorname{min}_{\{2,4\}}^2 \leq N_{(1)}^{2\varepsilon_1} N_{(3)}^2 N_{(5)}^2$,
    which holds because either $\operatorname{min}_{\{1,3,5\}}$ or $\operatorname{min}_{\{2,4\}}$ is $N_{(5)}$, and moreover both are $\leq N_{(3)}$, we get
    \begin{align}    
    \eqref{eq:No_Pairings_ToBound_RRDDD} 
    & \lesssim \Big(\log \frac{1}{\varepsilon}\Big)^2
    \frac{N_{(1)}^{3\varepsilon_1} N_{(3)}^2 N_{(5)}^2 }{N_{(1)}^{2+2\alpha}  N_{(2)} N_{(3)} N_{(4)} N_{(5)} \min_{\{1,2\}}^{1+2\alpha}(N_3N_4N_5)^{2 \alpha + 2\sigma - 1} } \\
    & \lesssim \Big(\log \frac{1}{\varepsilon}\Big)^2
    \frac{N_{(1)}^{3\varepsilon_1}}{N_{(1)}^{2+2\alpha} \min_{\{1,2\}}^{1+2\alpha}(N_3N_4N_5)^{2 \alpha + 2\sigma - 1} }.
    \end{align}
    
    \item Since $\lVert  h^b_1 \rVert_{n_2n_3n_5 \to nn_1n_4}^2
    \lesssim
    \operatorname{mid}_{\{2,3,5\}} \operatorname{min}_{\{2,3,5\}}^{1+\varepsilon_1} \, \operatorname{max}_{\{1,4\}} \operatorname{min}_{\{1,4\}}^{1+\varepsilon_1}$, 
    we get
    \begin{align}
    \eqref{eq:No_Pairings_ToBound_RRDDD} &\lesssim \Big(\log \frac{1}{\varepsilon}\Big)^2 \frac{N_{(1)}^{3\varepsilon_1}}{N_{(1)}^{1 + 2\alpha} N_{(2)} \min_{\{1,2\}}^{1 + 2\alpha} (N_3N_4N_5)^{2 \alpha + 2\sigma - 1} }  
    \end{align}
    because $\max_{\{2,3,5\}} = N_{(2)}$. 

    \item Since $\lVert  h^b_1 \rVert_{n_3n_5 \to nn_1n_2n_4}^2
    \lesssim
    N_5^{\varepsilon_1} \, \max_{\{1,2,4\}} \operatorname{mid}_{\{1,2,4\}}^{1+\varepsilon_1}\operatorname{min}_{\{1,2,4\}}^2$, 
    we get
    \begin{align*}
    \eqref{eq:No_Pairings_ToBound_RRDDD} &\lesssim \Big(\log \frac{1}{\varepsilon}\Big)^2 \frac{N_{(1)}^{3\varepsilon_1} \,  \min_{\{1,2,4\}} }{(N_1N_2)^{1 + 2\alpha}(N_3N_5)^{2 \alpha + 2\sigma}N_4^{2 \alpha + 2\sigma - 1}}. 
    \end{align*}
    If $\min_{\{1,2,4\}} = N_4 \leq N_1, N_2$, since $N_4 \leq N_5 \leq N_3 $ we have $N_{(2)} = \max(N_3, \min_{\{1,2\}})$ 
    and therefore 
\begin{align}
    \eqref{eq:No_Pairings_ToBound_RRDDD} & \lesssim  \Big(\log \frac{1}{\varepsilon}\Big)^2 \frac{N_{(1)}^{3\varepsilon_1}}{N_{(1)}^{1 + 2\alpha} \min_{\{1,2\}}^{1 + 2\alpha} N_3 (N_3N_4N_5)^{2 \alpha + 2\sigma - 1 }}, 
\end{align}
which has a factor $N_{(1)}^{1+2\alpha} N_{(2)}$ in the denominator. 
Else, if $\min_{\{1,2,4\}} = N_2$, then $N_2 \leq N_4 \leq N_5 \leq N_3 \leq N_1$, and therefore
\begin{align}
    \eqref{eq:No_Pairings_ToBound_RRDDD}  \lesssim \Big(\log \frac{1}{\varepsilon}\Big)^2 \frac{N_{(1)}^{3\varepsilon_1} }{N_{(1)}^{1 + 2\alpha} N_{(2)} N_{(3)} N_2^{2\alpha} (N_3N_4N_5)^{2 \alpha + 2\sigma - 1}}.
\end{align}
If $\min_{\{1,2,4\}} = N_1$, the computation is symmetric and the same bound holds with replacing $N_2$ by $N_1$. 
\end{enumerate}

Thus, in all cases, with probability $\geq 1 - \varepsilon$ we get at least
\begin{equation}
    \eqref{eq:No_Pairings_ToBound_RRDDD} 
    \lesssim \Big(\log \frac{1}{\varepsilon}\Big)^2 \frac{N_{(1)}^{3\varepsilon_1}}{N_{(1)}^{1 + 2\alpha} N_{(2)}  (N_{(3)}N_{(4)}N_{(5)})^{2 \alpha + 2 \sigma - 1}}.
\end{equation}

\item Subcase 2. If $N_5 \leq N_4$, we split $n_5$ by Cauchy-Schwarz to get
\begin{align}
    \eqref{eq:No_Pairings_ToBound_RRDDD}
    & \lesssim  \Big(\log \frac{1}{\varepsilon} \Big)^2
    \frac{N_{(1)}^{\varepsilon_1}}{(N_1N_2)^{2 + 2\alpha} (N_3N_4N_5)^{2\alpha + 2\sigma}}
    \max_{X \dot{\cup} Y = \{1,2\}} \big\{
    \lVert  h^b \rVert_{n_Xn_3n_4 \to n_Ynn_5}^2
    \big\}, 
\end{align}
and we get the following bounds:
\begin{enumerate}
    \item Since $\lVert  h^b \rVert_{n_1n_2n_3n_4 \to nn_5}^2
    \lesssim \operatorname{mid}_{1,\{1,2,3,4\}} \operatorname{mid}_{2,\{1,2,3,4\}}^{1 + \varepsilon_1} \operatorname{min}_{\{1,2,3,4\}}^2 N_5$,  we bound by
    \begin{align*}
        \eqref{eq:No_Pairings_ToBound_RRDDD} &\lesssim \Big(\log \frac{1}{\varepsilon} \Big)^2
    \frac{N_{(1)}^{2\varepsilon_1} \operatorname{min}_{\{1,2,3,4\}}}{N_{(1)}^{2 + 2\alpha} \min_{\{1,2\}}^{1 + 2\alpha}(N_3N_4N_5)^{2 \alpha + 2\sigma - 1}} \\
    &\lesssim \Big(\log \frac{1}{\varepsilon} \Big)^2
    \frac{N_{(1)}^{2\varepsilon_1}}{N_{(1)}^{2 + 2\alpha} \min_{\{1,2\}}^{2\alpha}(N_3N_4N_5)^{2 \alpha + 2\sigma - 1}}.
    \end{align*}
    
    \item Given that $\lVert  h^b \rVert_{n_1n_3n_4 \to nn_2n_5}^2
    \lesssim 
    \operatorname{mid}_{\{1,3,4\}} \min_{\{1,3,4\}}^{1+\varepsilon_1} \, N_2 N_5^{1+\varepsilon_1}$, we get
    \begin{align*}
        \eqref{eq:No_Pairings_ToBound_RRDDD} \lesssim \Big(\log \frac{1}{\varepsilon} \Big)^2 \frac{N_{(1)}^{3\varepsilon_1}}{N_{(1)}^{1 + 2\alpha} \min_{\{1,2\}}^{1 + 2\alpha}\max_{\{1,3,4\}}(N_3N_4N_5)^{2 \alpha + 2\sigma - 1}}, 
    \end{align*}
    and observe that, since $N_5 \leq N_4$, 
    the factor $\max_{\{1,3,4\}}$ is either $N_{(1)}$ or $N_{(2)}$. 
    
    \item Since $\lVert  h^b \rVert_{n_2n_3n_4 \to nn_1n_5}^2
    \lesssim 
    \operatorname{mid}_{\{2,3,4\}} \operatorname{min}_{\{2,3,4\}}^{1+\varepsilon_1} \, N_1 N_5^{1+\varepsilon_1}$, we bound as
    \begin{align*}
        \eqref{eq:No_Pairings_ToBound_RRDDD} &\lesssim \Big(\log \frac{1}{\varepsilon} \Big)^2 \frac{N_{(1)}^{3\varepsilon_1}}{ (N_1 N_2)^{1 + 2\alpha} \max_{\{2,3,4\}}  (N_3N_4N_5)^{2 \alpha + 2\sigma - 1}},
    \end{align*}
    and observe that $N_5 \leq N_4$ implies that $\operatorname{max}_{\{2,3,4\}}$ is either $N_{(1)}$ or $N_{(2)}$. 
    
    \item Given that $\lVert  h^b \rVert_{n_3n_4 \to nn_1n_2n_5}^2
    \lesssim
    \operatorname{min}_{\{3,4\}} \, \operatorname{max}_{\{1,2,5\}} \operatorname{mid}_{\{1,2,5\}}^{1+\varepsilon_1} \operatorname{min}_{\{1,2,5\}}^2$, 
    we get 
    \begin{align*}
        \eqref{eq:No_Pairings_ToBound_RRDDD} \lesssim \Big(\log \frac{1}{\varepsilon} \Big)^2 \frac{N_{(1)}^{2\varepsilon_1} N_{(5)}}{ N_{(1)}^{1+2\alpha} \min_{\{1,2\}}^{1+2\alpha}\operatorname{max}_{\{3,4\}} \, (N_3N_4N_5)^{2 \alpha + 2\sigma - 1}}, 
    \end{align*}
    because $\min_{\{1,2,5\}} = N_{(5)}$. 
    Observe also that $\operatorname{min}_{\{1,2\}}\operatorname{max}_{\{3,4\}}$ contains a factor of $N_{(2)}$.
    Therefore, 
    \begin{align*}
        \eqref{eq:No_Pairings_ToBound_RRDDD} \lesssim \Big(\log \frac{1}{\varepsilon} \Big)^2 \frac{N_{(1)}^{2\varepsilon_1} }{ N_{(1)}^{1+2\alpha} N_{(2)} \min_{\{1,2\}}^{2\alpha}\, (N_3N_4N_5)^{2 \alpha + 2\sigma - 1}}.
    \end{align*}
\end{enumerate}
Summing up, in all cases with probability $\geq 1 - \varepsilon$ we get the bound
\begin{equation}
    \eqref{eq:No_Pairings_ToBound_RRDDD} 
    \lesssim \Big(\log \frac{1}{\varepsilon} \Big)^2  \frac{N_{(1)}^{3\varepsilon_1}}{N_{(1)}^{1 + 2\alpha} N_{(2)} (N_{(3)}N_{(4)}N_{(5)})^{2 \alpha + 2\sigma - 1}}.
\end{equation}
\end{itemize}

\subsubsection{Case $RDRDD$.}
In this case we have
\begin{align}\label{eq:No_Pairings_ToBound_RDRDD}
    \sum_{|n| \simeq N}  \big| \mathcal{Q}_1(u_1, \ldots, u_5)_n \big|^2
    = \sum_n \Bigg| \sum_{n_1, \ldots, n_5 } h^b_1 \frac{g_{n_1} g_{n_3}}{(\langle n_1 \rangle \langle n_3 \rangle)^{1+\alpha}} \overline{D_{n_2}} \, \overline{D_{n_4}} D_{n_5} \Bigg|^2.
\end{align}
By symmetry, we assume that $N_4 \leq N_2$ and $N_{(1)} = N_1$. 
We proceed as for \eqref{eq:No_Pairings_ToBound_RRDDD}:

\begin{itemize}

\item Subcase 1.
 If $N_4 \leq N_5$, splitting $n_4$ by Cauchy-Schwarz and using Theorem \ref{thm:Random_Tensor} we get
\begin{align}
    \eqref{eq:No_Pairings_ToBound_RDRDD} 
    & \lesssim \Big(\log \frac{1}{\varepsilon} \Big)^2 \frac{N_{1}^{\varepsilon_1}}{(N_1N_3)^{2 +2\alpha}\, (N_2N_4N_5)^{2 \alpha + 2\sigma}}
    \max_{X \dot{\cup} Y = \{1,3\} } \big\{
    \lVert  h^b_1 \rVert_{n_Xn_2n_5 \to n_Ynn_4}^2
    \big\}.
\end{align}
Again by the Schur test and counting Lemma \ref{lemma:combest-nopairings-cubic}:
\hfill \break

\begin{enumerate}
    \item Since $\lVert  h^b_1 \rVert_{n_1n_2n_3n_5 \to nn_4}^2
    \lesssim \operatorname{max}_{\{2,3,5\}} \operatorname{mid}_{\{2,3,5\}}^{1+\varepsilon_1} \operatorname{min}_{\{2,3,5\}}^2 N_4^{\varepsilon_1}$, we get
    \begin{align}
    \eqref{eq:No_Pairings_ToBound_RDRDD} 
    \lesssim \Big(\log \frac{1}{\varepsilon} \Big)^2 \frac{N_1^{3\varepsilon_1}}{N_1^{2 +2\alpha}\,  N_4^{2 \alpha + 2\sigma} N_3^{2\alpha} (N_2N_5)^{2 \alpha + 2 \sigma - 1}}.
\end{align}
\item Since $\lVert  h^b_1 \rVert_{n_1n_2n_5 \to nn_3n_4}^2
    \lesssim 
    N_2N_5^{1+\varepsilon_1} \, N_3 N_4^{1+\varepsilon_1}$, we get
\begin{align}
    \eqref{eq:No_Pairings_ToBound_RDRDD} \lesssim \Big(\log \frac{1}{\varepsilon} \Big)^2  \frac{N_1^{3\varepsilon_1}}{N_1^{2 +2\alpha}\, N_3^{1 +2\alpha} (N_2N_4N_5)^{2 \alpha + 2 \sigma - 1}}.
\end{align}
\item Since $\lVert  h^b_1 \rVert_{n_2n_3n_5 \to nn_1n_4}^2
    \lesssim
    \operatorname{mid}_{\{2,3,5\}} \operatorname{min}_{\{2,3,5\}}^{1+\varepsilon_1} \, N_1 N_4^{1+\varepsilon_1}$ and $\operatorname{max}_{\{2,3,5\}} = N_{(2)}$,  we get
\begin{align}
    \eqref{eq:No_Pairings_ToBound_RDRDD} \lesssim \Big(\log \frac{1}{\varepsilon} \Big)^2
    \lesssim
    \frac{N_1^{3\varepsilon_1}}{N_1^{1 +2\alpha}\, N_{(2)} N_3^{1 +2\alpha} (N_2N_4N_5)^{2 \alpha + 2 \sigma - 1}}.
\end{align}
\item Since $\lVert  h^b_1 \rVert_{n_2n_5 \to nn_1n_3n_4}^2
    \lesssim
    \operatorname{min}_{\{2,5\}} \, N_1 \operatorname{max}_{\{3,4\}}^{1+\varepsilon_1} \operatorname{min}_{\{3,4\}}^2$, 
we get
\begin{align}
    \eqref{eq:No_Pairings_ToBound_RDRDD} &\lesssim \Big(\log \frac{1}{\varepsilon} \Big)^2
    \frac{N_1^{2\varepsilon_1} \, \min_{\{3,4\}}}{(N_1N_3)^{1 +2\alpha}\, \max_{\{2,5\}} \,  (N_2N_4N_5)^{2 \alpha + 2 \sigma - 1}} \\
    &\lesssim \Big(\log \frac{1}{\varepsilon} \Big)^2
    \frac{N_1^{2\varepsilon_1} }{N_1^{1 +2\alpha} N_{(2)} N_3^{2\alpha}\, (N_2N_4N_5)^{2 \alpha + 2 \sigma - 1}},
\end{align}
because $\min_{\{3,4\}} = N_{(5)}$, and $N_3\max_{\{2,5\}}$ contains a factor $N_{(2)}$.
\end{enumerate}
\hfill \break
In all cases, we have with probability at least $1 - \varepsilon$ that
\begin{align}
    \eqref{eq:No_Pairings_ToBound_RDRDD} 
    \lesssim \Big(\log \frac{1}{\varepsilon} \Big)^2 \frac{N_{1}^{3\varepsilon_1} \, }{N_1^{1 +2\alpha}\, N_{(2)} \,   (N_2N_4N_5)^{2 \alpha + 2 \sigma - 1}}.
\end{align}

\item Subcase 2. 
If $N_5 \leq N_4$, we split $n_5$ with Cauchy-Schwarz and proceed similarly to get 
    \begin{align*}
    \eqref{eq:No_Pairings_ToBound_RDRDD} 
    & \lesssim \Big(\log \frac{1}{\varepsilon} \Big)^2 \frac{N_1^{\varepsilon_1}}{(N_1N_3)^{2 +2\alpha}\, (N_2N_4N_5)^{2 \alpha + 2\sigma}}
    \max_{X \dot{\cup} Y = \{ 1,3\}} \big\{
    \lVert  h^b_1 \rVert_{n_Xn_2n_4 \to n_Ynn_5}^2
    \big\}.
\end{align*}
By the Schur test and counting Lemma \ref{lemma:combest-nopairings-cubic}, we have the following possibilities:
\hfill \break 
\begin{enumerate}
    \item Given that $\lVert  h^b_1 \rVert_{n_1n_2n_3n_4 \to nn_5}^2
    \lesssim \operatorname{max}_{\{2,3,4\}} \operatorname{mid}_{\{2,3,4\}}^{1+\varepsilon_1} \operatorname{min}_{\{2,3,4\}}^2 N_5$, 
    we get
    \begin{align*}
    \eqref{eq:No_Pairings_ToBound_RDRDD} 
    \lesssim \Big(\log \frac{1}{\varepsilon} \Big)^2 \frac{N_1^{2\varepsilon_1}}{N_1^{2 + 2 \alpha} N_3^{2\alpha} (N_2 N_4 N_5)^{2 \alpha + 2\sigma - 1} }.
    \end{align*}

    \item Since $\lVert  h^b_1 \rVert_{n_1n_2n_4 \to nn_3n_5}^2
    \lesssim 
    N_2N_4^{1+\varepsilon_1} \, N_3 N_5^{1+\varepsilon_1}$, we get
    \begin{align*}
    \eqref{eq:No_Pairings_ToBound_RDRDD} 
    \lesssim \Big(\log \frac{1}{\varepsilon} \Big)^2 \frac{N_1^{3\varepsilon_1}}{N_1^{2 + 2\alpha} N_3^{1 + 2\alpha} (N_2 N_4 N_5)^{2 \alpha + 2\sigma - 1} }.
    \end{align*}

    \item Since $\lVert  h^b \rVert_{n_2n_3n_4 \to nn_1n_5}^2
    \lesssim
    \operatorname{mid}_{\{2,3,4\}} \operatorname{min}_{\{2,3,4\}}^{1+\varepsilon_1} \, N_1 N_5^{1+\varepsilon_1}$, 
    we get
    \begin{align*}
    \eqref{eq:No_Pairings_ToBound_RDRDD} 
    \lesssim  \Big(\log \frac{1}{\varepsilon} \Big)^2  \frac{N_1^{3\varepsilon_1}}{N_1^{1 + 2 \alpha} N_{(2)} \, N_3^{1+2\alpha}  (N_2N_4 N_5)^{2 \alpha + 2\sigma - 1} },
    \end{align*}
    because $N_5 \leq N_4 \leq N_2$ implies that $N_{(2)} = \operatorname{max}_{\{2,3,4\}}$.  

    \item Given that $\lVert  h^b \rVert_{n_2n_4 \to nn_1n_3n_5}^2
    \lesssim
    N_4^{\varepsilon_1} \, N_1 \operatorname{max}_{\{3,5\}}^{1+\varepsilon_1} \operatorname{min}_{\{3,5\}}^2$, 
    we get 
    \begin{align}
    \eqref{eq:No_Pairings_ToBound_RDRDD} \lesssim \Big(\log \frac{1}{\varepsilon} \Big)^2 \frac{N_1^{3\varepsilon_1} \, \operatorname{min}_{\{3,5\}}}{(N_1N_3)^{1+2\alpha} (N_2N_4)^{2\alpha + 2\sigma} N_5^{2 \alpha + 2\sigma - 1}}.
    \end{align}
    Observe that $\min_{\{3,5\}} = N_{(5)}$, and that $N_2N_3N_4$ contains the factors $N_{(2)}$ and $N_{(3)}$ because $N_5 \leq N_4 \leq N_2$. Hence,  
    \begin{align}
    \eqref{eq:No_Pairings_ToBound_RDRDD} \lesssim \Big(\log \frac{1}{\varepsilon} \Big)^2 \frac{N_1^{3\varepsilon_1} }{N_1^{1+2\alpha} N_{(2)} N_{(3)}  N_3^{2\alpha}  (N_2N_4N_5)^{2 \alpha + 2\sigma - 1}}.
    \end{align}
\end{enumerate}
In all cases, with probability $\geq 1 - \varepsilon$ we have
\begin{equation}
    \eqref{eq:No_Pairings_ToBound_RDRDD} 
    \lesssim \Big(\log \frac{1}{\varepsilon} \Big)^2 \frac{N_1^{3\varepsilon_1}}{N_1^{1+2\alpha} N_{(2)} (N_{(3)}N_{(4)} N_{(5)})^{2 \alpha + 2\sigma - 1}}.
\end{equation}

\end{itemize}

\subsubsection{Case $DRDRD$.} This last case with two $R$ is proved in an analogous way. It is even simpler since $N_{(1)} = N_2$ and by symmetry we can assume $N_1 \geq N_3 \geq N_5$, so we separate $n_5$ by Cauchy-Schwartz and we get the same bound as in previous cases.

\subsubsection{Cases $RRRDD$, $RRDRD$ and $RDRDR$.}
By symmetry, these are the cases with three $R$. 

\begin{itemize}
    \item Case $RRRDD$. We have
\begin{align}\label{eq:No_Pairings_ToBound_RRRDD}
    \sum_{|n| \simeq N}  \big| \mathcal{Q}_1(u_1, \ldots, u_5)_n \big|^2
    = \sum_n 
    \Bigg| \sum_{n_1, \ldots, n_5 } h^b_1 
    \frac{g_{n_1} \overline{g_{n_2}} g_{n_3}}{ (\langle n_1 \rangle \langle n_2 \rangle\langle n_3 \rangle)^{1+\alpha}} \overline{D_{n_4}} D_{n_5} \Bigg|^2, 
\end{align}
where by symmetry between $n_1$ and $n_3$ we may assume that $N_{(1)} = \max_{\{1,2\}}$. 
Directly bounding it with the random tensor estimate in Theorem~\ref{thm:Random_Tensor}, we get
\begin{align}
    \eqref{eq:No_Pairings_ToBound_RRRDD}
    & \lesssim \bigg\lVert \sum_{n_1, n_2, n_3} h^b_1 
    \frac{g_{n_1} \overline{g_{n_2}} g_{n_3}}{ (\langle n_1 \rangle \langle n_2 \rangle\langle n_3 \rangle)^{1+\alpha}} \bigg\rVert_{n_4n_5 \to n}^2 \lVert D_{n_4} \rVert_{n_4}^2 \lVert D_{n_5} \rVert_{n_5}^2 \\
    & \lesssim \Big(\log \frac{1}{\varepsilon} \Big)^3  \frac{N_{(1)}^{\varepsilon_1}}{(N_1N_2N_3)^{2 + 2\alpha} (N_4N_5)^{2\alpha + 2\sigma}} \, 
    \max_{X \dot\cup Y = \{1,2,3\} } \lVert h^b \rVert_{n_4n_5n_X \to nn_Y}^2.
\end{align}
There are eight partitions of $\{1,2,3\}$, each of which we treat separately by the Schur test. We write the following estimates up to the $(\log 1/\varepsilon)^3$ term.  
\hfill \break

\begin{enumerate}
    \item Since $\lVert h^b_1 \rVert_{n_1n_2n_3n_4n_5 \to n}^2 \lesssim N_{(2)}N_{(3)}^{1+\varepsilon_1}N_{(4)}^2 N_{(5)}^2$, we get 
    \begin{align}
    \eqref{eq:No_Pairings_ToBound_RRRDD} 
    \lesssim 
    \frac{N_{(1)}^{2\varepsilon_1} N_{(4)} N_{(5)}}{N_{(1)} (N_1N_2N_3)^{1+ 2\alpha} (N_4N_5)^{2\alpha + 2\sigma - 1 }} 
    \lesssim 
    \frac{N_{(1)}^{2\varepsilon_1}}{N_{(1)}^{2 + 2\alpha} (\min_{\{1,2\}} N_3)^{2\alpha} (N_4N_5)^{2\alpha + 2\sigma - 1 }}.
    \end{align}

    \item Since $\lVert h^b_1 \rVert_{n_2n_3n_4n_5 \to nn_1}^2  \lesssim 
    \operatorname{mid}_{1,\{2,3,4,5\}} \operatorname{mid}_{2,\{2,3,4,5\}}^{1+\varepsilon_1} \operatorname{min}_{\{2,3,4,5\}}^2
    \cdot N_1$, we get 
    \begin{align}
        \eqref{eq:No_Pairings_ToBound_RRRDD} 
        \lesssim  
        \frac{N_{(1)}^{2\varepsilon_1} \operatorname{min}_{\{2,3,4,5\}} }{(N_1 N_2 N_3)^{1+2\alpha} \operatorname{max}_{\{2,3,4,5\}}(N_4N_5)^{2\alpha + 2\sigma - 1}}
        \leq 
        \frac{N_{(1)}^{2\varepsilon_1} }{N_{(1)}^{1+2\alpha} N_{(2)} \min_{\{1,2\}}^{1+2\alpha} N_3^{2\alpha} (N_4N_5)^{2\alpha + 2\sigma - 1}}
    \end{align}
    because $\operatorname{max}_{\{2,3,4,5\}} \geq N_{(2)}$.

    \item  Since $\lVert h^b_1 \rVert_{n_1n_3n_4n_5 \to nn_2}^2 \lesssim 
    \operatorname{mid}_{1,\{1,3,4,5\}} \operatorname{mid}_{2,\{1,3,4,5\}}^{1+\varepsilon_1} \operatorname{min}_{\{1,3,4,5\}}^2
    \cdot N_2^{\varepsilon_1}$, we get 
    \begin{align}
        \eqref{eq:No_Pairings_ToBound_RRRDD}
        \lesssim 
        \frac{N_{(1)}^{3\varepsilon_1} \operatorname{min}_{\{1,3,4,5\}}}{ N_2^{2 + 2\alpha} (N_1N_3)^{1 + 2\alpha} \operatorname{max}_{\{1,3,4,5\}} (N_4N_5)^{2\alpha + 2\sigma - 1 }}
        \leq 
        \frac{N_{(1)}^{3\varepsilon_1}}{ N_{(1)}^{1+2\alpha} N_{(2)} \min_{\{1,2\}}^{1+2\alpha} N_2 N_3^{2\alpha} (N_4N_5)^{2\alpha + 2\sigma - 1}}.
    \end{align}
    because $\operatorname{max}_{\{1,3,4,5\}} \geq N_{(2)}$.
    
    \item Since $\lVert h^b_1 \rVert_{n_1n_2n_4n_5 \to nn_3}^2 \lesssim 
    \operatorname{mid}_{1,\{1,2,4,5\}} \operatorname{mid}_{2,\{1,2,4,5\}}^{1+\varepsilon_1} \operatorname{min}_{\{1,2,4,5\}}^2
    \cdot N_3$, we get 
    \begin{align}
        \eqref{eq:No_Pairings_ToBound_RRRDD}
        \lesssim 
        \frac{ N_{(1)}^{2\varepsilon_1} \operatorname{min}_{\{1,2,4,5\}} }{(N_1N_2N_3)^{1 + 2\alpha} \operatorname{max}_{\{1,2,4,5\}}(N_4N_5)^{2\alpha + 2\sigma - 1}}
        \leq 
        \frac{ N_{(1)}^{2\varepsilon_1} }{N_{(1)}^{1+2\alpha} N_{(2)}  N_3^{1 + 2\alpha} \min_{\{1,2\}}^{2\alpha}(N_4N_5)^{2\alpha + 2\sigma - 1 }}
    \end{align}
    because $\operatorname{max}_{\{1,2,4,5\}} \geq N_{(2)}$.

    \item Since $\lVert h^b_1 \rVert_{n_3n_4n_5 \to nn_1n_2}^2 \lesssim \operatorname{mid}_{\{3,4,5\}} \, \operatorname{min}_{\{3,4,5\}}^{1+\varepsilon_1} \, N_1N_2^{1+\varepsilon_1}$, we get 
    \begin{align}
        \eqref{eq:No_Pairings_ToBound_RRRDD}
        \lesssim 
        \frac{N_{(1)}^{3\varepsilon_1}}{(N_1N_2N_3)^{1+2\alpha} \max_{\{3,4,5\}} \, (N_4N_5)^{2\alpha + 2\sigma - 1}}
        \leq 
        \frac{N_{(1)}^{3\varepsilon_1}}{(N_{(1)} \min_{\{1,2\}} N_3)^{1+2\alpha} \max_{\{3,4,5\}} \, (N_4N_5)^{2\alpha + 2\sigma - 1}},
    \end{align}
    and observe that $\min_{\{1,2\}}N_3 \max_{\{3,4,5\}}$ contains a factor $N_{(2)}$.

    \item Since $\lVert h^b_1 \rVert_{n_2n_4n_5 \to nn_1n_3}^2 \lesssim \operatorname{mid}_{\{2,4,5\}} \, \operatorname{min}_{\{2,4,5\}}^{1+\varepsilon_1} \, N_1N_3^{1+\varepsilon_1}$, we get 
    \begin{align}
        \eqref{eq:No_Pairings_ToBound_RRRDD}
        \lesssim 
        \frac{ N_{(1)}^{3\varepsilon_1} }{(N_1N_2N_3)^{1 + 2\alpha} \max_{\{2,4,5\}} (N_4N_5)^{2\alpha + 2\sigma - 1}}
        \leq \frac{ N_{(1)}^{3\varepsilon_1} }{(N_{(1)} \min_{\{1,2\}}N_3)^{1 + 2\alpha} \max_{\{2,4,5\}} (N_4N_5)^{2\alpha + 2\sigma - 1}},  
    \end{align}
    and observe that $\min_{\{1,2\}}N_3 \max_{\{2,4,5\}}$ also contains a factor of $N_{(2)}$.

    \item Since $\lVert h^b_1 \rVert_{n_1n_4n_5 \to nn_2n_3}^2 \lesssim \operatorname{mid}_{\{1,4,5\}} \operatorname{min}_{\{1,4,5\}}^{1+\varepsilon_1} N_2 N_3^{1+\varepsilon_1}$, we get 
    \begin{align}
        \eqref{eq:No_Pairings_ToBound_RRRDD}
        \lesssim 
        \frac{N_{(1)}^{3\varepsilon_1}}{ (N_1N_2N_3)^{1 + 2\alpha} \max_{\{1,4,5\}} (N_4N_5)^{2\alpha + 2\sigma - 1}}
        =
        \frac{N_{(1)}^{3\varepsilon_1}}{ (N_{(1)} \min_{\{1,2\}} N_3)^{1 + 2\alpha} \max_{\{1,4,5\}} (N_4N_5)^{2\alpha + 2\sigma - 1}}, 
    \end{align}
    and $\min_{\{1,2\}}N_3 \max_{\{1,4,5\}}$ also contains a factor of $N_{(2)}$.

    \item Since $\lVert h^b_1 \rVert_{n_4n_5 \to nn_1n_2n_3}^2 \lesssim \operatorname{min}_{\{4,5\}}  \cdot \operatorname{max}_{\{1,2,3\}} \operatorname{mid}_{\{1,2,3\}}^{1+\varepsilon_1} \operatorname{min}_{\{1,2,3\}}^2$, we get 
    \begin{align}
        \eqref{eq:No_Pairings_ToBound_RRRDD}
        & \lesssim 
        \frac{ N_{(1)}^{2\varepsilon_1} \operatorname{min}_{\{1,2,3\}} }{(N_1 N_2 N_3)^{1 + 2\alpha} \max_{\{4,5\}}(N_4N_5)^{2\alpha + 2\sigma - 1}}
        \\
        & = \frac{ N_{(1)}^{2\varepsilon_1} }{(N_{(1)} \max(\min_{\{1,2\}}, N_3))^{1 + 2\alpha} \max_{\{4,5\}} (N_4N_5)^{2\alpha + 2\sigma - 1}}, 
    \end{align}
    and observe that $\max(\min_{\{1,2\}}, N_3) \max_{\{4,5\}}$ contains $N_{(2)}$ too. 
\end{enumerate}
\hfill \break
In all cases, with probability $\geq 1 - \varepsilon$ we get at least the bound 
\begin{equation}
    \eqref{eq:No_Pairings_ToBound_RRRDD}
        \lesssim \Big(\log \frac{1}{\varepsilon} \Big)^3 \frac{N_{(1)}^{3\varepsilon_1}}{N_{(1)}^{1+2\alpha} N_{(2)} \, (N_{(3)}N_{(4)}N_{(5)})^{2\alpha + 2\sigma - 1}}.
\end{equation}
\item Case $RRDRD$.  In this case we have
\begin{align}\label{eq:No_Pairings_ToBound_RRDRD}
    \sum_{|n| \simeq N}  \big| \mathcal{Q}_1(u_1, \ldots, u_5)_n \big|^2
    = \sum_n 
    \Bigg| \sum_{n_1, \ldots, n_5 } h^b_1 
    \frac{g_{n_1} \overline{g_{n_2}} \, \overline{g_{n_4}}}{ (\langle n_1 \rangle \langle n_2 \rangle\langle n_4 \rangle)^{1+\alpha}} D_{n_3} D_{n_5} \Bigg|^2,
\end{align}
which as in the previous case we bound directly with the operator norm by
\begin{align}
    \eqref{eq:No_Pairings_ToBound_RRDRD}
    \lesssim \Big(\log \frac{1}{\varepsilon} \Big)^3
    \frac{N_{(1)}^{\varepsilon_1}}{(N_1N_2N_4)^{2 + 2\alpha} (N_3N_5)^{2\alpha + 2\sigma}} \, 
    \max_{X \dot\cup Y = \{1,2,4\} } \lVert h^b \rVert_{n_3n_5n_X \to nn_Y}^2, 
\end{align}
with probability $\geq 1 - \varepsilon$.
There are eight partitions of $\{1,2,4\}$ which, as before, are addressed using the Schur test and counting Lemma \ref{lemma:combest-nopairings-cubic}. 
We skip computations since they are analogous to the previous case $RRRDD$ (the role of $n_3$  now being played by $n_4$, which actually makes some of the counting more favorable).
We get the same bound
\begin{equation}
    \eqref{eq:No_Pairings_ToBound_RRDRD}
        \lesssim \Big(\log \frac{1}{\varepsilon} \Big)^3 \frac{N_{(1)}^{3\varepsilon_1}}{N_{(1)}^{1+2\alpha} N_{(2)} \, (N_{(3)} N_{(4)} N_{(5)})^{2\alpha + 2\sigma - 1}}. 
\end{equation}

\item Case $RDRDR$. For this last case, the same procedure gives, with probability at least $1-\varepsilon$, 
\begin{align}\label{eq:No_Pairings_ToBound_RDRDR}
    \sum_{|n| \simeq N}  \big| \mathcal{Q}_1(u_1, \ldots, u_5)_n \big|^2
    & = \sum_n 
    \Bigg| \sum_{n_1,n_2,n_3,n_4,n_5 } h^b_1 
    \frac{g_{n_1} g_{n_3} g_{n_5}}{ (\langle n_1 \rangle \langle n_3 \rangle\langle n_5 \rangle)^{1+\alpha}}\overline{D_{n_2}} \,  \overline{D_{n_4}}  \Bigg|^2 \\
    & \lesssim \Big(\log \frac{1}{\varepsilon} \Big)^3 \frac{N_{(1)}^{3\varepsilon_1}}{N_{(1)}^{1+2\alpha} N_{(2)}  \, (N_{(3)} N_{(4)} N_{(5)})^{2\alpha + 2\sigma - 1}}.
\end{align}
\end{itemize}

\subsubsection{Cases $RRRRD$ and $RRRDR$.}
Both cases being analogous, we only compute explicitly
\begin{align}\label{eq:No_Pairings_ToBound_RRRRD}
    \sum_{|n| \simeq N}  \big| \mathcal{Q}_1(u_1, \ldots, u_5)_n \big|^2 
    = \sum_n 
    \Bigg| \sum_{n_1, \ldots, n_5 } h^b_1 
    \frac{g_{n_1} \overline{g_{n_2}} g_{n_3} \overline{g_{n_4}} }{ (\langle n_1 \rangle \langle n_2 \rangle\langle n_3 \rangle \langle n_4 \rangle)^{1+\alpha}} D_{n_5} \Bigg|^2.
\end{align}
By symmetry, we reduce to $N_{(1)} = \max_{\{1,2\}}$. By the random tensor estimate in Theorem \ref{thm:Random_Tensor},
    \begin{align*}
        \eqref{eq:No_Pairings_ToBound_RRRRD} 
        \lesssim \Big(\log \frac{1}{\varepsilon} \Big)^4 \frac{N_{(1)}^{\varepsilon_1}}{ (N_1N_2N_3N_4)^{2 + 2\alpha} N_5^{2\alpha + 2\sigma}} \max_{X \dot{\cup} Y = \{1,2,3,4\}} \|h^b_1\|_{n_5n_X \rightarrow nn_Y}^2, 
    \end{align*}
    where the maximum is taken among the following estimates for the operator norms: 
    \begin{align*}
    &\lVert  h^b \rVert_{n_1n_2n_3n_4n_5 \to n}^2 \lesssim N_{(2)} N_{(3)}^{1 + \varepsilon_1} N_{(4)}^2 N_{(5)}^2  \\
    &\lVert  h^b \rVert_{n_2n_3n_4n_5 \to nn_1}^2 \lesssim  \operatorname{mid}_{1,\{2,3,4,5\}} \operatorname{mid}_{2,\{2,3,4,5\}}^{1+\varepsilon_1}  \operatorname{max}_{\{2,3,4,5\}}^2 \cdot  N_1, \\
    &\lVert  h^b \rVert_{n_1n_2n_3n_5 \to nn_4}^2 \lesssim \text{mid}_{1,\{1,2,3,5\}} \text{mid}_{2,\{1,2,3,5\}}^{1+\varepsilon_1} \text{min}_{\{1,2,3,5\}}^2  \cdot N_4^{\varepsilon_1}, \\
    &\lVert  h^b \rVert_{n_1n_2n_4n_5 \to nn_3}^2 \lesssim \text{mid}_{1,\{1,2,4,5\}} \text{mid}_{2,\{1,2,4,5\}}^{1 + \varepsilon_1} \text{min}_{\{1,2,4,5\}}^2  \cdot N_3, \\
    & \lVert  h^b \rVert_{n_3n_4n_5 \to nn_1n_2}^2 \lesssim  \text{mid}_{\{3,4,5\}} \text{min}_{\{3,4,5\}}^{1+\varepsilon_1} \cdot N_1 N_2^{1 + \varepsilon_1}, \\
    & \lVert  h^b \rVert_{n_1n_4n_5 \to nn_2n_3}^2 \lesssim  \text{mid}_{\{1,4,5\}} \text{min}_{\{1,4,5\}}^{1 + \varepsilon_1} \cdot \text{max}_{\{2,3\}} \text{min}_{\{2,3\}}^{1 + \varepsilon_1}, \\
    & \lVert  h^b \rVert_{n_1n_2n_5 \to nn_3n_4}^2 \lesssim  \text{mid}_{\{1,2,5\}} \text{min}_{\{1,2,5\}}^{1 + \varepsilon_1} \cdot  \text{max}_{\{3,4\}} \text{min}_{\{3,4\}}^{1 + \varepsilon_1}, \\
    & \lVert  h^b \rVert_{n_1n_3n_5 \to nn_2n_4}^2 \lesssim  \operatorname{mid}_{\{1,3,5\}}^{\varepsilon_1} \operatorname{min}_{\{1,3,5\}}^2 \cdot  \operatorname{max}_{\{2,4\}}^{\varepsilon_1} \operatorname{min}_{\{2,4\}}^2,  \\
    & \lVert  h^b \rVert_{n_2n_4n_5 \to nn_1n_3}^2 \lesssim  \text{mid}_{\{2,4,5\}} \text{min}_{\{2,4,5\}}^{1 + \varepsilon_1}  \cdot \operatorname{max}_{\{1,3\}} \operatorname{min}_{\{1,3\}}^{1 + \varepsilon_1}, \\
    & \lVert  h^b \rVert_{n_2n_3n_5 \to nn_1n_4}^2 \lesssim  \text{mid}_{\{2,3,5\}} \text{min}_{\{2,3,5\}}^{1 + \varepsilon_1} \cdot \operatorname{max}_{\{1,4\}} \operatorname{min}_{\{1,4\}}^{1 + \varepsilon_1}, \\
    & \lVert  h^b \rVert_{n_1n_5 \to nn_2n_3n_4}^2 \lesssim N_5^{\varepsilon_1} \cdot \text{max}_{\{2,3,4\}}  \text{mid}_{\{2,3,4\}}^{1 + \varepsilon_1} \text{min}_{\{2,3,4\}}^2, \\
    & \lVert  h^b \rVert_{n_2n_5 \to nn_1n_3n_4}^2 \lesssim \operatorname{min}_{\{2,5\}} \cdot \, \operatorname{max}_{\{1,3,4\}} \operatorname{mid}_{\{1,3,4\}}^{1 + \varepsilon_1} \operatorname{min}_{\{1,3,4\}}^2, \\
    & \lVert  h^b \rVert_{n_3n_5 \to nn_1n_2n_4}^2 \lesssim \text{min}_{\{3,5\}}^{\varepsilon_1} \cdot N_{(1)} \text{mid}_{\{1,2,4\}}^{1 + \varepsilon_1} \text{min}_{\{1,2,4\}}^2, \\
    & \lVert  h^b \rVert_{n_4n_5 \to nn_1n_2n_3}^2 \lesssim \operatorname{min}_{\{4,5\}} \cdot  \, N_{(1)} \text{mid}_{\{1,2,3\}}^{1 + \varepsilon_1}  \text{min}_{\{1,2,3\}}^2  , \\
    & \lVert  h^b \rVert_{n_5 \to nn_1n_2n_3n_4}^2 \lesssim  N_{(1)} \,  \text{mid}_{1,\{1,2,3,4\}}^{1 + \varepsilon_1} \text{mid}_{2,\{1,2,3,4\}}^2 \text{min}_{\{1,2,3,4\}}^2 .
    \end{align*}
In all cases, with probability $\geq 1 - \varepsilon$ we get 
\begin{align*}
     \eqref{eq:No_Pairings_ToBound_RRRRD} \lesssim \Big(\log \frac{1}{\varepsilon} \Big)^4 \frac{N_{(1)}^{3\varepsilon_1}}{N_{(1)}^{1 + 2\alpha} N_{(2)} N_{(5)} \,  (N_{(3)}N_{(4)} N_{(5)})^{2\alpha + 2\sigma - 1}}.
\end{align*}

\subsection{$2$-pairings: $\boldsymbol{\mathcal{Q}_2}$}
\label{sec:Two_Pairings}
Omitting constants, the term 
\begin{equation}
        \mathcal{Q}_2(u)_n = 
     \mathcal W_3(u)_n
    \end{equation}
    is the only one in \eqref{eq:qnl-fcoeff}  arising from 2-pairings. It is intrinsically trilinear
    and it corresponds to the nonlinearity of the Wick-ordered cubic NLS. 
    In this case, we have to bound 
    \begin{equation}
    \label{eq:Two_Pairings_ToBound}
        \sum_{|n| \simeq N}  \big| \mathcal{Q}_2(u_1, u_2, u_3)_n \big|^2
        = \sum_n \Big| \sum_{n_1, n_2, n_3} h_{\mathcal W_3}^b (u_1)_{n_1} \overline{(u_2)_{n_2}} (u_3)_{n_3} \Big|^2, 
    \end{equation}
    such that according to the definition of $\mathcal{A}_2$ in \eqref{eq:Aj},
    the base tensor is given by
    \begin{align*}
        h^b_2 
        = h^b_{\mathcal{W}_3}(n_1,n_2,n_3,n)  \mathbbm 1_{n_4 = n_5} \mathbbm 1_{|n_4| \simeq N_4}, 
    \end{align*}
    and $h^b_{\mathcal{W}_3}$ is the tensor corresponding to the cubic Wick-ordered nonlinearity
    \begin{align}\label{eq:Base_Tensor_Wick}
    h^b_{\mathcal{W}_3} 
    =  
    \mathbbm 1_{n_1 - n_2 + n_3 = n} \,  
    \mathbbm 1_{|n|^2 - |n_1|^2 + |n_2|^2 - |n_3|^2 = \mu} \, \mathbbm 1_{n_1 \neq n_2} \,  \mathbbm 1_{n_3 \neq n_2} \,
    \mathbbm 1_{|n| \simeq N} \, 
    \prod_{k=1}^3 \mathbbm 1_{|n_k| \simeq N_k}.
    \end{align}
    We again separate cases depending on the choices for $u_i \in \{ R, D\}$. 
    \hfill \break

    \noindent
    \textbf{Case RRR.} 
    By the Wick-Isserlis theorem we have 
        \begin{align}
            & \sum_{\substack{n_1,n_2,n_3 \\ n_1', n_2', n_3' }} h^b_{\mathcal{W}_3}(n_1,n_2,n_3,n) h^b_{\mathcal{W}_3}(n_1',n_2',n_3',n) 
            \frac{ \mathbb E\big[ g_{n_1} \overline{g_{n_2}} g_{n_3} \overline{g_{n_1'} } g_{n_2'} \overline{g_{n_3'}} \big]}{ \prod_{j=1}^3 (\langle n_j \rangle \langle n_j' \rangle )^{1+\alpha}} 
            \simeq \frac{1}{(N_1N_2N_3)^{2+2\alpha}} \sum_{n_1,n_2,n_3} h^b_{\mathcal{W}_3}, 
        \end{align}
    and therefore 
    \begin{align}
            \mathbb E [\eqref{eq:Two_Pairings_ToBound}]
            & = \frac{1}{(N_1N_2N_3)^{2+2\alpha}} \sum_n \sum_{n_1,n_2,n_3} h^b_{\mathcal{W}_3} 
            \lesssim  \frac{N_{(1)}N_{(2)}^{1+\varepsilon_1} N_{(3)}^2}{(N_1N_2N_3)^{2+2\alpha}} = \frac{N_{(1)}^{\varepsilon_1}}{(N_{(1)}N_{(2)})^{1+2\alpha}N_{(3)}^{2\alpha}}. 
    \end{align}
    Hence, by Proposition \ref{lemma:wiener-chaos}, with probability $\geq 1-\varepsilon$ we have 
    \begin{align}\label{eq:wiener-app2}
        \eqref{eq:Two_Pairings_ToBound} 
        \lesssim \Big(\log \frac{1}{\varepsilon} \Big)^3 \frac{N_{(1)}^{2\varepsilon_1}}{N_{(1)}^{1+2\alpha}N_{(2)}^{1+2\alpha}N_{(3)}^{2\alpha}}.
    \end{align}

\noindent
\textbf{Case RRD.} By the random tensor estimate in Theorem~\ref{thm:Random_Tensor}, we get
        \begin{align*}
        \eqref{eq:Two_Pairings_ToBound}
        = \sum_n \bigg| \sum_{n_1,n_2,n_3} h^b_{\mathcal{W}_3}    R_{n_1} \overline{R_{n_2}} D_{n_3} \bigg|^2   
        & \leq  \bigg\| \sum_{n_1,n_2} h^b_{\mathcal{W}_3} \frac{g_{n_1} \overline{g_{n_2}}}{(\langle n_1 \rangle \langle n_2 \rangle)^{2 + 2\alpha}} \bigg\|_{n_3 \rightarrow n}^2 \| D_{n_3} \|_{n_3}^2 \\
        & \lesssim \Big(  \log \frac{1}{\varepsilon} \Big)^2 \frac{N_{(1)}^{\varepsilon_1}}{(N_1 N_2)^{2 + 2\alpha} } \frac{1}{N_3^{2\alpha + 2\sigma}} \max_{X \dot\cup Y = \{1,2\}} \lVert h^b_{\mathcal{W}_3} \rVert_{n_3n_X \to nn_Y}^2.
        \end{align*}
        By the Schur test and the counting estimates in Lemma~\ref{lemma:combest-nopairings-cubic}, we have 
        \begin{align}\label{eq:schur-prob-3p}
            \begin{array}{ll}
               \lVert h^b_{\mathcal{W}_3} \rVert_{n_1n_2n_3 \to n}^2
            \lesssim N_{(2)} N_{(3)}^{1+\varepsilon_1},   & \qquad \qquad  \lVert h^b_{\mathcal{W}_3} \rVert_{n_1n_3 \to nn_2}^2 \lesssim  \operatorname{min}_{\{1,3\}}^{\varepsilon_1} \cdot N_2^{\varepsilon_1}, \\
              \lVert h^b_{\mathcal{W}_3} \rVert_{n_2n_3 \to nn_1}^2 \lesssim  \operatorname{min}_{\{2,3\}} \cdot N_1,   &  \qquad \qquad  \lVert h^b_{\mathcal{W}_3} \rVert_{n_3 \to nn_1n_2}^2 \lesssim N_1 N_2^{1+\varepsilon_1}.
            \end{array}
        \end{align}
        The worst case being the last one, with probability $1 - \varepsilon$ we get
        \begin{align*}
            \eqref{eq:Two_Pairings_ToBound}
            \lesssim \Big(\log \frac{1}{\varepsilon} \Big)^2 \frac{N_{(1)}^{2\varepsilon_1}}{(N_1N_2)^{1 + 2\alpha} N_3^{2\alpha + 2\sigma}}
            \lesssim \Big(\log \frac{1}{\varepsilon} \Big)^2 \frac{N_{(1)}^{2\varepsilon_1}}{N_{(1)}N_{(2)}N_{(3)}(N_1N_2)^{2\alpha} N_3^{2\alpha + 2\sigma - 1}}. 
        \end{align*}

    \noindent
    \textbf{Case RDR. }
    Similarly, 
    \begin{align*}
        \eqref{eq:Two_Pairings_ToBound}
        = \sum_n \bigg| \sum_{n_1,n_2,n_3} h^b_{\mathcal{W}_3}    R_{n_1} \overline{D_{n_2}} R_{n_3} \bigg|^2   
        & \lesssim  \Big(  \log \frac{1}{\varepsilon} \Big)^2\frac{  N_{(1)}^{\varepsilon_1}}{ (N_1 N_3)^{2 + 2\alpha} N_2^{2\alpha + 2\sigma}} \max_{X \dot\cup Y = \{1,3\}} \lVert h^b_{\mathcal{W}_3} \rVert_{n_2n_X \to nn_Y}^2.
    \end{align*}
    Like in \eqref{eq:schur-prob-3p}, we have
    \begin{align}\label{eq:schur-prob-3p_2}
    \begin{array}{ll}
        \lVert h^b_{\mathcal{W}_3} \rVert_{n_1n_2n_3 \to n}^2
            \lesssim N_{(2)} N_{(3)}^{1+\varepsilon_1}, & \qquad \qquad \lVert h^b_{\mathcal{W}_3} \rVert_{n_1n_2 \to nn_3}^2 \lesssim \operatorname{min}_{\{1,2\}} \cdot N_3, \\
         \lVert h^b_{\mathcal{W}_3} \rVert_{n_2n_3 \to nn_1}^2 \lesssim    \operatorname{min}_{\{2,3\}} \cdot N_1, & \qquad \qquad \lVert h^b_{\mathcal{W}_3} \rVert_{n_2 \to nn_1n_3}^2 \lesssim N_1 N_3^{1+\varepsilon_1}.
    \end{array}
    \end{align}
    and the worst case being the last one, with probability $1 - \varepsilon$ we get
        \begin{align*}
        \eqref{eq:Two_Pairings_ToBound}
        \lesssim \Big(\log \frac{1}{\varepsilon} \Big)^2 \frac{N_{(1)}^{2\varepsilon_1}}{(N_1N_3)^{1 + 2\alpha} N_2^{2\alpha + 2\sigma} }
        \lesssim \Big(\log \frac{1}{\varepsilon} \Big)^2 \frac{N_{(1)}^{2\varepsilon_1}}{N_{(1)}N_{(2)}N_{(3)}(N_1N_3)^{2\alpha} N_2^{2\alpha + 2\sigma - 1}}.
        \end{align*}

\noindent
\textbf{Case RDD. }  In this case we can assume $N_{(1)} = N_1$ by Proposition \ref{prop:Qsmoothing-mest}. As before, 
    \begin{align}\label{eq:RDD_With_RT}
        \eqref{eq:Two_Pairings_ToBound}
        = \sum_n \bigg| \sum_{n_1,n_2,n_3} h^b_{\mathcal{W}_3} R_{n_1} \overline{D_{n_2}} D_{n_3} \bigg|^2
        \leq  \bigg\| \sum_{n_1}  h^b_{\mathcal{W}_3} \frac{g_{n_1}}{\langle n_1 \rangle^{1 + \alpha}} \bigg\|_{n_2n_3 \rightarrow n}^2 \| D_{n_2} \|_{n_2}^2 \| D_{n_3} \|_{n_3}^2. 
    \end{align}
    Since by the Schur test and the counting estimates in Lemmata \ref{lemma:schur-test} and \ref{lemma:combest-nopairings-cubic} we have 
    \begin{align}
        \max \{ \lVert h^b_{\mathcal{W}_3} \rVert_{n_1n_2n_3 \to n}^2, 
        \lVert h^b_{\mathcal{W}_3} \rVert_{n_2n_3 \to nn_1}^2 \}
        \leq 
        \max\{ N_{(2)}N_{(3)}^{1+\varepsilon_1}, N_1 N_{(3)} \}
        \leq N_{(1)} N_{(3)}^{1+\varepsilon_1}, 
    \end{align}
    the random tensor estimate in Theorem \ref{thm:Random_Tensor} implies that with probability $1 - \varepsilon$ we have
    \begin{align*}
        \eqref{eq:Two_Pairings_ToBound}
        & \lesssim 
        \Big(\log \frac{1}{\varepsilon} \Big) \frac{N_{(1)}^{\varepsilon_1}}{N_1^{2 + 2\alpha} (N_2 N_3)^{2\alpha + 2\sigma}} \max \{ \lVert h^b_{\mathcal{W}_3} \rVert_{n_1n_2n_3 \to n}^2, 
        \lVert h^b_{\mathcal{W}_3} \rVert_{n_2n_3 \to nn_1}^2 \} \\
        & \lesssim \Big(\log \frac{1}{\varepsilon} \Big) \frac{N_{(1)}^{2\varepsilon_1}}{N_{(1)}^{1 + 2\alpha} N_{(2)} (N_2 N_3)^{2\alpha + 2\sigma - 1}}.
    \end{align*}

\noindent\textbf{Case DRD. }  Now we can assume $N_{(1)} = N_2$ by Proposition \ref{prop:Qsmoothing-mest}. Again,
    \begin{align}\label{eq:DRD_With_RT}
        \eqref{eq:Two_Pairings_ToBound}
        = \sum_n \bigg| \sum_{n_1,n_2,n_3} h^b_{\mathcal{W}_3} D_{n_1} \overline{R_{n_2}} D_{n_3} \bigg|^2
        \leq  \bigg\| \sum_{n_2}  h^b_{\mathcal{W}_3} \frac{\overline{g_{n_2}}}{\langle n_2 \rangle^{1 + \alpha}} \bigg\|_{n_1n_3 \rightarrow n}^2 \| D_{n_1} \|_{n_1}^2 \| D_{n_3} \|_{n_3}^2. 
    \end{align}
    Proceeding as in the previous case, since
    \begin{align}
        \max \{ \lVert h^b_{\mathcal{W}_3} \rVert_{n_1n_2n_3 \to n}^2, 
        \lVert h^b_{\mathcal{W}_3} \rVert_{n_1n_3 \to nn_2}^2 \}
        \leq 
        \max\{ N_{(2)}N_{(3)}^{1+\varepsilon_1}, N_{(3)}^{\varepsilon_1} N_2^{\varepsilon_1} \}
        \leq N_{(1)}^{\varepsilon_1} N_{(2)} N_{(3)}, 
    \end{align}
    Theorem \ref{thm:Random_Tensor} implies that with probability $1 - \varepsilon$ we get
    \begin{align*}
        \eqref{eq:Two_Pairings_ToBound}
        & \lesssim 
        \Big(\log \frac{1}{\varepsilon} \Big) \frac{N_{(1)}^{2\varepsilon_1}}{N_{(1)}^{2 + 2\alpha} (N_{(2)} N_{(3)})^{2\alpha + 2\sigma - 1}}.
    \end{align*}
    
Summing up, in all cases, 
with probability $\geq 1-\varepsilon$ we get
    \begin{equation}\label{eq:Bound_Cubic_Wick_Ordered}
    \eqref{eq:Two_Pairings_ToBound}
        \lesssim \Big(\log \frac{1}{\varepsilon} \Big)^3 \frac{N_{(1)}^{2\varepsilon_1}}{ N_{(1)}^{1+2\alpha} N_{(2)}  (N_{(2)} N_{(3)})^{2\alpha + 2\sigma - 1}}.
    \end{equation}

    \subsection{$3$-pairings: $\boldsymbol{\mathcal{Q}_3}$, $\boldsymbol{\mathcal{Q}_4}$, $\boldsymbol{\mathcal{Q}_5}$ and $\boldsymbol{\mathcal{Q}_6}$}\label{sec:smoothing-3pairings}
    Dropping constants, we consider the operators 
    \begin{align}\label{eq:Qj-2pairing}
    \begin{array}{ll}
       \mathcal{Q}_3(u)_n = 
     |u_n|^2 \, \mathcal W_3(u)_n,  & 
    \qquad \qquad \mathcal{Q}_4(u)_n = \mathcal W_3(\mathcal P_3(u),u,u)_n, \\
       \mathcal{Q}_5(u)_n = \mathcal W_3(u,\mathcal P_3(u),u)_n,  & \qquad \qquad \mathcal{Q}_6(u)_n = u_n^2 \, \overline{\mathcal W_3(u)_n}, 
    \end{array}
    \end{align}
    from \eqref{eq:qnl-fcoeff}. 
    We treat them in pairs. 
    
    \subsubsection{$\mathcal Q_3$ and $\mathcal Q_6$}
    Both terms reduce to the same estimate for 
    \begin{align}
    \label{eq:Three_Pairings_ToBound36}
        \sum_{|n| \simeq N}  \big| \mathcal{Q}_3(u_1, \ldots, u_5)_n \big|^2
        = 
        \sum_{|n| \simeq N}  |(u_4)_n|^2 \, |(u_5)_n|^2  \bigg| \sum_{n_1,n_2,n_3} h^b_{\mathcal{W}_3} u_{n_1} \overline{u_{n_2}} u_{n_3} \bigg|^2,
    \end{align}
    where $h^b_{\mathcal W_3}$ was defined in \eqref{eq:Base_Tensor_Wick}. 
    Recalling that $u_i \in \{ R, D \}$, we first separate cases for $u_4 u_5$. 
    Using the pointwise bound for $R_n$ in Lemma \ref{lem:pbound-gn},
    and $|D_n| \leq \lVert D_n \rVert_{\ell^2} \lesssim N^{-\alpha - \sigma}$ from \eqref{eq:Bound_Dn}, we get 
    \begin{align}
        |(u_4)_n \, (u_5)_n|^2 
        \lesssim  
        \begin{cases}
            \displaystyle \big(\log 1/\varepsilon \big)^2 \frac{N^{\varepsilon_1}}{N^{4 + 4\alpha}}, & \text{ if } u_4 u_5 = RR, \\
            \displaystyle \big(\log 1/\varepsilon \big) \frac{N^{\varepsilon_1}}{N^{2 + 4\alpha  + 2\sigma}}, & \text{ if } u_4 u_5 = RD, \\
            \displaystyle \frac{1}{N^{4\alpha + 4\sigma}}, & \text{ if } u_4 u_5 = DD.
        \end{cases}
    \end{align}
    Since we are aiming for $\sigma < 1/2$, the worst case is $u_4u_5 = DD$, so we get 
    \begin{align}
        \sum_{|n| \simeq N}  \big| \mathcal{Q}_3(u_1, \ldots, u_5)_n \big|^2
        \lesssim  
        \Big(\log  \frac{1}{\varepsilon} \Big)^{|\mathcal R_{45}|}
        \,  \frac{N^{\varepsilon_1}}{N^{4\alpha + 4\sigma}} 
        \sum_{|n| \simeq N}   \bigg| \sum_{n_1,n_2,n_3} h^b_{\mathcal{W}_3} u_{n_1} \overline{u_{n_2}} u_{n_3} \bigg|^2,
    \end{align}
    where $|\mathcal R_{45}|$ is the number of $R$ in $\{u_4, u_5\}$.
    The problem is thus reduced to bounding the same Wick-reordered cubic term as in \eqref{eq:Bound_Cubic_Wick_Ordered} in Subsection~\ref{sec:Two_Pairings}. 
    Hence, with probability $\geq 1 - \varepsilon$ we get
    \begin{equation}
        \eqref{eq:Three_Pairings_ToBound36}
        \lesssim 
        \Big(\log  \frac{1}{\varepsilon} \Big)^{|\mathcal R|} \, 
        \frac{N_{(1)}^{3\varepsilon_1}}{ N_{(1)}^{1+2\alpha} N_{(2)}  (N_{(2)} N_{(3)})^{2\alpha + 2\sigma - 1} N^{4\alpha + 4\sigma }}.
    \end{equation}

    \subsubsection{$\mathcal Q_4$ and $\mathcal{Q}_5$}\label{subsec:eq:3pairings-n1n4n5}
    Both can be treated in an almost identical way, so we give explicit computations only for $\mathcal{Q}_4$. We need to bound
   	\begin{align}\label{eq:3pairings-n1n2n3}
        \sum_{|n| \simeq N}  \big| \mathcal{Q}_4(u_1, \ldots, u_5)_n \big|^2
        = \sum_n \bigg| \sum_{\substack{ n_1, n_2, n_3 }} h^b_{\mathcal{W}_3} \, (u_1)_{n_1}  \overline{(u_4)_{n_1}} (u_5)_{n_1}  \overline{(u_2)_{n_2}} (u_3)_{n_3} \bigg|^2.
   	\end{align}
    By the breaking of symmetry between $n_1$ and $n_3$, we need to consider all the possibilities for $N_{(1)} = \max (N_1, N_2, N_3)$. 
    We separate cases as follows:

    \begin{itemize}
        \item Case $RRRRR$. Taking the expectation directly, 
        \begin{align}
            \mathbb E [\eqref{eq:3pairings-n1n2n3}]
            & \lesssim  \sum_n \sum_{\substack{n_1,n_2,n_3 \\ n_1',n_2',n_3'}} h^b_{\mathcal W_3}(n_1,n_2,n_3,n) h^b_{\mathcal W_3}(n_1',n_2',n_3',n) \frac{\mathbb E \big[ g_{n_1} \overline{g_{n_2}} g_{n_3} \overline{g_{n_1'}} g_{n_2'} \overline{g_{n_3'}} \big]}{(N_1^3 N_2 N_3)^{2 + 2\alpha}}\\
            & \lesssim \frac{1}{(N_1^3 N_2 N_3)^{2 + 2\alpha}} \sum_{n,n_1,n_2,n_3}h^b_{\mathcal W_3} \lesssim \frac{  N_{(2)}^{\varepsilon_1} }{(N_{(1)}N_{(2)})^{1+2\alpha} N_1^{4 + 4\alpha} N_{(3)}^{2\alpha}},
        \end{align}
        since the counting from Lemma \ref{lemma:combest-nopairings-cubic} provides a factor of $N_{(1)} \, N_{(2)}^{1+\varepsilon_1} \,  N_{(3)}^2$. Thus, by hypercontractivity in Proposition \ref{lemma:wiener-chaos}, with probability $1 - \varepsilon$ we conclude that
        \begin{align}\label{eq:wiener-app3}
            \eqref{eq:3pairings-n1n2n3} \lesssim \Big( \log \frac{1}{\varepsilon} \Big)^5 \frac{  N_{(1)}^{2\varepsilon_1} }{(N_{(1)}N_{(2)})^{1+2\alpha} N_1^{4 + 6\alpha} N_{(3)}^{2\alpha}}.
        \end{align}
        \item Case $N_{(1)} = N_1$. By Proposition \ref{prop:Qsmoothing-mest}, we may assume that the triple resonance factor is $u_1 \overline{u_4} u_5 = R \overline R R$. By Cauchy-Schwarz and Lemma \ref{lem:pbound-gn}, 
    \begin{align}\label{eq:3pairings-n1n2n3-int1}
        \eqref{eq:3pairings-n1n2n3} &\lesssim   \sum_{n_1} |R_{n_1}|^6  \cdot   
        \sum_{n, n_1}  \Bigg|  \sum_{n_2,n_3} h^b_{\mathcal{W}_3} \overline{(u_2)_{n_2}} (u_3)_{n_3} \Bigg|^2
        \lesssim 
        \Big(\log  \frac{1}{\varepsilon} \Big)^3 \frac{N_1^{\varepsilon_1}}{N_1^{4+6\alpha}} 
        \sum_{n, n_1} \Bigg|  \sum_{n_2,n_3} h^b_{\mathcal{W}_3} \overline{(u_2)_{n_2}} (u_3)_{n_3} \Bigg|^2.
    \end{align} 
    We now separate cases for $u_2$ and $u_3$: 
    \vspace{0.25cm}
        \begin{itemize}
            \item If $u_2 = u_3 = D$,  
            by Schur's test and the counting in \eqref{eq:schur-prob-3p}
            we directly bound
            \begin{align}
                \eqref{eq:3pairings-n1n2n3} 
                \lesssim
                \Big(\log  \frac{1}{\varepsilon} \Big)^3 \frac{N_1^{\varepsilon_1}}{N_1^{4+6\alpha} (N_2 N_3)^{ 2\alpha + 2\sigma}}  \lVert h_{\mathcal W_3}^b \rVert_{n_2n_3 \to nn_1}^2 
                & \lesssim \Big(\log  \frac{1}{\varepsilon} \Big)^3  \frac{N_1^{\varepsilon_1} N_1 \min_{\{2,3\}}}{N_1^{4 + 6\alpha} \, (N_2 N_3)^{2\alpha + 2\sigma}} \\
                &\lesssim \Big(\log  \frac{1}{\varepsilon} \Big)^3  \frac{N_{(1)}^{\varepsilon_1}}{N_{(1)}^{3 + 6\alpha} \, \max_{\{2,3\}}  (N_2N_3)^{2\alpha + 2\sigma - 1 }}.
            \end{align}

            \item If $u_2 = D$ and $u_3 = R$, 
            by the random tensor estimate in Theorem \ref{thm:Random_Tensor} and \eqref{eq:schur-prob-3p_2} we get 
            \begin{align}
                \eqref{eq:3pairings-n1n2n3}
                & \lesssim
                \Big(\log  \frac{1}{\varepsilon} \Big)^4 \frac{N_{(1)}^{2\varepsilon_1}}{N_1^{4+6\alpha} N_2^{ 2\alpha + 2\sigma} N_3^{ 2 + 2\alpha}} 
                \max \big( \lVert h^b_{\mathcal W_3} \rVert_{n_2n_3 \to nn_1}^2, \lVert h^b_{\mathcal W_3} \rVert_{n_2 \to nn_1n_3}^2 \big) \\
                & \lesssim \Big(\log  \frac{1}{\varepsilon} \Big)^4 \frac{N_{(1)}^{3\varepsilon_1} \, N_1 N_3^{1+\varepsilon_1}}{N_1^{4 + 6\alpha}N_2^{ 2\alpha + 2\sigma}N_3^{2 +  2\alpha}}   \lesssim \Big(\log  \frac{1}{\varepsilon} \Big)^4 \frac{N_{(1)}^{4\varepsilon_1}}{N_{(1)}^{3+6\alpha} N_{(2)} N_{(3)}  N_2^{ 2\alpha + 2\sigma - 1} \, N_3^{ 2\alpha}}.
            \end{align}

            \item If $u_2 = R$ and $u_3 = D$, the computations are symmetric to the case above, so we get 
            \begin{align}
                \eqref{eq:3pairings-n1n2n3}
                \lesssim \Big(\log  \frac{1}{\varepsilon} \Big)^4 \frac{N_{(1)}^{4\varepsilon_1}}{N_{(1)}^{3+6\alpha} N_{(2)} N_{(3)} N_2^{ 2\alpha}  N_3^{2\alpha +  2\sigma - 1}}.
            \end{align}
        \end{itemize}
        Since the case $u_2 = R = u_3$ was previously considered,  
        for all cases under $N_1 = N_{(1)}$ we get, with probability $\geq 1 - \varepsilon$,
            \begin{align*}
                \eqref{eq:3pairings-n1n2n3}
                & \lesssim \Big(\log  \frac{1}{\varepsilon} \Big)^{|\mathcal R|}\frac{N_{(1)}^{4\varepsilon_1}}{N_{(1)}^{3+6\alpha}N_{(2)} (N_2N_3)^{2\alpha + 2\sigma - 1}}.
            \end{align*}

        \item Case $N_{(1)} = \operatorname{max}_{\{2,3\}}$. We separate the following cases: 
        \vspace{0.2cm}
        \begin{itemize}
            \item If $u_2 u_3 = RR$, using the operator norm we have
         \begin{align*}
             \eqref{eq:3pairings-n1n2n3} & \lesssim  \bigg\| \sum_{n_2,n_3} h^b_{\mathcal{W}_3} \frac{\overline{g_{n_2}} g_{n_3}}{(\langle n_2 \rangle \langle n_3 \rangle)^{1 + \alpha}} \bigg\|_{n_1 \rightarrow n}^2 
             \Big\| (u_1)_{n_1} \, \overline{(u_4)_{n_1}} (u_5)_{n_1}  \Big\|_{n_1}^2. 
         \end{align*}
        Bounding $R_{n_1}$ pointwise by Lemma \ref{lem:pbound-gn}, since $\lVert D_{n_1} \rVert_{n_1}^2 \lesssim N_1^{-2\sigma - 2\alpha}$, we get 
        \begin{align}\label{eq:3pairings-n2n4n5-n2}
            \big\| (u_1)_{n_1} \, \overline{(u_4)_{n_1}} (u_5)_{n_1}  \big\|_{n_1}^2
            \lesssim \begin{cases}
                \big(\log(1/\varepsilon)\big)^3 \displaystyle \frac{N_{1}^{\varepsilon_1}}{N_1^{4 + 6 \alpha}}, & \text{ if } u_1 \overline{u_4} u_5 = R\overline R R, \\ 
                \big(\log(1/\varepsilon) \big)^2 \displaystyle \frac{N_1^{\varepsilon_1}}{N_1^{4 + 2\sigma + 6\alpha}}, & \text{ if } u_1 \overline{u_4} u_5 = R \overline R D, \\
                \log(1/\varepsilon) \displaystyle\frac{N_1^{\varepsilon_1}}{N_1^{2 + 4\sigma + 6\alpha}}, & \text{ if } u_1 \overline{u_4} u_5 = R \overline D D, \\
                \displaystyle \frac{1}{N_1^{6\sigma + 6\alpha}}, & \text{ if } u_1 \overline{u_4} u_5 = D \overline D D,. \\
            \end{cases}
        \end{align}
        with probability $\geq 1 - \varepsilon$. 
        Recalling that we are aiming for $\sigma < 1/2$, the $DDD$ case is the worst one,
        so by the random tensor estimate in Theorem \ref{thm:Random_Tensor} we get
        \begin{align*}
             \eqref{eq:3pairings-n1n2n3} & \lesssim 
             \Big(\log  \frac{1}{\varepsilon} \Big)^{|\mathcal R|}
            \frac{ N_{(1)}^{2\varepsilon_1}}{(N_2 N_3)^{2 + 2\alpha}N_1^{6 \alpha + 6 \sigma}}  \max_{X \dot\cup Y = \{2,3\}} \{\|h^b_{\mathcal{W}_3}\|_{n_1n_X \rightarrow nn_Y}^2\} \\
            &\lesssim  \Big(\log  \frac{1}{\varepsilon} \Big)^{|\mathcal R|} \frac{N_{(1)}^{3\varepsilon_1}}{(N_{(1)} \operatorname{min}_{\{2,3\}})^{1 + 2\alpha} N_1^{6 \alpha + 6 \sigma}  },
         \end{align*}
        with probability  $\geq 1 - \varepsilon$, since by \eqref{eq:schur-prob-3p} 
        the maximum is bounded by $N_2 N_3^{1+\varepsilon_1}$. 
        Observe that in $\operatorname{min}_{\{2,3\}}^{1 + 2\alpha} N_1^{6 \sigma + 6 \alpha}$ there is at least one factor $N_{(2)}$.

            \item If $u_2 u_3 = RD$, we can assume $N_{(1)} = N_2$ by Proposition \ref{prop:Qsmoothing-mest}. We use \eqref{eq:3pairings-n2n4n5-n2} and the random tensor estimate from Theorem \ref{thm:Random_Tensor} to get 
        \begin{align*}
             \eqref{eq:3pairings-n1n2n3}  
             &\lesssim \Big( \log \frac{1}{\varepsilon} \Big)^{|\mathcal{R}|} \frac{N_{(1)}^{\varepsilon_1}}{N_2^{2 + 2\alpha}\, N_3^{ 2\alpha  + 2\sigma }} \, 
             \frac{N_1^{\varepsilon_1}}{N_1^{6\alpha + 6\sigma}} \, \max\{\|h^b_{\mathcal{W}_3}\|_{n_1n_2n_3 \rightarrow n}^2 , \|h^b_{\mathcal{W}_3}\|_{n_1n_3 \rightarrow nn_2}^2 \} \\
             & \lesssim \Big(\log  \frac{1}{\varepsilon} \Big)^{|\mathcal R|} \frac{N_{(1)}^{3\varepsilon_1}}{ N_{(1)}^{2 + 2\alpha} N_{1}^{6\alpha + 6\sigma - 1 }  N_{3}^{2\alpha + 2\sigma - 1}}
        \end{align*}
        with probability $\geq 1 - \varepsilon$, since the maximum is bounded by $N_{(2)} N_{(3)} N_{(1)}^{\varepsilon_1} $.  

            \item If $u_2 u_3 = DR$, then we can assume $N_{(1)} = N_3$, and proceeding as above we get
        \begin{align*}
             \eqref{eq:3pairings-n1n2n3}  &\lesssim \Big( \log \frac{1}{\varepsilon} \Big)^{|\mathcal{R}|} \frac{N_{(1)}^{\varepsilon_1} }{ N_1^{6\alpha + 6 \sigma} N_2^{ 2\alpha  + 2\sigma } N_3^{2 + 2\alpha}} \, 
             \max\{\|h^b_{\mathcal{W}_3}\|_{n_1n_2n_3 \rightarrow n}^2 , \|h^b_{\mathcal{W}_3}\|_{n_1n_2 \rightarrow nn_3}^2 \} \\
             &\lesssim \Big(\log  \frac{1}{\varepsilon} \Big)^{|\mathcal R|} \frac{N_{(1)}^{2\varepsilon_1}}{ N_{(1)}^{1 + 2\alpha} N_{(2)} N_{1}^{6\alpha + 6\sigma - 1}  N_{2}^{2\alpha + 2\sigma - 1}}.
         \end{align*}
         with probability $\geq 1 - \varepsilon$, since the maximum is bounded by $N_3 N_{(3)}^{1+\varepsilon_1}$. 
        \end{itemize}      

    Summing up, with probability $\geq 1 - \varepsilon$ we always get the bound 
    \begin{align}
         \eqref{eq:3pairings-n1n2n3}
         \lesssim \Big(\log  \frac{1}{\varepsilon} \Big)^{|\mathcal R|}  \frac{N_{(1)}^{3\varepsilon_1}}{N_{(1)}^{1 + 2\alpha} \, N_{(2)}\,  N_{(3)}^{2\alpha + 2\sigma - 1}}.
    \end{align}

    \end{itemize}
    
    \subsection{$4$-pairings: $\boldsymbol{\mathcal{Q}_7}$, $\boldsymbol{\mathcal{Q}_8}$, $\boldsymbol{\mathcal{Q}_9}$ and $\boldsymbol{\mathcal{Q}_{10}}$}\label{sec:smoothing-4pairings} 
    Again dropping constants from \eqref{eq:qnl-fcoeff}, we have 
    \begin{align}
        \begin{array}{ll}
        & \\
        \mathcal{Q}_7(u)_n = u_n^3 \, \sum_{ \substack{ n_2 + n_4 = 2n \\ n_2 \neq n} } \overline{ u_{n_2} } \, \overline{ u_{n_4}},  & 
        \qquad \qquad \mathcal{Q}_8(u)_n = 
        \sum_{ \substack{ 2n_1 - n_2 = n \\ n_1 \neq n_2} } u_{n_1}^2 |u_{n_1}|^2 \, \overline{u_{n_2}}, \\
        & \\
       \mathcal{Q}_9(u)_n = 
        |u_n|^2 \overline{u_n} \sum_{ \substack{ n_1 + n_3 = 2n \\ n_1 \neq n } } u_{n_1} u_{n_3},  & \qquad \qquad \mathcal{Q}_{10}(u)_n = |u_n|^2 u_n. 
    \end{array}
    \end{align}

    \subsubsection{$\mathcal Q_7$ and $\mathcal Q_9$} 
    Both operators are analogous and studied in a similar way. 
    We only work with 
    \begin{align}\label{eq:4pairings-n1n3n5n}
        \sum_{|n| \simeq N}  \big| \mathcal{Q}_7(u_1, \ldots, u_5)_n \big|^2
        =
        \sum_n \big|(u_1)_n(u_3)_n(u_5)_n \big|^2 \Bigg| \sum_{ n_2, n_4} h^b_{41} \,  \overline{(u_2)_{n_2}} \, \overline{(u_4)_{n_4}} \Bigg|^2,
    \end{align}
    where
    \begin{align*}
        h^b_{41} := h^b_{41}(n_2,n_4,n) 
        = \mathbbm 1_{|n_2| \simeq N_2}
        \mathbbm 1_{|n_4| \simeq N_4}
        \mathbbm 1_{n_2 + n_4 = 2n}
        \mathbbm 1_{n_2^2 + n_4^2 = 2n^2 + \mu}
        \mathbbm 1_{n_2 \neq n}.
    \end{align*}
    By symmetry, we can assume $N_4 \leq N_2$.  Hence, we always have $N \lesssim N_2$, so $N_{(1)} = N_2$ and by Proposition~\ref{prop:Qsmoothing-mest} we reduce to $u_2 = R$. By \eqref{eq:3pairings-n2n4n5-n2}, with probability $\geq 1 - \varepsilon$ we get
    \begin{align}\label{eq:eq:4pairings-n1n3n5n_2}
        \eqref{eq:4pairings-n1n3n5n}
        \lesssim 
        \Big( \log \frac{1}{\varepsilon} \Big)^3 \frac{N^{\varepsilon_1}}{N^{6\alpha + 6\sigma}}
        \sum_n \bigg| \sum_{ n_2, n_4} h^b_{41} \,  \overline{R_{n_2}} \, \overline{(u_4)_{n_4}} \bigg|^2.
    \end{align}
    We distinguish two cases.
    \begin{itemize}
        \item Case $u_4 = D$. By the random tensor estimate in Theorem \ref{thm:Random_Tensor}
        we get
        \begin{align}
            \sum_n \bigg| \sum_{ n_4} \Big( \sum_{n_2} h^b_{41} \,  \overline{R_{n_2}} \Big) \overline{D_{n_4}} \bigg|^2
            & \lesssim
            \Big( \log \frac{1}{\varepsilon} \Big)
            \frac{N_2^{\varepsilon_1}}{N_2^{2 + 2\alpha}N_4^{ 2\alpha + 2\sigma}} \, \max(\big\lVert h^b_{41}  \big\rVert_{n_2n_4 \to n}^2,  \big\lVert h^b_{41}  \big\rVert_{n_4 \to nn_2}^2) \\
            & \lesssim 
            \Big( \log \frac{1}{\varepsilon} \Big) \frac{N_2^{2\varepsilon_1} N_4^{\varepsilon_1}}{N_2^{2 + 2\alpha}N_4^{2\alpha + 2\sigma }}, 
        \end{align}
        where the last inequality follows from the sphere counting in Lemma \ref{lemma:combest-nopairings-cubic}.

        \item Case $u_4 = R$. The expectation is 
        \begin{align}
            \mathbb E \sum_n \Bigg| \sum_{ n_2, n_4} h^b_{41} \,  \overline{R_{n_2}} \, \overline{R_{n_4}} \Bigg|^2 
            \simeq \frac{1}{(N_2 N_4)^{2 + 2\alpha}} \sum_n \sum_{n_2, n_4} h^b_{41} \lesssim \frac{N_2^{\varepsilon_1}}{N_2^{2+2\alpha} N_4^{2\alpha}},
        \end{align}
        where the last inequality holds because
        \begin{align*}
            |\{ (n_2,n_4) \, \in \mathbb{Z}^2 :  \, |n_2| \simeq N_2, \, |n_4| \simeq N_4,  \, |n_2 - n_4|^2 = 2\mu \, \}| \lesssim N_2^{\varepsilon_1} N_4^2, 
        \end{align*}
        which follows from Lemma~\ref{lemma:combest-nopairings-cubic}. 
        Hence, by Proposition \ref{lemma:wiener-chaos}, with probability $\geq 1-\varepsilon$ we get
        \begin{equation}\label{eq:wiener-app4}
            \sum_n \Bigg| \sum_{ n_2, n_4} h^b_{41} \,  \overline{R_{n_2}} \, \overline{R_{n_4}} \Bigg|^2
            \lesssim 
            \Big( \log\frac{1}{\varepsilon} \Big)^2 \frac{N_2^{2\varepsilon_1}}{N_2^{2+2\alpha} N_4^{2\alpha}}.
        \end{equation}
    \end{itemize}
    The largest contribution being from the $u_4 = R$ case, joining it with \eqref{eq:eq:4pairings-n1n3n5n_2} we get
    \begin{align}
        \eqref{eq:4pairings-n1n3n5n}
        \lesssim \Big( \log \frac{1}{\varepsilon} \Big)^{|\mathcal R|} \frac{N_{(1)}^{3\varepsilon_1}}{N^{6\alpha + 6\sigma} N_{(1)}^{2 + 2\alpha} N_{4}^{2\alpha}}.
    \end{align}

    \subsubsection{$\mathcal Q_8$}
    In this case, we estimate
    \begin{align}\label{eq:4pairings-n1n3n4n5}
        \sum_{|n| \simeq N}  \big| \mathcal{Q}_8(u_1, \ldots, u_5)_n \big|^2
        = \sum_n \bigg| \sum_{ n_1, n_2} h^b_{42} \, (u_1)_{n_1} (u_3)_{n_1} \overline{(u_4)_{n_1}} (u_5)_{n_1} \overline{(u_2)_{n_2}} \bigg|^2.
    \end{align}
    where the base tensor is
    \begin{align*}
        h^b_{42} 
        = \mathbbm 1_{|n_1| \simeq N_1} 
        \mathbbm 1_{|n_2| \simeq N_2}
        \mathbbm 1_{2n_1 - n_2 = n}
        \mathbbm 1_{2n_1^2 - n_2^2 = n^2 - \mu}
        \mathbbm 1_{n_1 \neq n}.
    \end{align*}
    We again separate cases.
    \begin{itemize}
        \item Case $N_1 = N_{(1)}$. From Proposition \ref{prop:Qsmoothing-mest}, we may assume $u_1 u_3 \overline{u_4}u_5 = RR\overline R R$, and we separate cases in $u_2$. 

        $-$ If $u_2 = R$, then  
        \begin{align}
            \mathbb E \sum_n \Bigg| \sum_{ n_1, n_2} h^b_{42} \, R_{n_1}^2 |R_{n_1}|^2 \overline{R_{n_2}} \Bigg|^2 
            \simeq \frac{1}{N_1^{8+8\alpha} N_2^{2+2\alpha}} \sum_n \sum_{n_1, n_2} h^b_{42} 
            \lesssim 
            \frac{N_1^{\varepsilon_1}}{N_1^{8+8\alpha} N_2^{2\alpha}} 
        \end{align}
        since the last sum corresponds to the sphere-counting estimate
        \begin{align}\label{eq:sphere-count-hb42}
           |\{ (n_1, n_2) \in \mathbb{Z}^2 \, : \, |n_1 - n_2|^2 = \mu/2 \, , \, \, |n_1| \simeq N_1, \, |n_2| \simeq N_2 \}| \lesssim N_1^{\varepsilon_1} N_2^2.
        \end{align} 
        from Lemma~\ref{lemma:combest-nopairings-cubic}. Hence, using Proposition \ref{lemma:wiener-chaos}, we get with probability $\geq 1-\varepsilon$ that 
        \begin{equation}\label{eq:Tensor42_R}
            \sum_n \Bigg| \sum_{ n_1, n_2} h^b_{42} \, R_{n_1}^2 |R_{n_1}|^2 \overline{R_{n_2}} \Bigg|^2
            \lesssim \Big( \log \frac{1}{\varepsilon} \Big)^5 \frac{N_{(1)}^{2\varepsilon_1}}{N_{(1)}^{8+8\alpha} N_2^{2\alpha}}.
        \end{equation}

        $-$ If $u_2 = D$, then Lemma \ref{lem:pbound-gn} and $\lVert h^b_{42} \rVert_{n_1n_2 \to n}^2 \lesssim N_2$ directly give
        \begin{align}\label{eq:Tensor42_D}
                \sum_n \Bigg| \sum_{ n_1, n_2} h^b_{42} \, R_{n_1}^2 |R_{n_1}|^2 \overline{D_{n_2}} \Bigg|^2
                & \leq \lVert h^b_{42} \rVert_{n_1n_2 \to n}^2 \lVert R_{n_1}^4 \rVert_{n_1}^2 \, 
                \lVert D_{n_2} \rVert_{n_2}^2
                \lesssim \Big( \log \frac{1}{\varepsilon} \Big)^4  \frac{N_{(1)}^{\varepsilon_1}}{N_{(1)}^{6+8\alpha} N_2^{2\alpha + 2\sigma - 1}}
        \end{align}
        with probability $\geq 1-\varepsilon$.

        \item Case $N_2 = N_{(1)} \gg N_1$. In this case, $u_1 u_3 \overline{u_4} u_5$ is any combination of $R$ and $D$. With the analogue to \eqref{eq:3pairings-n2n4n5-n2} with four factors, using  Theorem \ref{thm:Random_Tensor} and Lemma \ref{lemma:combest-nopairings-cubic} we get
        \begin{align*}
            \eqref{eq:4pairings-n1n3n4n5} & \lesssim \Big\|\sum_{n_2} h^b_{42} \overline{R_{n_2}} \Big\|_{n_1 \rightarrow n}^2 \big\|(u_1)_{n_1} (u_3)_{n_3} \overline{(u_4)_{n_4}} (u_5)_{n_5} \big\|_{n_1}^2 \\
            &\lesssim \Big( \log \frac{1}{\varepsilon} \Big)^{|\mathcal{R}|} \frac{N_{(1)}^{2\varepsilon_1}}{N_2^{2 + 2\alpha} N_1^{8\alpha + 8 \sigma}} \max\{ \| h^b_{42} \|_{n_2n_1 \rightarrow n}^2 , \| h^b_{42} \|_{n_1 \rightarrow nn_2}^2 \} \\
            &\lesssim \Big( \log \frac{1}{\varepsilon} \Big)^{|\mathcal{R}|} \frac{N_{(1)}^{3\varepsilon_1}}{N_{(1)}^{2 + 2\alpha} N_1^{8 \alpha + 8\sigma - 1}}.
        \end{align*}
    \end{itemize}
    Summing up, the largest bound coming from $N_2 \gg N_1$, with probability $\geq 1 - \varepsilon$ we always get 
    \begin{equation}
        \eqref{eq:4pairings-n1n3n4n5} 
        \lesssim \Big( \log \frac{1}{\varepsilon} \Big)^{|\mathcal R|}  \frac{N_{(1)}^{3\varepsilon_1}}{N_{(1)}^{2 + 2\alpha} N_{(2)}^{2\alpha + 8\sigma - 1 }}.
    \end{equation} 

    \subsubsection{$\mathcal Q_{10}$}
    It corresponds to the cubic completely resonant term, in which case we need to bound
    \begin{align}
        \sum_{|n| \simeq N}  \big| \mathcal{Q}_{10}(u_1, u_2, u_3)_n \big|^2
        = \sum_{|n| \simeq N} \big|(u_1)_n \overline{(u_2)_n} (u_3)_n \big|^2.
    \end{align} 
    By Proposition~\ref{prop:Qsmoothing-mest}, 
    it suffices to work with the case $u_1\overline{u_2} u_3 = R \overline R R$, and in that case, the pointwise bound from Lemma \ref{lem:pbound-gn} directly gives, with probability $\geq 1 - \varepsilon$,
    \begin{align}
        \sum_{|n| \simeq N} |R_n|^6
        \lesssim \Big( \log \frac{1}{\varepsilon} \Big)^3 \frac{N^{\varepsilon_1}}{N^{4+6\alpha}}.
    \end{align} 

    \subsection{Purely resonant case}
    It corresponds to $\mathcal Q_{11}$ in \eqref{eq:qnl-fcoeff}, for which by Proposition~\ref{prop:Qsmoothing-mest} we may assume that $u_1 \overline{u_2} u_3 \overline{ u_4} u_5 = R \overline R R \overline R R$. Hence, with probability $\geq 1 - \varepsilon$,
    \begin{align}
        \sum_{|n| \simeq N}  \big| \mathcal{Q}_{11}(u_1, \ldots, u_5)_n \big|^2
        = \sum_{|n| \simeq N} |R_n|^{10}        
        \lesssim  \Big( \log \frac{1}{\varepsilon} \Big)^5 \frac{N^{\varepsilon_1}}{N^{8 + 10\alpha}}.
    \end{align} 
    
    With all cases covered, the proof of Proposition~\ref{prop:Final_Estimate} is complete.

    \appendix

    \section{Deterministic toolbox}
    \label{app:Deterministic}
    
    Let us first recall the Strichartz estimates for the Schrödinger linear flow on $\mathbb{T}^{d+1}$. 
    Given a function $f$ and a cube $Q$ of side-length $N$, then for any $\varepsilon_1 > 0$ we have
    \begin{align}\label{eq:strichartz}
    \begin{array}{rlll}
        \| P_Q e^{it\Delta}f\|_{L^p_{t,x}(\mathbb{T}^{d+d})} 
        & \lesssim 
        & N^{\varepsilon_1}\|f\|_{L^2_x(\mathbb{T}^d)}, 
        & \qquad p = \frac{2(d+2)}{d}, \\
        \| P_Q e^{it\Delta}f\|_{L^p_{t,x}(\mathbb{T}^{d+1})} 
        & \lesssim 
        & N^{\frac{d}{2} - \frac{d+2}{p}} \|f\|_{L^2_x(\mathbb{T}^d)}, 
        & \qquad p > \frac{2(d+2)}{d}.
    \end{array}
    \end{align}
We also use the following counting estimates throughout this article. 

\begin{lem}[p. 430-432 in \cite{Bourgain1996}]
\label{lemma:combest-nopairings-cubic}
The following are true for every $\varepsilon_1 > 0$. 
\begin{enumerate}
    \item Let $N \in 2^{\mathbb N}$. Then,
    \begin{align}
        \# \big\{ \, n \in \mathbb Z^d \, : \, |n| \lesssim N, \,\, \,  |n - n_0| = \mu \, \big\} \lesssim N^{d-2+\varepsilon_1}, 
        \qquad \forall n_0 \in \mathbb Z^d, \quad \forall \mu \in \mathbb Z. 
    \end{align}

    \item Let $\{N_j\} \subset 2^{\mathbb{N}}$ and $n \in \mathbb{Z}^2$, and let 
    \begin{align*}
        R_n = \bigg\{ (n_1,n_2,n_3) \in (\mathbb{Z}^2)^3 \, : \,  |n_j| \simeq N_j, \, n_2 \neq n_1 , n_3, \, 
        \begin{array}{ll}
            n_1 - n_2 + n_3 = n \\
            n_1^2 - n_2^2 + n_3^2 = n^2 + \mu.
        \end{array}
        \bigg\}.
    \end{align*}
    Then, 
    \begin{align*}
        |R_n| \lesssim  N_2^2 \, \min\{N_1,N_3\}^{\varepsilon_1}
        \qquad \text{ and } \qquad  
        |R_n| \lesssim  N_{(2)}N_{(3)}^{1 + \varepsilon_1}.
    \end{align*}

    \item For $n_1, n_2 \in \mathbb Z^2$, let also
    \begin{align*}
        \begin{array}{ll}
            S(n_1) = \{ (n_2,n_3) \in (\mathbb{Z}^2)^2 \,\, : \, \, |n_j| \simeq N_j, \, \, n_2 \neq n_1, n_3, \, \,  (n_1 - n_2) \cdot (n_2 - n_3) = \mu \, \}, \\
            S(n_2) = \{ (n_1,n_3) \in (\mathbb{Z}^2)^2 \, \, : \, \, |n_j| \simeq N_j, \, \, n_2 \neq n_1, n_3, \, \, \, (n_1 - n_2) \cdot (n_2 - n_3) = \mu \, \}. 
        \end{array}
    \end{align*}
    Then,
    \begin{align}
        |S(n_1)| \lesssim N_2 N_3 \, \min\{ N_2 , N_3 \}^{\varepsilon_1}, \qquad \text{ and } \qquad 
        |S(n_1)| \lesssim N_3^2 \, N_2^{\varepsilon_1},
    \end{align}
    and 
    \begin{align}
        |S(n_2)| \lesssim N_1N_3 \, \min\{N_1 , N_3\}^{\varepsilon_1}. 
    \end{align}
    
\end{enumerate}
\end{lem}

We will combine this counting lemma with the Schur test.  

\begin{lem}[Schur test]\label{lemma:schur-test}
Let $X, Y$ be measurable spaces,
and let $T$ be an integral operator 
\begin{equation}
    Tf(x) = \int K(x,y) f(y) \, dy
\end{equation}
with non-negative kernel $K$.
Let $p(x)$ and $q(y)$ be positive functions, and $\alpha, \beta > 0$ such that 
\begin{equation}\label{eq:hyp-schur}
    \int K(x,y) p(x) \, dx \leq \alpha q(y), 
    \qquad \qquad 
    \int K(x,y) q(y) \, dy \leq \beta p(x), 
\end{equation}
for almost all $x\in X$ and $y \in Y$. Then, $T: L^2(Y) \to L^2(X)$ is a bounded operator with
\begin{equation}
    \lVert T \rVert_{L^2(Y) \to L^2(X)} \leq \sqrt{\alpha \beta}.
\end{equation}
\end{lem}

We next recall some well-known facts about Bourgain's Fourier restriction spaces. 
	
\begin{lem}[\cite{ErdoganTzirakis2016}, Lemmata 3.10, 3.11 and 3.12]
\label{lemma:pre-rest-spaces}
    Let $s \in \mathbb{R}$ and $1/2 < b \leq 1$.  Then
    \begin{align*}
        \begin{array}{rcl}
                \|\eta(t) e^{it\Delta} f(x) \|_{X^{s,b}} & \lesssim & \|f\|_{H^s(\mathbb{T}^2)}, \\
                \displaystyle \Big\| \eta(t) \int_0^t e^{i(t - t')\Delta} F(\cdot,t') dt' \Big\|_{X^{s,b}} & \lesssim & \| F \|_{X^{s,b-1}}.
                \end{array}
        \end{align*}
        Moreover, for  $-1/2  < b' < b < 1/2$ and $\delta >0$, 
        \begin{align}
            \| F \|_{X^{s,b'}_{\delta}} \lesssim  \delta^{b-b'} \|F\|_{X^{s,b}_{\delta}}.
        \end{align}
    \end{lem}

    \noindent
    The $X^{s,b}$ spaces also satisfy the following transference principle adapted from \cite[Lemma 2.9]{Tao2006-book}. 
	\begin{lem}\label{lemma:TransferLemma}
		Let $Y$ be a Banach space of functions 
		$F : (t,x) \in \mathbb{R} \times \mathbb{T}^2 \mapsto \mathbb{C}$. 
		Assume
		\begin{align*}
			\| e^{i \alpha t} e^{i t \Delta} f \|_Y \leq C \| f \|_{H^s(\mathbb{T}^2)},
            \qquad \qquad \forall \alpha \in \mathbb R,
		\end{align*}
		for some constant $C > 0$ independent of $\alpha \in \mathbb{R}$. Let $b > 1/2$. Then,
		\begin{align*}
			\|F\|_Y \leq C \|F\|_{X^{s,b}}, 
            \qquad \qquad  \forall F \in Y. 
		\end{align*}
	\end{lem}
    Combining it with  with Strichartz estimates \eqref{eq:strichartz} and maximal estimates \eqref{eq:max-est}, we obtain the following embeddings for every $\varepsilon_1 > 0$ and $b > 1/2$:
    \begin{align}\label{eq:Transferred_Estimates}
        \lVert F \rVert_{L^4_{t,x}(\mathbb T \times \mathbb T^2)} \lesssim \lVert F \rVert_{X^{\varepsilon_1, b}} 
        \qquad \text{ and } \qquad 
        \big\lVert \sup_{0 \leq t \leq 1} |F| \big\rVert_{L^2_{x}(\mathbb T^2)} \lesssim \lVert F \rVert_{X^{\frac12 + \varepsilon_1, b}}.
    \end{align}
    In particular, for every cube $Q \subset \mathbb T^2$ of side-length $N$ we have 
    \begin{align}\label{eq:Transferred_Estimates_Cube}
        \lVert P_Q F \rVert_{L^4_{t,x}(\mathbb T \times \mathbb T^2)} \lesssim N^{\varepsilon_1} \, \lVert F \rVert_{X^{0, b}}.
    \end{align}
    The time-regularity in this estimate may be slightly improved, which plays an important role in the proof of Lemma \ref{lemma:Multilinear_Estimates}.  

    \begin{lem}\label{lemma:st-int}
    Let $N \in \mathbb N$ and $Q$ a cube of side-length $N$. For every $\varepsilon_1 > 0$ small enough, we have
    \begin{align*}
        \|P_Qu\|_{L^4_{t,x}( \mathbb T \times \mathbb{T}^2)} \lesssim_{\varepsilon_1} N^{2\varepsilon_1} \|u\|_{X^{0,\frac{1}{2} - \varepsilon_1}}.
    \end{align*}
\end{lem}
\begin{proof}
    It follows by interpolating \eqref{eq:Transferred_Estimates_Cube} with the estimate
    \begin{align}
        \lVert P_Q u \rVert_{L^4_{t,x}(\mathbb  T \times \mathbb T^2)} \lesssim N^{1/2} \, \lVert P_Q u \rVert_{X^{0,\frac14 + \varepsilon_1}(\mathbb  T \times \mathbb T^2)},
    \end{align}
    that follows directly by Hausdorff-Young and H\"older inequalities. 
\end{proof}

\section{Probabilistic toolbox}
    \label{app:probprel}  
    Let $g^{\omega}$ be a standard Gaussian random variable.
    By definition, for $\varepsilon > 0$ small enough, there exists a set $A_\varepsilon$ in the probability space with $\mathbb{P}(A_\varepsilon) \geq 1 - \varepsilon$ such that
        \begin{align}\label{eq:pbounds-gaussian-one}
            |g^{\omega}| \lesssim \Big( \log \frac{1}{\varepsilon} \Big)^{1/2}, \qquad \forall \omega \in A_\varepsilon.
        \end{align}
    Now, let $(g_n^{\omega})_{n \in \mathbb Z^d}$ be a family of i.i.d. standard Gaussian random variables. 
    If $|n| > 1$, then by \eqref{eq:pbounds-gaussian-one} for every $\varepsilon > 0$ there is a set 
    $A_{n,\varepsilon}$
    such that $\mathbb{P}(A_{n,\varepsilon}) \geq 1 - \varepsilon/|n|^{d+1}$ and
    \begin{align}\label{eq:Gaussian_Uniform}
        |g_n^{\omega}| 
        \lesssim \Big( \log \frac{|n|^{d + 1}}{\varepsilon} \Big)^{1/2}  \lesssim (\log |n|)^{1/2}  \Big( \log \frac{1}{\varepsilon} \Big)^{1/2}, 
        \qquad \forall \omega \in A_{n,\varepsilon},
    \end{align}
    where the implicit constant does not depend on $n$ nor $\varepsilon$.
    If for $|n| \leq 1$ we define $A_{n,\varepsilon} = A_\varepsilon$ and 
    \begin{align*}
        \Sigma_{\varepsilon} = \bigcap_{n \in \mathbb{Z}^d} A_{n,\varepsilon},
        \qquad \text{ which satisfies } \qquad 
        \mathbb P(\Sigma_{\varepsilon}^c) \lesssim  \sum_{n \in \mathbb{Z}^d} \frac{\varepsilon}{\langle n \rangle^{d + 1}} \lesssim \varepsilon , 
    \end{align*}
    we get the bound \eqref{eq:Gaussian_Uniform} uniformly in $n$ for every $\omega \in \Sigma_\varepsilon$, 
    which we collect in the following lemma. 

    \begin{lem}\label{lem:pbound-gn}
    Let $(g_n^{\omega})_{n \in \mathbb Z^d}$ be i.i.d. standard Gaussian random variables. Then, for every $\varepsilon > 0$,  
        \begin{align}\label{eq:pbound-gn}
        |g_n^{\omega}|
        \lesssim  \Big( \log \frac{1}{\varepsilon} \Big)^{1/2} \, \big( 1 + \log \langle n \rangle \big)^{1/2}, \qquad \qquad \forall n \in \mathbb{Z}^d, 
    \end{align}
    with probability $\geq 1 - \varepsilon$.
    \end{lem}

    More complex random variables built upon gaussians but not gaussian themselves appear frequently throughout the article, 
    and they can be bounded in a similar way by hypercontractivity. 
    
    \begin{thm}[Theorem 5.10 in \cite{Janson1997}]\label{lemma:wiener-chaos}
    Let $(g_n^{\omega})_{n \in \mathbb{Z}^d}$ be a sequence of complex-valued, i.i.d. standard Gaussian random variables on $(\Omega,\mathbb{P})$, and let 
    \begin{align*}
        H = \Big\{\sum_{n \in \mathbb{Z}^d} c_n g_n^{\omega} \, : \, c_n \in \mathbb{C} \, \, , \, \sum_{n \in \mathbb{Z}^d} |c_n|^2 < \infty \Big\}
    \end{align*}
    be the corresponding Gaussian Hilbert space. 
    Then, for all $k \in \mathbb N$ and $1 \leq p,q < \infty$ we have
    \begin{align*}
        \|X\|_{L^q(\Omega)} \, \lesssim \, c(p,q)^k \, \|X\|_{L^p(\Omega)}, \qquad \forall X \in \overline{\mathcal{P}_k(H)},
    \end{align*}
    where $\mathcal{P}_k(H)$ are the polynomials on $H$ of degree $k$, $\overline{\mathcal{P}_k(H)}$ denotes its closure in $L^2(\Omega,\mathbb{P})$, and 
    \begin{align}
        c(p,q) \leq  (q-1)^{1/2}, \quad \text{ if } 2 \leq p \leq q < \infty,  \quad 
        \qquad \text{ and } \qquad \quad 
        c(p,2) \leq e^{\frac{2}{p} - 1}, \quad \text{ if } \quad p \leq 2. 
    \end{align}   
\end{thm}

    By Chebyshev's inequality, Theorem~\ref{lemma:wiener-chaos} implies large deviation estimates of the form
    \begin{align*}
        \mathbb{P}\left( \{ \omega \in \Omega : |X| > \lambda \} \right) \lesssim
        \exp \left[ - \bigg( \frac{ c\,  \lambda}{ \, \|X\|_{L^2(\Omega)} }\bigg)^{2/k} \right], 
        \qquad \text{ for every } \lambda \text{ large enough, } 
    \end{align*}
    for $X \in \overline{\mathcal P_k(H)}$, 
    as long as we can compute the variance $\lVert X \rVert_{L^2(\Omega)}$. 
    Moreover, given that
    \begin{align}
        c(1,q) \leq c(1,2)c(2,q) \leq e(q-1)^{1/2}, \qquad \forall q \geq 2, 
    \end{align}
    we also get the large deviation estimate with $\lVert X \rVert_{L^1(\Omega)}$. 
    We use these bounds in Section~\ref{sec:proof-smoothing}, in particular to bound purely random terms like $RRRRR$ term in Section~\ref{sec:RRRRR}, which has the form
    \begin{align*}
        X = \sum_{|n| \simeq N} \sum_{ \substack{ n_1, \ldots, n_5 \\ n_1', \ldots , n_5'}}  h^b_1(n_1, \ldots, n_5, n) \, h^b_1(n_1', \ldots, n_5',n) \,  g_{n_1} \overline{g_{n_2}}  g_{n_3} \overline{g_{n_4}} g_{n_5} \overline{g_{n_1'}} g_{n_2'} \overline{g_{n_3'}} g_{n_4'} \overline{g_{n_5'}} 
        \in \mathcal{P}_{10}(H)
    \end{align*}
    and similarly in \eqref{eq:wiener-app2}, \eqref{eq:wiener-app3}, \eqref{eq:wiener-app4} and \eqref{eq:Tensor42_R}. 

    A more elementary application 
    are the following probabilistic dispersive and Strichartz estimates for the linear evolution of the random data $f^\omega$ in \eqref{eq:random-data}. They also follow by more basic means, see for instance \cite[Section 3]{BurqTzvetkov2008-1}. 
    
    \begin{lem}\label{lemma:Large_Deviation_Schrodinger_Linear}
Let $\alpha > 0$ and $f^{\omega}$ be the random datum in \eqref{eq:random-data}. 
Let $q \geq 1$,  $s \in [0,\alpha)$ and $N \in \mathbb N$.
    \begin{enumerate}
    \item (Dispersive bounds) The norm $\lVert P_N \langle D \rangle^s e^{it\Delta} f^\omega \rVert_{L_x^q(\mathbb T^d)}$ is sub-Gaussian for every $t \in \mathbb  R$. Namely,
    \begin{equation}
        \mathbb P \big( \lVert P_N \langle D\rangle^s e^{it\Delta} f^\omega \rVert_{L_x^q(\mathbb T^d)} > \lambda \big) \leq C e^{- c\lambda^2 N^{2(\alpha - s)}}, 
        \qquad \quad  \forall \lambda \geq e \sqrt q, 
    \end{equation}
    for $C,c>0$ independent of $N$ and $t$. 
    As a consequence, for any $\varepsilon > 0$ small enough there exist $A_{\varepsilon,N} \subset \Omega$ with $\mathbb P(A_{\varepsilon,N}) > 1-\varepsilon$ and $C_1 > 0$ independent of $N$ and $t$ such that
    \begin{equation}
        \lVert P_N \langle D \rangle^s e^{it\Delta} f^\omega \rVert_{L_x^q(\mathbb T^d)} 
        \leq \Big( \log \frac{1}{\varepsilon}\Big)^{1/2} \frac{C_1}{N^{\alpha - s}} , 
        \qquad \forall \omega \in A_{\varepsilon, N}. 
    \end{equation}

    \item (Strichartz bounds) The norm $\lVert P_N \langle D \rangle^s e^{it\Delta} f^\omega \rVert_{L^q_{x,t}(\mathbb T^{d+1})}$ is sub-Gaussian. Namely, 
    \begin{equation}
        \mathbb P \big( \lVert P_N \langle D \rangle^s e^{it\Delta} f^\omega \rVert_{L^q_{t,x}(\mathbb T \times \mathbb T^d)} > \lambda \big) \leq C e^{-c \lambda^2 N^{2(\alpha - s)}}, \qquad \quad \forall \lambda \geq e \sqrt q, 
    \end{equation}
    for $C,c>0$ independent of $N$. 
    As a consequence, for $\varepsilon > 0$ small enough, there exist $ A_{\varepsilon,N} \subset \Omega$ with $\mathbb P( A_{\varepsilon,N}) > 1-\varepsilon$ and $C_1 > 0$ independent of $N$ such that
    \begin{equation}
        \lVert P_N \langle D \rangle^s e^{it\Delta} f^\omega \rVert_{L_{t,x}^q (\mathbb T \times \mathbb T^d)} \leq 
        \Big( \log \frac{1}{\varepsilon}\Big)^{1/2} \frac{C_1}{N^{\alpha - s}},  
        \qquad \forall \omega \in A_{\varepsilon, N}. 
    \end{equation}
    \end{enumerate}
\end{lem}
\begin{proof}
    We prove the dispersive bounds, as Strichartz bounds follow in a similar way. 
    By Proposition~\ref{lemma:wiener-chaos} with $k = 1$, for $Q \geq 1$ large enough we directly compute 
    \begin{align} 
        \lVert P_N \langle D \rangle^s e^{it\Delta} f^\omega \rVert_{L^Q_{\omega}(\Omega)}
        \lesssim \sqrt{Q} \, 
        \lVert P_N \langle D \rangle^s e^{it\Delta} f^\omega \rVert_{L^2_{\omega}(\Omega)}
        \simeq \sqrt{Q} \bigg( \sum_{|n| \simeq N} \frac{1}{\langle n \rangle^{ 2 + 2\alpha  - 2s}} \bigg)^{1/2}
        \lesssim \frac{\sqrt Q}{N^{\alpha - s}},
    \end{align}
    independently of $t$. 
    Therefore, by Minkowski's inequality, 
    for any $Q \geq q$ we have
    \begin{align*}
        \left\| \lVert P_N \langle D \rangle^s e^{it\Delta} f^\omega \rVert_{L^q_x(\mathbb T^d)} \right\|_{L^Q_{\omega}(\Omega)} 
        \leq 
        \left\| \lVert P_N \langle D \rangle^s e^{it\Delta} f^\omega \rVert_{L^Q_{\omega}(\Omega)} \right\|_{L_x^q(\mathbb T^d)}
        \lesssim \frac{\sqrt Q}{N^{\alpha - s}},
    \end{align*}
    and hence by Chebyshev's inequality we get
    \begin{align*}
        \mathbb{P} \big(\lVert P_N \langle D \rangle^s e^{it\Delta} f^\omega \rVert_{L_x^q(\mathbb T^d)} > \lambda \big) 
        \lesssim \left( \frac{ \sqrt Q}{\lambda N^{\alpha - s}} \right)^Q 
        \lesssim \exp\left( Q \log\left( \frac{\sqrt Q}{\lambda N^{\alpha - s}} \right) \right).
    \end{align*}
    If $\lambda \geq e \sqrt q$, we can take $Q$ such that $ e \sqrt Q = \lambda N^{\alpha - s}$, so that
    \begin{align*}
        \mathbb{P} \big(\lVert P_N \langle D \rangle^s e^{it\Delta} f^\omega \rVert_{L_x^q(\mathbb T^d)} > \lambda \big) 
        \leq \exp\left(-\frac{\lambda^2 N^{2(\alpha - s)}}{ e^2}\right).
    \end{align*}
    Finally, for $\varepsilon  > 0$ small enough, the last statement follows from choosing 
    \begin{align}
        \hspace{3cm}
        \lambda = \Big( \log \frac{1}{\varepsilon} \Big)^{1/2} \frac{e}{N^{\alpha - s}} 
        \qquad \Longleftrightarrow \qquad
        \exp\left(-\frac{\lambda^2 N^{2(\alpha - s)}}{e^2}\right) = \varepsilon.
        \hspace{3.5cm}
        \qedhere
    \end{align}
\end{proof}

The bounds in Lemma~\ref{lemma:Large_Deviation_Schrodinger_Linear} also hold uniformly in dyadic frequency scales, by the same argument used to prove Lemma~\ref{lem:pbound-gn}. 
\begin{lem}\label{lemma:Strichartz_Prob_Uniform}
Let $\alpha > 0$ and $f^{\omega}$ be the random datum in \eqref{eq:random-data}. 
Let $q \geq 1$ and  $s \in [0,\alpha)$.
Then, for every $\varepsilon > 0$, with probability $\geq 1 - \varepsilon$ we have 
\begin{equation}\label{eq:Lp_bound-Prob_APP}
        \lVert  P_{N} \langle D \rangle^s e^{it\Delta}f^\omega  \rVert_{L_{t,x}^p(\mathbb{T} \times \mathbb{T}^d)}
        \lesssim \Big( \log \frac{1}{\varepsilon} \Big)^{1/2} \frac{(\log N )^{1/2}}{N^{\alpha - s}}, 
        \qquad \quad \forall N \in 2^{\mathbb N}.
    \end{equation}
\end{lem}
\begin{proof}
For $\varepsilon > 0$, 
take the sets $A_{\varepsilon/N, N}$ from  Lemma~\ref{lemma:Large_Deviation_Schrodinger_Linear} and define the set 
\begin{align}
    A_{\varepsilon} = \bigcap_{N \in 2^{\mathbb N}} A_{\varepsilon/N, N}
    \qquad \text{ such that } \qquad 
    \mathbb P(A_{\varepsilon}^c) \leq \sum_{N \in 2^{\mathbb N}} \mathbb P(A_{\varepsilon/N, N}^c) \leq \sum_{N \in 2^{\mathbb N}}  \frac{\varepsilon}{N} \simeq \varepsilon. 
\end{align}
Thus, for every $\omega \in A_\varepsilon$, we have 
\begin{align}
    \hspace{1.5cm} \| P_N \langle D \rangle^s e^{it\Delta}f^\omega \|_{L^p_{t,x}(\mathbb{T} \times \mathbb{T}^d)} \lesssim \frac{( \log (N/\varepsilon) )^{1/2}}{N^{\alpha}} \lesssim  \Big( \log \frac{1}{\varepsilon} \Big)^{1/2} \, \frac{(\log N)^{1/2} }{N^{\alpha}}, 
    \quad \text{ for all } N \in 2^{\mathbb N}.
    \hspace{2.5cm}  \qedhere
\end{align}
\end{proof}

\section*{Acknoledgments}
The authors would like to thank Renato Lucà, Andrea R. Nahmod and Haitian Yue for clarifying discussions. 

The authors are supported by the Spanish Research Agency (AEI) through project PID2024-156169NB-I00 and the BCAM Severo Ochoa excellence
accreditation CEX2021-01142-S, and also by the Basque Government through the BERC 2022–2025 program.
D. E. is also supported by Spanish Research Agency (AEI) project RYC2023-042719-I.
P. M. is also supported by the predoctoral program of the Department of Education of the Basque Government.

\printbibliography

\end{document}